\documentclass[11pt]{article}

\usepackage[utf8]{inputenc}
\usepackage{amsmath}
\usepackage{amssymb}
\usepackage{amsthm}

\usepackage{graphicx}
\usepackage{fullpage}
\usepackage{xcolor,array,bm}
\usepackage{braket}
\usepackage{tikz-cd}
    \usetikzlibrary{knots}
    \usetikzlibrary{decorations.markings}
    \tikzstyle directed=[postaction={decorate,decoration={markings,
    mark=at position #1 with {\arrow{>}}}}]
    \tikzstyle rdirected=[postaction={decorate,decoration={markings,
    mark=at position #1 with {\arrow{<}}}}]
\usepackage[all]{xy}

\usepackage[colorlinks=true,linkcolor=blue,citecolor=blue,urlcolor=blue,bookmarks=true]{hyperref}
\usepackage[backend=biber,sorting=none,backref=true]{biblatex}
\usepackage{schmidhuber}

\usepackage{authblk}
\usepackage{nicefrac,xfrac}
\usepackage{wrapfig}
\usepackage{caption}
\usepackage{subcaption}
\usepackage{physics}
\usepackage{bbold}
\usepackage{mathtools}

\usepackage{tocloft}

\definecolor{darkblack}{rgb}{0,0,.5}
\usepackage[nameinlink]{cleveref}
\theoremstyle{plain}
\newtheorem{theorem}{Theorem}

\newtheorem{corollary}[theorem]{Corollary}
\newtheorem{proposition}[theorem]{Proposition}
\newtheorem{lemma}[theorem]{Lemma}

\theoremstyle{definition}
\newtheorem{example}[theorem]{Example}

\newtheorem{conjecture}[theorem]{Conjecture}
\theoremstyle{definition}
\newtheorem{remark}[theorem]{Remark}

\numberwithin{equation}{section}

\RequirePackage{setspace}

\newcommand{\hackcenter}[1]{
 \xy (0,0)*{#1}; \endxy}

 \usepackage{bbm}
 \def\1{\mathbbm{1}}%
\newcommand{\maps}{\colon}

\usepackage{todonotes}

\newcommand{\Ker}{\text{Ker }}

\newcommand{\tens}{\otimes}

\newcommand{\del}{\partial}
\newcommand{\x}{\textbf{x}}

\newcommand{\bigslant}[2]{{\raisebox{.2em}{#1}\left/\raisebox{-.2em}{#2}\right.}}

\renewcommand{\to}{\rightarrow}

\newcommand{\sphere}[1][1]{
\xy
(0,0)*{
\begin{tikzpicture} [fill opacity=0.2, scale=#1]
	\path [fill=black] (0,0) circle (1);
	\draw (-1,0) .. controls (-1,-.4) and (1,-.4) .. (1,0);
	\draw[dashed] (-1,0) .. controls (-1,.4) and (1,.4) .. (1,0);
	\draw[very thick] (0,0) circle (1);
\end{tikzpicture}};
\endxy}

\usepackage{amsmath}
\newcommand{\und}{\underline}
\newcommand\no[1]{}

\makeatletter
\DeclareFontFamily{OMX}{MnSymbolE}{}
\DeclareSymbolFont{MnLargeSymbols}{OMX}{MnSymbolE}{m}{n}
\SetSymbolFont{MnLargeSymbols}{bold}{OMX}{MnSymbolE}{b}{n}
\DeclareFontShape{OMX}{MnSymbolE}{m}{n}{
    <-6>  MnSymbolE5
   <6-7>  MnSymbolE6
   <7-8>  MnSymbolE7
   <8-9>  MnSymbolE8
   <9-10> MnSymbolE9
  <10-12> MnSymbolE10
  <12->   MnSymbolE12
}{}
\DeclareFontShape{OMX}{MnSymbolE}{b}{n}{
    <-6>  MnSymbolE-Bold5
   <6-7>  MnSymbolE-Bold6
   <7-8>  MnSymbolE-Bold7
   <8-9>  MnSymbolE-Bold8
   <9-10> MnSymbolE-Bold9
  <10-12> MnSymbolE-Bold10
  <12->   MnSymbolE-Bold12
}{}

\let\llangle\@undefined
\let\rrangle\@undefined
\DeclareMathDelimiter{\llangle}{\mathopen}%
                     {MnLargeSymbols}{'164}{MnLargeSymbols}{'164}
\DeclareMathDelimiter{\rrangle}{\mathclose}%
                     {MnLargeSymbols}{'171}{MnLargeSymbols}{'171}

\usepackage{pict2e}

\makeatletter

\newcommand*\@KP@Large@frame[2]{%
    \setlength\unitlength{\fontdimen 22 #1\tw@}%
    \vrule \@width\z@ \@height 4\unitlength \@depth\tw@\unitlength
    \begin{picture}(6,2)(-3,-1)%
        \def\@KP@Radius     {3}%
        \def\@KP@Hole@radius{.5}
        \def\@KP@Diameter   {6}%
        #2%
    \end{picture}%
}
\newcommand*\@KP@Small@frame[2]{%
    \setlength\unitlength{\fontdimen 22 #1\tw@}%
    \vrule \@width\z@ \@height \thr@@\unitlength \@depth\@ne\unitlength
    \begin{picture}(4,2)(-2,-1)%
        \def\@KP@Radius     {2}%
        \def\@KP@Hole@radius{.5}
        \def\@KP@Diameter   {4}%
        #2%
    \end{picture}%
}

\newcommand*\@KP@Radius     {}
\newcommand*\@KP@Hole@radius{}
\newcommand*\@KP@Diameter   {}
\newcommand*\@KP@Shape@A{%
    \put(0,0){\circle{\@KP@Diameter}}%
}
\newcommand*\@KP@Shape@B{%
    \Line(-\@KP@Radius,\@KP@Radius )(\@KP@Radius,-\@KP@Radius)%
    \Line(-\@KP@Radius,-\@KP@Radius)(-\@KP@Hole@radius,-\@KP@Hole@radius)%
    \Line(\@KP@Radius ,\@KP@Radius )(\@KP@Hole@radius ,\@KP@Hole@radius )%
}
\newcommand*\@KP@Shape@C{%
    \cbezier(-\@KP@Radius,\@KP@Radius )(0,0)(0,0)(\@KP@Radius,\@KP@Radius )%
    \cbezier(-\@KP@Radius,-\@KP@Radius)(0,0)(0,0)(\@KP@Radius,-\@KP@Radius)%
}
\newcommand*\@KP@Shape@D{%
    \cbezier(-\@KP@Radius,-\@KP@Radius)(0,0)(0,0)(-\@KP@Radius,\@KP@Radius)%
    \cbezier(\@KP@Radius ,-\@KP@Radius)(0,0)(0,0)(\@KP@Radius ,\@KP@Radius)%
}

\newcommand*\@KP@Atomic@mathpalette[1]{%
    \mathinner{
        \mathchoice{%
            \linethickness{.6\p@}
            \@KP@Large@frame \textfont {#1}%
        }{%
            \linethickness{.4\p@}
            \@KP@Small@frame \textfont {#1}%
        }{%
            \linethickness{.3\p@}
            \@KP@Small@frame \scriptfont {#1}%
        }{%
            \linethickness{.2\p@}
            \@KP@Small@frame \scriptscriptfont {#1}%
        }%
    }%
}

\newcommand*\KPA{\@KP@Atomic@mathpalette \@KP@Shape@A}
\newcommand*\KPB{\@KP@Atomic@mathpalette \@KP@Shape@B}
\newcommand*\KPC{\@KP@Atomic@mathpalette \@KP@Shape@C}
\newcommand*\KPD{\@KP@Atomic@mathpalette \@KP@Shape@D}

\makeatother

\title{A quantum algorithm for Khovanov homology}
\author[1]{Alexander Schmidhuber\thanks{Corresponding author: alexsc@mit.edu}}
\author[2]{Michele Reilly}
\author[3,4]{Paolo Zanardi}
\author[2]{\\Seth Lloyd}
\author[3,4]{Aaron Lauda\thanks{Corresponding author: lauda@usc.edu}}

\affil[1]{Center for Theoretical Physics, Massachusetts Institute of Technology, Cambridge, MA 02139}
\affil[2]{Department of Mechanical Engineering, Massachusetts Institute of Technology, Cambridge, MA 02139}
\affil[3]{Department of Physics, University of Southern California, Los Angeles, CA 90007}
\affil[4]{Department of Mathematics, University of Southern California, Los Angeles, CA 90007}

\date{\today}
\begin{document}

\noindent
\hspace{\fill} MIT-CTP/5803
\begingroup
\let\newpage\relax
    \maketitle
\endgroup

\maketitle

\begin{abstract}
Khovanov homology is a topological knot invariant that categorifies the Jones polynomial, recognizes the unknot, and is conjectured to appear as an observable in $4D$ supersymmetric Yang--Mills theory. Despite its rich mathematical and physical significance, the computational complexity of Khovanov homology remains largely unknown. To address this challenge, this work initiates the study of efficient quantum algorithms for Khovanov homology.
 
We provide simple proofs that increasingly accurate additive approximations to the ranks of Khovanov homology are $\DQC$-hard, $\BQP$-hard, and $\sharpP$-hard, respectively. For the first
two approximation regimes, we propose a novel quantum algorithm. Our algorithm is efficient provided the corresponding Hodge Laplacian thermalizes in polynomial time and has a sufficiently large spectral gap, for which we give numerical and analytical evidence.
 
Our approach introduces a pre-thermalization procedure that allows our quantum algorithm to succeed even if the Betti numbers of Khovanov homology are much smaller than the dimensions of the corresponding chain spaces, overcoming a limitation of prior quantum homology algorithms. We introduce novel connections between Khovanov homology and graph theory to derive analytic lower bounds on the spectral gap.  

\end{abstract}

\clearpage
{\hypersetup{linkcolor=black}
\tableofcontents
}


\clearpage


\section{Introduction}

A central goal of quantum computing research is to identify new computational problems that are intractable for classical computers but can be solved efficiently on a quantum computer. Algebraic topology has emerged as a rich source of such problems, with the approximation of the Jones polynomial as a key example. In this work, we focus on quantum algorithms for a more general topological invariant, called \emph{Khovanov homology}, which is a categorification of the Jones polynomial with connections to supersymmetric quantum field theory and the unknotting problem. 

\paragraph{Jones polynomial and topological quantum field theory.} The Jones polynomial \cite{jones1997polynomial} is a topological invariant associated with knots and links in 3-dimensional space.  This polynomial does not change under continuous deformations of a knot and can thereby be used to distinguish different knots. As a computational problem, computing even an approximation of the Jones polynomial is intractable for classical computers, as the complexity scales exponentially in the number of crossings of the knot~\cite{kuperberg2015hard}.  Somewhat surprisingly, though, there exist \emph{quantum} algorithms that can efficiently approximate the Jones polynomial, exponentially faster than any known classical algorithm~\cite{Aharonov-Jones-Landau}. In fact, an algorithm that can approximate the Jones polynomial can simulate any computation that a quantum computer can perform efficiently, that is, the problem of approximating the Jones polynomial is $\BQP$-complete~\cite{freedman2002simulation,aharonov2011bqp}.   

One compelling explanation for why the Jones polynomial is so amenable to quantum algorithms is Witten's work showing that the Jones polynomial has a physical incarnation as observables associated to Wilson loops in 3D Chern--Simons theory~\cite{Witten-Jones}. This physical interpretation of the Jones polynomial within topological quantum field theory has also led to advances in low-dimensional topology and the study of 3-manifolds~\cite{MR1091619}, integrable systems and the theory of quantum groups~\cite{Kauff,Jones-stat}, and
inspired new approaches to quantum computing, such as topological quantum computation~\cite{FKLW,Kitaev}, which leverages the topological nature of these theories to develop naturally error-resistant quantum computational schemes.    

\paragraph{Khovanov homology and categorification.} 
Nearly two decades after its discovery, Mikhail Khovanov realized that the Jones polynomial is in fact a simplified manifestation of a much richer invariant associated with knots and links~\cite{Kh1}.  Khovanov showed that the Jones polynomial could be lifted to a new homological invariant of knots, where each knot is assigned a chain complex whose homology is topologically invariant. Each homology group is equipped with an internal grading, and the dimensions of each graded space of each homology group are succinctly summarized as the ranks, or \emph{Betti numbers}, of these spaces.  The Jones polynomial can be recovered from Khovanov homology by forgetting some of the information -- specifically, taking the graded Euler characteristic of the homology theory, which is the alternating sum of the graded dimensions of each homology group (see Section~\ref{subsec:Categorification} for more details).   This type of enhancement of a mathematical object to one with more structure is called \textit{categorification}; we will give more intuition to categorification in Section~\ref{subsec:Categorification}.  

Khovanov homology is a much more powerful topological invariant of knots in that many knots that have the same Jones polynomial are distinguished by the Betti numbers of their Khovanov homology. Beyond distinguishing different knots, Khovanov homology has a property known as functoriality that has proven to be a powerful tool in the field of low-dimensional topology, providing a new approach for probing 4-dimensional smooth topology.  We discuss these connections more in Section~\ref{subsubsec-Why}. 

Just as for the Jones polynomial, all known classical algorithms for computing or approximating Khovanov homology are inefficient. Bar-Natan pioneered the first such algorithm~\cite{BN1} and subsequently developed improved divide and conquer algorithms~\cite{BN3}. While these algorithms can perform well in practice for knots of reasonable size, their asymptotic scaling is still expected to scale exponentially in (the square root of) the number of crossings of the knot. 

Given that the connection between the Jones polynomial and observables in 3D Chern--Simons theory resulted in provable exponential quantum speedups for approximating the Jones polynomial, it is a natural and long-standing open question to design a quantum algorithm for efficiently approximating Khovanov homology.    One indication that this might be possible is that the categorification of the Jones polynomial extends to its connection to physical theories.  Witten showed that Chern--Simons theory can be categorified via  4D $\mathcal{N}=2$ supersymmetric Yang--Mills theory~\cite{Witten-fivebrane,witten2012khovanov, Gaiotto-Witten}, where Khovanov homology can be interpreted as surfaces operators. Earlier work initiated by Gukov-Schwarz-Vafa~\cite{Gukov_2005,Aganagic} interprets Khovanov homology as BPS states in topological string theory, see also more recent work~\cite{GPPV,Aganagic_2024}.  These physical realizations of Khovanov homology suggest a framework may exist to design efficient quantum algorithms for Khovanov homology.  

\paragraph{Quantum algorithms for homology and topological data analysis. }
Another motivation for studying quantum algorithms for Khovanov homology is the emerging field of quantum Topological Data Analysis (qTDA), which has experienced a surge of attention recently, see e.g. \cite{lloyd2016quantum,quantumAdvantage,persistent,GoogleBettiBerry,AmazonBettiMcArdle,schmidhuber2023complexity,crichigno2022clique,king2023promise,hayakawa2024quantum,gyurik_schmidhuber_2024quantum}. This line of research established that quantum algorithms can offer exponential quantum speedups \cite{gyurik_schmidhuber_2024quantum} for computing certain properties of simplicial homology, specifically the persistent clique homology appearing in Topological Data Analysis (TDA). The quantum homology algorithm proposed in \cite{lloyd2016quantum} applies to any simplicial complex. In this work, we study whether these speedups extend to Khovanov homology, which is an example of a more general homology theory not arising as the simplicial homology of a simplicial complex. 

It has already been realized that the quantum speedups arising in clique homology can sometimes be understood through the lens of supersymmetric quantum mechanics \cite{Witten:1982im, cade2021complexity}. This body of work serves as another suggestion that the recurring connections between algebraic topology and quantum physics may be a source of quantum speedups for computing homologies. 

\paragraph{Detecting the unknot.}
An additional motivation for studying Khovanov homology arises from the \emph{unknotting problem}, which is the task of algorithmically recognizing whether a given description of a knot corresponds to the unknot. Unlike the Jones polynomial, Khovanov homology is known to provably detect the unknot due to Kronheimer and Mrowka~\cite{KM}. 

It is a major unresolved challenge to determine whether there exists a polynomial time (classical or quantum) algorithm for the unknotting problem. This is despite the fact that no complexity-theoretic hardness results for the unknotting problem are known. Detecting the unknot is in the complexity class $\NP$ $\cap$ co-$\NP$, meaning it is believed to be an \emph{NP-intermediate} problem. NP-intermediate problems, such as integer factorization, might be good candidates for exponential quantum advantage, since an efficient quantum algorithm for such a problem does not clash with the widely held belief that $\NP \not \subseteq \BQP$. 

Our quantum algorithm for Khovanov homology is one possible approach toward detecting the unknot, although the efficiency of this approach is still open, as we discuss in more detail at the end of the next subsection.

\subsection{Results}
This paper initiates the formal study of efficient quantum algorithms for computing Khovanov homology. To ensure accessibility for a broad audience, we include a self-contained introduction to Khovanov homology and its underlying constructions in \Cref{sec:background}, which requires no prior knowledge of algebraic topology. We then present our main results, which are outlined in the following. 

\paragraph{A quantum algorithm for Khovanov homology.}
We start by describing general quantum algorithms for arbitrary homologies in \Cref{sec:general_homology}.   The original quantum homology algorithm \cite{lloyd2016quantum} operates by treating the Hodge Laplacian of chain complexes as a Hamiltonian for a physical system. Because the Hodge Laplacian is sparse, standard Hamiltonian simulation techniques can be used together with the quantum phase estimation algorithm \cite{kitaev1995quantum} to project onto the ground state sector of that system. This quantum procedure reveals the dimension of that kernel -- the Betti numbers of the homology -- and the corresponding harmonic representatives. 

In \Cref{sec:algorithm}, we then explicitly describe a quantum algorithm that computes an additive approximation to the ranks of Khovanov homology, given a planar representation of a knot as input. Unlike the original quantum homology algorithm, we incorporate a pre-thermalization procedure to enhance the overlap with the kernel of the Hodge Laplacian, allowing for the estimation of Betti numbers even when they are significantly smaller than the dimension of the chain space. Our algorithm works by encoding the Khovanov homology of a knot in the ground state space of the homology's Hodge Laplacian. In this language, the boundary operator of Khovanov homology is described in terms of Jordan Wigner creation operators. To implement our algorithm, we introduce certain encodings of a knot, construct efficient encodings of the Hodge Laplacian and of the boundary operator of Khovanov homology, and employ a variation of tools developed in the context of quantum TDA. The runtime of our quantum algorithm is primarily determined by the inverse spectral gap of the Hodge Laplacian and the thermalization time needed to cool the Laplacian, both of which we discuss in detail below. 

Since Khovanov homology categorifies the Jones polynomial, our quantum algorithm can also be used to approximate the Jones polynomial. While previous quantum algorithms only approximate the Jones polynomial at a root of unity, our quantum algorithm extends to arbitrary values, although it is not necessarily always efficient.

\paragraph{Complexity-theoretic hardness.}
In \Cref{sec:complexity}, we derive lower-bounds on the hardness of approximating Khovanov homology. Specifically, we show that increasingly accurate additive approximations to the ranks (Betti numbers) of Khovanov homology are $\DQC$-hard, $\BQP$-hard, and $\sharpP$-hard, respectively. For a definition of these complexity classes, see \Cref{sec:complex_prelim}. We also discuss and highlight several open questions in that regard. Our complexity-theoretic lower-bounds are simple corollaries of known results for estimating the Jones polynomial due to Freedman et al. \cite{freedman2002simulation}, Aharonov et al. \cite{Aharonov-Jones-Landau}, Kuperberg \cite{kuperberg2015hard}, Shor et al. \cite{shor2007estimating}, and Aharonov et al. \cite{aharonov2011bqp}.

\paragraph{Spectral gaps and homological perturbation theory.}
As for other quantum algorithms for computing homologies, our algorithm has to resolve the spectral gap of the Hodge Laplacian in order to accurately count Betti numbers. If this spectral gap is exponentially small, there is no hope to distinguish efficiently between generators of Khovanov homology and other low-lying eigenstates of the Laplacian.

This difficulty is additionally enhanced by the fact that, unlike the groundstate space of the Laplacian, the spectral gap is \emph{not} a topological invariant of the knot. In \Cref{sec:spectral_gaps}, we leverage homological perturbation theory to give analytic results characterizing the behavior of the spectral gap under certain types of changes in the knot diagram.  The techniques from homological perturbation theory may be of independent interest as they allow us to reason about the impact that maps of chain complexes can have on their corresponding Hodge Laplacians.   In this section, we also discuss a strategy for increasing the spectral gap using known techniques to simplify the classical computation of Khovanov homology.

\paragraph{Numerical and analytic bounds on the spectral gap.}
In \Cref{sec:numerics_gap}, we provide extensive numerical computations of the spectral gap of the Hodge Laplacian in Khovanov homology. Our results suggest that the scaling of the spectral gap is only inverse polynomial in the number of crossings.  We combine these numerical results with matching theoretical lower bounds in \Cref{sec:graph}. 
This section also develops new connections between extremal homological degrees in Khovanov homology and graph theory.  We show that the combinatorial Laplacians in these degrees coincide with the so-called `signless Laplace matrices" of an associated graph we introduce in this work.    Utilizing this relationship we give bounds on the spectral gap in these homological degrees.

\paragraph{Thermalization of the Hodge Laplacian.} A key novel subroutine in our quantum algorithm is a pre-thermalization procedure that generalizes the original quantum homology algorithm. The original quantum homology algorithm \cite{lloyd2016quantum} uses a random initial state for quantum phase estimation, therefore the projection onto the ground state sector of the Hodge Laplacian succeeds only when the kernel is not too small compared to the dimension of the entire Hilbert space of the chain complex. We find empirically that the Betti numbers of Khovanov homology can indeed be small, hence a different approach is required.


In \Cref{sec:thermalization}, we show that replacing the random initial state for quantum phase estimation with a Gibbs state can significantly improve the efficiency of the algorithm. To do so, we first simulate cooling the quantum system whose Hamiltonian is the Hodge Laplacian down to low temperature, so that the state of the system has significant overlap with the ground state sector.  We do not expect thermalization to succeed when the ground state sector of the Laplacian encodes the answer to some $\sharpP$-hard problem.  However, recent work on Gibbs sampling shows that such thermalization can often be accomplished efficiently~\cite{BrandaoThermal21, SommaThermal22, GilyenThermal23a, GilyenThermal23, GilyenThermal24, LinLinThermal24}.   

Next, as in the original algorithm, we use quantum phase estimation to project onto the kernel of the Laplacian.  As long as we are able to thermalize to a temperature not much larger than the gap of the Laplacian, this projection will succeed with sufficiently high probability. The system is now in a fully mixed state over the kernel of the Laplacian.    

By preparing a low-temperature Gibbs state rather than a high-temperature fully mixed state, we have lost the connection between the probability of success of the projection onto the kernel and the estimate of the corresponding Betti number. Now we employ a method proposed in \cite{scali2024topology}. We prepare multiple such fully mixed states over the kernel, and apply a SWAP test: as long as the dimension of the kernel is small, the probability of success of the SWAP test then gives an estimate of the corresponding Betti number.   We see that where before, the low dimension of the kernel was a hindrance to estimating the Betti numbers, now it is an asset. 

The pre-thermalization procedure is particularly interesting in the context of the unknotting problem. The unknot has trivial homology -- an unknotted loop has a simple two-dimensional kernel.  We conjecture that it may be possible to efficiently cool the Hodge Laplacian for (any representation of) the unknot to sufficiently low temperatures that the system has inverse-polynomial overlap with that two-dimensional kernel. Successful projection onto the kernel then allows us to verify that a given knot is in fact the unknot. However, proving this conjecture -- or more generally, identifying the conditions under which efficient thermalization of the Hodge Laplacian is achievable -- lies beyond the scope of this work.

\subsection{Related work}
The present work can be understood as the continuation of three different lines of research. We continue Bar-Natan's line of work~\cite{BN1,BN3} of developing faster algorithms for Khovanov homology by considering quantum algorithms instead of classical algorithms. The encoding we use in our quantum algorithm for Khovanov homology is inspired by Bar-Natan's classical algorithm for computing Khovanov homology~\cite{BN1}. 
More generally, the techniques used to develop our quantum algorithms are inspired by quantum algorithms for computing homology, which were first proposed in the context of quantum Topological Data Analysis \cite{lloyd2016quantum}.
Lastly, our complexity-theoretic proofs are based on the long line of results on quantum algorithms for the Jones polynomial and more general topological invariants, see e.g.  \cite{freedman2002simulation,Aharonov-Jones-Landau,kuperberg2015hard,shor2007estimating,aharonov2011bqp,Potts,jordan2008estimating}.  

Connections between Khovanov homology and quantum computing have been previously explored in a different context. Kauffman \cite{kauffman2010-Khovanov} describes an (inefficient) quantum algorithm for the Jones polynomial and shows that the algorithm (when seen as a unitary) commutes with the boundary operator in Khovanov homology, but does not consider quantum algorithms for Khovanov homology. Audoux \cite{audoux2014application} and later Harned et al.~\cite{harned2024khovanov} consider applications of Khovanov homology to quantum codes.

\subsection{Acknowledgments}
A.S. is grateful to Anna Beliakova and Aram Harrow for their valuable comments. A.S. is supported by the Simons Foundation (MP-SIP-00001553, AWH) and NSF grant PHY-2325080. S.L. was supported by DOE and by ARO under MURI grant W911NF2310255. P.Z. was supported by NSF grant PHY2310227. A.D.L.\ is partially supported by NSF grants DMS-1902092 and DMS-2200419, and the Simons Foundation collaboration grant on New Structures in Low-dimensional Topology.  A.D.L. and P.Z. were both supported by the Army Research Office W911NF-20-1-0075.  Computations associated with this project were conducted utilizing the Center for Advanced Research Computing (CARC) at the University of Southern California.  These computations made use of Bar-Natan's KnotTheory package as well as knot encodings from the KnotInfo site.

\section{Background on Khovanov homology}
\label{sec:background}
In the following sections, we provide a brief but self-contained introduction to Khovanov homology and its underlying constructions.

\subsection{Knots}
A knot is, informally, anything that can be constructed by tying up a string and then gluing together its two ends. Two knots are considered equivalent if they can be continuously deformed into each other without cutting the string or allowing it to pass through itself.  More formally, a knot $K$ is an equivalence class of embeddings of the circle $S^1$ into 3-dimensional Euclidean space $\R^3$ (or sometimes the 3-sphere for convenience).    Knots are considered equivalent if they are \emph{ambient isotopic}, meaning that a continuous deformation of $\R^3$ takes one into the other.  
\[
\hackcenter{\begin{tikzpicture}[scale=.35]
\draw [black, very thick, ->] (0,-.3) circle (1);
\end{tikzpicture}}
\qquad
\begin{tikzpicture}
 \draw[very thick, ->] (0,0) to (1,0);
\end{tikzpicture}
\qquad
\hackcenter{
\begin{tikzpicture} [scale=.44]
	\draw [very thick, fill=gray, fill opacity=0.2] (0,-.3) circle (3);
	\draw[dashed, thick] (-3,-.3) .. controls (-1.8,.9) and (1.8,.9) .. (3,-.3);
\draw[black,very thick] (0,0) to [out=90,in=270] (.5,1);
\draw[black,very thick] (.2,.6) to [out=135,in=270] (0,1);
\draw[black,very thick] (.5,0) to [out=90,in=315] (.3,.4);
\draw[black,very thick] (0,-1) to [out=90,in=270] (.5,0);
\draw[black,very thick] (.2,-.4) to [out=135,in=270] (0,0);
\draw[black,very thick] (.5,-1) to [out=90,in=315] (.3,-.6);
\draw[black,very thick] (0,-2) to [out=90,in=270] (.5,-1);
\draw[black,very thick] (.2,-1.4) to [out=135,in=270] (0,-1);
\draw[black,very thick] (.5,-2) to [out=90,in=315] (.3,-1.6);
\draw[black,very thick, directed=.5] (.5,1) to [out=90,in=180] (.75,1.4) to [out=0,in=90] (1,1) to
	(1,-2) to [out=270,in=0] (.75,-2.4) to [out=180,in=270] (.5,-2);
\draw[black, very thick, directed=.5] (0,1) to [out=90,in=0] (-.25,1.4) to [out=180,in=90] (-.5,1) to
	(-.5,-2) to [out=270,in=180] (-.25,-2.4) to [out=0,in=270] (0,-2);
	\draw[thick] (-3,-.3) .. controls (-1.8,-1.5) and (1.8,-1.5) .. (3,-.3);
\end{tikzpicture} }
\]

While knots live in three dimensions, for practical purposes they are depicted by 2-dimensional \textit{knot diagrams} obtained from a generic projection onto $\R^2$ that keeps track of over/under crossings such as in Figure~\ref{fig:trefoil-diagram}. Evidently, different knot diagrams can represent the same knot.
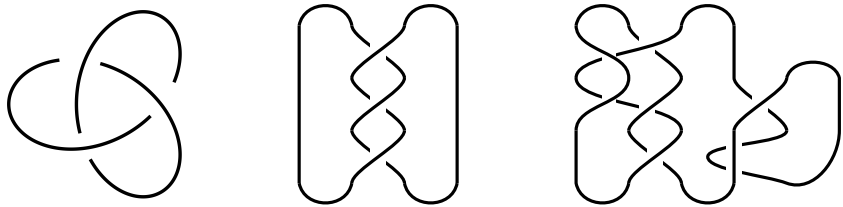
\begin{figure}[h!]
    \centering
$
\hackcenter{ 
\begin{tikzpicture}[scale=0.45, xscale=-1.0]
\draw[ domain=-10:95, variable=\t,samples=200, line width=0.45mm, black]
plot ({-cos(\t)+2*cos(2*\t)}, {sin(\t)+2*sin(2*\t)});
\draw[ domain=230:340, variable=\t,samples=200, line width=0.45mm, black]
plot ({-cos(\t)+2*cos(2*\t)}, {sin(\t)+2*sin(2*\t)});
\draw[ domain=110:215, variable=\t,samples=200, line width=0.45mm, black]
plot ({-cos(\t)+2*cos(2*\t)}, {sin(\t)+2*sin(2*\t)});
\end{tikzpicture}  } 
\qquad \qquad 
\hackcenter{
 \begin{tikzpicture}[scale=.7]
  \draw [black, very thick]    (2,1) .. controls +(0,.25) and +(0,-.25) .. (1,2);
  \draw [black, very thick]    (2,2) .. controls +(0,.25) and +(0,-.25) .. (1,3);
  \draw [black, very thick]    (2,3) .. controls +(0,.25) and +(0,-.25) .. (1,4);
  \path [fill=white] (1.35,1) rectangle (1.65,4);
  \draw [black, very thick]    (1,1) .. controls +(0,.25) and +(0,-.25) .. (2,2);
  \draw [black, very thick]    (1,2) .. controls +(0,.25) and +(0,-.25) .. (2,3);
    \draw [black, very thick]    (1,3) .. controls +(0,.25) and +(0,-.25) .. (2,4);
  \draw [black, very thick]    (1,4) .. controls +(-.1,.5) and +(+.1,.5) .. (0,4);
  \draw [black, very thick]    (2,4) .. controls +(+.1,.5) and +(-.1,.5) .. (3,4);
  \draw [black, very thick]    (0,4) -- (0,1);
  \draw [black, very thick]    (3,4) -- (3,1);
  \draw [black, very thick]    (1,1) .. controls +(-0.1,-0.5) and+(+0.1,-0.5) .. (0,1);
  \draw [black, very thick]    (2,1) .. controls +(0.1,-0.5) and +(-0.1,-0.5) .. (3,1);
    \end{tikzpicture}}
    \qquad  \qquad 
  \hackcenter{   \begin{tikzpicture}[scale=.7]
  \draw [black, very thick]    (2,1) .. controls +(0,.25) and +(0,-.25) .. (1,2);
  \draw [black, very thick]    (2,2) .. controls +(0,.5) and +(0,-.5) .. (0,3);
  \draw [black, very thick]    (2,3) .. controls +(0,.25) and +(0,-.25) .. (1,4);
  \path [fill=white] (1.35,1) rectangle (1.65,2);
    \path [fill=white] (1.2,2) rectangle (1.5,4);
  \draw [black, very thick]    (1,1) .. controls +(0,.25) and +(0,-.25) .. (2,2);
  \draw [black, very thick]    (1,2) .. controls +(0,.25) and +(0,-.25) .. (2,3);
    \draw [black, very thick]    (0,3) .. controls +(0,.5) and +(0,-.5) .. (2,4);
    \draw [black, very thick]    (4,2) .. controls +(0,.25) and +(0,-.25) .. (3,3) -- (3,4);
      \path [fill=white] (3.35,2) rectangle (3.65,3);
   \draw [black, very thick]    (4,3) .. controls +(0,-.25) and +(0,+.25) .. (3,2);
   \draw [black, very thick]   (4,2) .. controls +(0,-.25) and +(0,+.25) .. (2.5,1.5)
                                     .. controls +(0,-.25) and +(-.4,+.15) .. (4,1) 
                                      to[out=-20,in=-90] (5,2) -- (5,3) 
                                      .. controls +(-.1,.4) and +(.1,+.4) .. (4,3);
     \path [fill=white] (2.85,1.6) rectangle (3.15,1.8);
     \path [fill=white] (2.85,1.1) rectangle (3.15,1.4);
  \draw [black, very thick]    (1,4) .. controls +(-.1,.5) and +(+.1,.5) .. (0,4);
  \draw [black, very thick]    (2,4) .. controls +(+.1,.5) and +(-.1,.5) .. (3,4);
  \path [fill=white] (.5,2) rectangle (.75,4);
  \draw [black, very thick]    (0,4) .. controls +(0,-.5) and +(0,.5) .. (1,3)
                                .. controls +(0,-.5) and +(0,.5) .. (0,2);
  \draw [black, very thick]    (3,2) -- (3,1);
  \draw [black, very thick]    (1,1) .. controls +(-0.1,-0.5) and+(+0.1,-0.5) .. (0,1) -- (0,2);
  \draw [black, very thick]    (2,1) .. controls +(0.1,-0.5) and +(-0.1,-0.5) .. (3,1);
    \end{tikzpicture}}
$
    \caption{Three different planar diagrams for a trefoil knot.  }
    \label{fig:trefoil-diagram}
\end{figure}

A \emph{link} is a generalization of a knot, formed by multiple nonintersecting knots that may be linked or knotted together. It is sometimes useful to specify the orientation of a knot, which is one of two directional choices for traveling around a knot. Similarly, an oriented link is a link where each component (each nonintersecting knot) has a specified orientation.

A quantity associated with a knot or link diagram that is invariant under continuous deformations is called a \emph{knot invariant}.  Knot invariants can be constructed by assigning some algebraic data to a knot projection that is invariant under the three Reidemeister moves shown below~\cite{MR1472978,MR907872}. 

\begin{equation} \label{eq:Reidemeistere}
    \hackcenter{
\begin{tikzpicture}[scale=.5]
\draw[very thick, black] (.75,0) .. controls ++(-.2,.8) and ++(0,-.8) ..   (-.75,2) to (-.75,3);
\path [fill=white] (-.2,.7) rectangle (.2,1.25);
\draw[very thick, black](-.75,-1) to  (-.75,0) 
    .. controls ++(0,.8) and ++(-.2,-.8) ..   (.75,2) 
    .. controls +(.3,.8) and +(0,.8) .. (2.25,2) to (2.25,0)
    .. controls +(0,-.8) and +(.2,-.8) .. (.75,0);
\end{tikzpicture} }
\quad \overset{\raisebox{1ex}{R1}}{\Leftrightarrow} \quad 
\hackcenter{
\begin{tikzpicture}[scale=.5]
\draw[very thick, black] (.75,-1)  to (.75,3);
\end{tikzpicture} }
\qquad \qquad 
\hackcenter{
\begin{tikzpicture}[scale=.5]
\draw[very thick, black,-] (-.75,3) .. controls ++(0,.8) and ++(0,-0.8) ..   (0.75,5);
\draw[very thick, black,-] (.75,1) .. controls ++(0,.8) and ++(0,-0.8) ..   (-.75,3);
\path [fill=white] (-.2,1.6) rectangle (.2,4.5);
\draw[very thick, black,-] (.75,3) .. controls ++(0,.8) and ++(0,-0.8) ..   (-.75,5);
\draw[very thick, black,-] (-.75,1) .. controls ++(0,.8) and ++(0,-0.8) ..   (0.75,3);
\end{tikzpicture} }
 \quad
\overset{\raisebox{1ex}{R2}}{\Leftrightarrow} 
 \quad
 \hackcenter{
\begin{tikzpicture}[scale=.5]
\draw[very thick, black,-] (-.75,1)  to  (-0.75,5);
\draw[very thick, black,-] (.75,1)  to  (0.75,5);
\end{tikzpicture} }
\qquad \qquad
\hackcenter{
\begin{tikzpicture}[scale=.5]
\draw[very thick, black] (.75,0) .. controls ++(0,.8) and ++(0,-.8) ..   (-.75,2) to (-.75,3);
\draw[very thick, black] (2.25,0) to (2.25,1) .. controls ++(0,.8) and ++(0,-.8) ..  (.75,3);
\path [fill=white] (-.2,.7) rectangle (.2,1.25);
\path [fill=white] (1.3,1.74) rectangle (1.7,2.3);
\draw[very thick, black,-] (-.75,0) .. controls ++(0,1) and ++(0,-1) ..   (2.25,3) to (2.25,4.5);
\draw[very thick, black,-] (.75,3) .. controls ++(0,.8) and ++(0,-0.8) ..   (-.75,4.5);
\path [fill=white] (-.2,3.6) rectangle (.2,4);
\draw[very thick, black,-] (-.75,3) .. controls ++(0,.8) and ++(0,-0.8) ..   (0.75,4.5);
\end{tikzpicture} }
 \quad
\overset{\raisebox{1ex}{R3}}{\Leftrightarrow} 
 \quad
\hackcenter{
\begin{tikzpicture}[scale=.5, xscale=-1, yscale=-1]
\draw[very thick, black,-] (.75,0) .. controls ++(0,.8) and ++(0,-.8) ..   (-.75,2) to (-.75,3);
\draw[very thick, black,-] (2.25,0) to (2.25,1) .. controls ++(0,.8) and ++(0,-.8) ..  (.75,3);
\path [fill=white] (-.2,.7) rectangle (.2,1.25);
\path [fill=white] (1.3,1.74) rectangle (1.7,2.3);
\draw[very thick, black,-] (-.75,0) .. controls ++(0,1) and ++(0,-1) ..   (2.25,3) to (2.25,4.5);
\draw[very thick, black] (.75,3) .. controls ++(0,.8) and ++(0,-0.8) ..   (-.75,4.5);
\path [fill=white] (-.2,3.6) rectangle (.2,4);
\draw[very thick, black] (-.75,3) .. controls ++(0,.8) and ++(0,-0.8) ..   (0.75,4.5);
\end{tikzpicture} }
\end{equation}

\subsection{Jones polynomial}
\label{sec:Jones_polynomial}
The Jones polynomial is a topological invariant of an oriented knot or link discovered in 1984 by Vaughan Jones~\cite{jones1997polynomial}. It brought on major advances in knot theory~\cite{Murasugi,This} and revealed novel connections between low-dimensional topology, exactly solvable models in statistical mechanics~\cite{Kauff,Jones-stat},   and Chern--Simons theory via (topological) quantum field theory~\cite{Witten-Jones}.    

Consider an oriented knot or link $K$. Given a planar projection of $K$ with $m$ crossings (such as in \Cref{fig:Kauffman1}), the \emph{Kauffman bracket} $\langle K \rangle$ of $K$ is a Laurent polynomial in $\mathbb{Z}[q,q^{-1}]$, defined recursively via the following local relations:
\begin{align}
    \left<\KPA\right> &= (q+q^{-1})
        &&\text{trivial link} \label{eq:trivial} \\
    \left<K\cup\KPA\right> &= (q+q^{-1})\langle K\rangle
        &&\text{disjoint union} \label{eq:union} \\
    \left<\KPB\right> &= \left<\KPC\right> -q\left<\KPD\right>
        &&\text{skein relation} \label{eq:skein}
\end{align}
These axioms specify a recursive algorithm that reduces any knotted diagram to a collection of unknotted circles (trivial knots), which are then reduced to a Laurent polynomial in $q$. The first and second axioms state that any unknotted loop can be removed from the diagram at the cost of adding a multiplicative factor of $(q+q^{-1})$. The third relation is interpreted as locally replacing a crossing of a knot with a sum of diagrams obtained by resolving the crossing in two possible ways.  This recursive algorithm is illustrated in the first part of Figure~\ref{fig:Kauffman1}. 

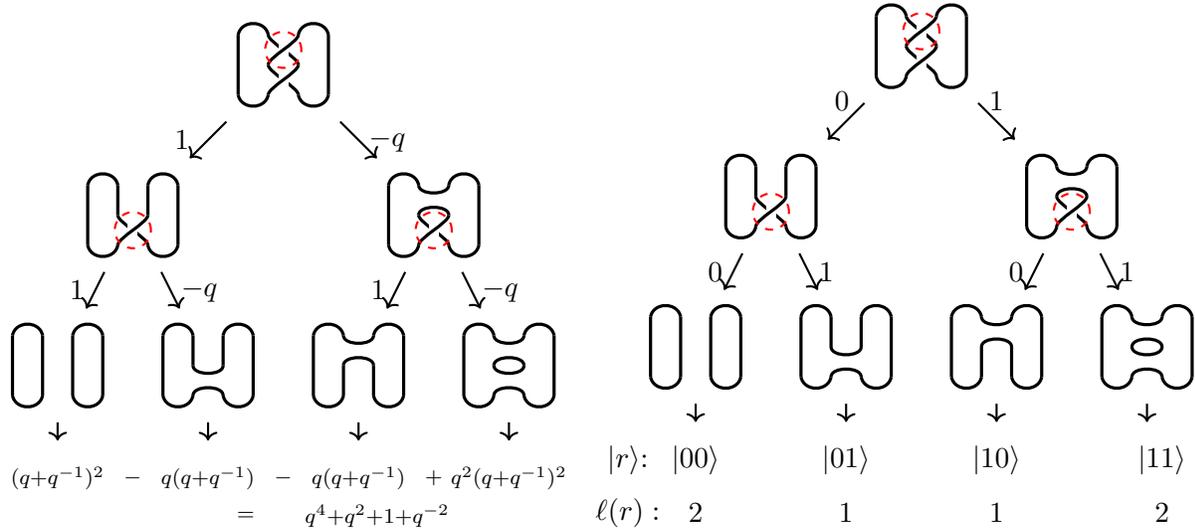
\begin{figure}
\begin{tikzpicture}

 \node (T) at (0,5) {
    \begin{tikzpicture}[scale=.4]
  \draw [black, very thick]    (2,1) .. controls +(0,.25) and +(0,-.25) .. (1,2);
  \draw [black, very thick]    (2,2) .. controls +(0,.25) and +(0,-.25) .. (1,3);
  \path [fill=white] (1.35,1) rectangle (1.65,3);
  \draw [black, very thick]    (1,1) .. controls +(0,.25) and +(0,-.25) .. (2,2);
  \draw [black, very thick]    (1,2) .. controls +(0,.25) and +(0,-.25) .. (2,3);
  \draw [black, very thick]    (1,3) .. controls +(-.1,.5) and +(+.1,.5) .. (0,3);
  \draw [black, very thick]    (2,3) .. controls +(+.1,.5) and +(-.1,.5) .. (3,3);
  \draw [black, very thick]    (0,3) -- (0,1);
  \draw [black, very thick]    (3,3) -- (3,1);
  \draw [black, very thick]    (1,1) .. controls +(-0.1,-0.5) and+(+0.1,-0.5) .. (0,1);
  \draw [black, very thick]    (2,1) .. controls +(0.1,-0.5) and +(-0.1,-0.5) .. (3,1);
 \draw [thick, red, dashed] (1.5,2.5) circle [radius=0.6];
    \end{tikzpicture} };
\node (MR) at (2,3) {
    \begin{tikzpicture}[scale=.4]
  \draw [black, very thick]    (2,1) .. controls +(0,.25) and +(0,-.25) .. (1,2);
  \path [fill=white] (1.35,1) rectangle (1.65,2);
  \draw [black, very thick]    (1,1) .. controls +(0,.25) and +(0,-.25) .. (2,2);
  \draw [black, very thick]    (1,2) .. controls +(0,.35) and +(0,.35) .. (2,2);
  \draw [black, very thick]    (2,3) .. controls +(0,-.35) and +(0,-.35) .. (1,3);
  \draw [black, very thick]    (1,3) .. controls +(-.1,.5) and +(+.1,.5) .. (0,3);
  \draw [black, very thick]    (2,3) .. controls +(+.1,.5) and +(-.1,.5) .. (3,3);
  \draw [black, very thick]    (0,3) -- (0,1);
  \draw [black, very thick]    (3,3) -- (3,1);
  \draw [black, very thick]    (1,1) .. controls +(-0.1,-0.5) and+(+0.1,-0.5) .. (0,1);
  \draw [black, very thick]    (2,1) .. controls +(0.1,-0.5) and +(-0.1,-0.5) .. (3,1);
 \draw [thick, red, dashed] (1.5,1.5) circle [radius=0.6];
    \end{tikzpicture} };
\node (ML) at (-2,3) {
        \begin{tikzpicture}[scale=.4]
  \draw [black, very thick]    (1,2) -- (1,3);
  \draw [black, very thick]    (2,2) -- (2,3);
  \draw [black, very thick]    (2,1) .. controls +(0,.25) and +(0,-.25) .. (1,2);
  \path [fill=white] (1.35,1) rectangle (1.65,2);
  \draw [black, very thick]    (1,1) .. controls +(0,.25) and +(0,-.25) .. (2,2);
  \draw [black, very thick]    (1,3) .. controls +(-.1,.5) and +(+.1,.5) .. (0,3);
  \draw [black, very thick]    (2,3) .. controls +(+.1,.5) and +(-.1,.5) .. (3,3);
  \draw [black, very thick]    (0,3) -- (0,1);
  \draw [black, very thick]    (3,3) -- (3,1);
  \draw [black, very thick]    (1,1) .. controls +(-0.1,-0.5) and+(+0.1,-0.5) .. (0,1);
  \draw [black, very thick]    (2,1) .. controls +(0.1,-0.5) and +(-0.1,-0.5) .. (3,1);
  \draw [thick, red, dashed] (1.5,1.5) circle [radius=0.6];
    \end{tikzpicture}   };
 \node (B4) at (3,1) {
        \begin{tikzpicture}[scale=.4]
  \draw [black, very thick]    (1,1) .. controls +(0,.35) and +(0,.35) .. (2,1);
  \draw [black, very thick]    (2,2) .. controls +(0,-.35) and +(0,-.35) .. (1,2);
  \draw [black, very thick]    (1,2) .. controls +(0,.35) and +(0,.35) .. (2,2);
  \draw [black, very thick]    (2,3) .. controls +(0,-.35) and +(0,-.35) .. (1,3);
  \draw [black, very thick]    (1,3) .. controls +(-.1,.5) and +(+.1,.5) .. (0,3);
  \draw [black, very thick]    (2,3) .. controls +(+.1,.5) and +(-.1,.5) .. (3,3);
  \draw [black, very thick]    (0,3) -- (0,1);
  \draw [black, very thick]    (3,3) -- (3,1);
  \draw [black, very thick]    (1,1) .. controls +(-0.1,-0.5) and+(+0.1,-0.5) .. (0,1);
  \draw [black, very thick]    (2,1) .. controls +(0.1,-0.5) and +(-0.1,-0.5) .. (3,1);
    \end{tikzpicture}};
 \node (B3) at ( 1,1) {
        \begin{tikzpicture}[scale=.4]
  \draw [black, very thick]    (1,2) -- (1,1);
  \draw [black, very thick]    (2,2) -- (2,1);
  \draw [black, very thick]    (1,2) .. controls +(0,.35) and +(0,.35) .. (2,2);
  \draw [black, very thick]    (2,3) .. controls +(0,-.35) and +(0,-.35) .. (1,3);
  \draw [black, very thick]    (1,3) .. controls +(-.1,.5) and +(+.1,.5) .. (0,3);
  \draw [black, very thick]    (2,3) .. controls +(+.1,.5) and +(-.1,.5) .. (3,3);
  \draw [black, very thick]    (0,3) -- (0,1);
  \draw [black, very thick]    (3,3) -- (3,1);
  \draw [black, very thick]    (1,1) .. controls +(-0.1,-0.5) and+(+0.1,-0.5) .. (0,1);
  \draw [black, very thick]    (2,1) .. controls +(0.1,-0.5) and +(-0.1,-0.5) .. (3,1);
        \end{tikzpicture}};
 \node (B2) at (-1,1) {
        \begin{tikzpicture}[scale=.4]
  \draw [black, very thick]    (1,1) .. controls +(0,.35) and +(0,.35) .. (2,1);
  \draw [black, very thick]    (2,2) .. controls +(0,-.35) and +(0,-.35) .. (1,2);
  \draw [black, very thick]    (1,2) -- (1,3);
  \draw [black, very thick]    (2,2) -- (2,3);
  \draw [black, very thick]    (1,3) .. controls +(-.1,.5) and +(+.1,.5) .. (0,3);
  \draw [black, very thick]    (2,3) .. controls +(+.1,.5) and +(-.1,.5) .. (3,3);
  \draw [black, very thick]    (0,3) -- (0,1);
  \draw [black, very thick]    (3,3) -- (3,1);
  \draw [black, very thick]    (1,1) .. controls +(-0.1,-0.5) and+(+0.1,-0.5) .. (0,1);
  \draw [black, very thick]    (2,1) .. controls +(0.1,-0.5) and +(-0.1,-0.5) .. (3,1);
    \end{tikzpicture}};
  \node (B1) at (-3,1) {\begin{tikzpicture}[scale=.4]
  \draw [black, very thick]    (1,1) -- (1,3);
  \draw [black, very thick]    (2,1) -- (2,3);
  \draw [black, very thick]    (1,3) .. controls +(-.1,.5) and +(+.1,.5) .. (0,3);
  \draw [black, very thick]    (2,3) .. controls +(+.1,.5) and +(-.1,.5) .. (3,3);
  \draw [black, very thick]    (0,3) -- (0,1);
  \draw [black, very thick]    (3,3) -- (3,1);
  \draw [black, very thick]    (1,1) .. controls +(-0.1,-0.5) and+(+0.1,-0.5) .. (0,1);
  \draw [black, very thick]    (2,1) .. controls +(0.1,-0.5) and +(-0.1,-0.5) .. (3,1);
    \end{tikzpicture}};
    \draw [thick,->] (T) -- node[left]{$1\;$} (ML);
    \draw [thick,->] (T) -- node[right]{$ -q$} (MR);
    \draw [thick,->] (ML) -- node[left]{$ 1$} (B1);
    \draw [thick,->] (ML) -- node[right]{$ -q$} (B2);
    \draw [thick,->] (MR) -- node[left]{$ 1$} (B3);
    \draw [thick,->] (MR) -- node[right]{$ -q$} (B4);
  \draw [thick,->] (B1) --  (-3,0);
  \draw [thick,->] (B2) --  (-1,0);
  \draw [thick,->] (B3) --  (1,0);
  \draw [thick,->] (B4) --  (3,0);
  \node at (-3,-.5) {$\scriptstyle  (q+q^{-1})^2$};
  \node at (-2,-.5) {$\scriptstyle -$};
  \node at (-1,-.5) {$\scriptstyle q(q+q^{-1})$};
    \node at (0,-.5) {$\scriptstyle -$};
  \node at (1,-.5) {$\scriptstyle q(q+q^{-1})$};
    \node at (2,-.5) {$\scriptstyle +$};
  \node at (3,-.5) {$\scriptstyle q^2(q+q^{-1})^2$};
      \node at (-.5,-1) {$\scriptstyle= $};
      \node at (1.25,-1) {$\scriptstyle q^4 + q^2+1+q^{-2} $};
\end{tikzpicture}
\begin{tikzpicture}

 \node (T) at (0,5) {
    \begin{tikzpicture}[scale=.4]
  \draw [black, very thick]    (2,1) .. controls +(0,.25) and +(0,-.25) .. (1,2);
  \draw [black, very thick]    (2,2) .. controls +(0,.25) and +(0,-.25) .. (1,3);
  \path [fill=white] (1.35,1) rectangle (1.65,3);
  \draw [black, very thick]    (1,1) .. controls +(0,.25) and +(0,-.25) .. (2,2);
  \draw [black, very thick]    (1,2) .. controls +(0,.25) and +(0,-.25) .. (2,3);
  \draw [black, very thick]    (1,3) .. controls +(-.1,.5) and +(+.1,.5) .. (0,3);
  \draw [black, very thick]    (2,3) .. controls +(+.1,.5) and +(-.1,.5) .. (3,3);
  \draw [black, very thick]    (0,3) -- (0,1);
  \draw [black, very thick]    (3,3) -- (3,1);
  \draw [black, very thick]    (1,1) .. controls +(-0.1,-0.5) and+(+0.1,-0.5) .. (0,1);
  \draw [black, very thick]    (2,1) .. controls +(0.1,-0.5) and +(-0.1,-0.5) .. (3,1);
 \draw [thick, red, dashed] (1.5,2.5) circle [radius=0.6];
    \end{tikzpicture} };
\node (MR) at (2,3) {
    \begin{tikzpicture}[scale=.4]
  \draw [black, very thick]    (2,1) .. controls +(0,.25) and +(0,-.25) .. (1,2);
  \path [fill=white] (1.35,1) rectangle (1.65,2);
  \draw [black, very thick]    (1,1) .. controls +(0,.25) and +(0,-.25) .. (2,2);
  \draw [black, very thick]    (1,2) .. controls +(0,.35) and +(0,.35) .. (2,2);
  \draw [black, very thick]    (2,3) .. controls +(0,-.35) and +(0,-.35) .. (1,3);
  \draw [black, very thick]    (1,3) .. controls +(-.1,.5) and +(+.1,.5) .. (0,3);
  \draw [black, very thick]    (2,3) .. controls +(+.1,.5) and +(-.1,.5) .. (3,3);
  \draw [black, very thick]    (0,3) -- (0,1);
  \draw [black, very thick]    (3,3) -- (3,1);
  \draw [black, very thick]    (1,1) .. controls +(-0.1,-0.5) and+(+0.1,-0.5) .. (0,1);
  \draw [black, very thick]    (2,1) .. controls +(0.1,-0.5) and +(-0.1,-0.5) .. (3,1);
 \draw [thick, red, dashed] (1.5,1.5) circle [radius=0.6];
    \end{tikzpicture} };
\node (ML) at (-2,3) {
        \begin{tikzpicture}[scale=.4]
  \draw [black, very thick]    (1,2) -- (1,3);
  \draw [black, very thick]    (2,2) -- (2,3);
  \draw [black, very thick]    (2,1) .. controls +(0,.25) and +(0,-.25) .. (1,2);
  \path [fill=white] (1.35,1) rectangle (1.65,2);
  \draw [black, very thick]    (1,1) .. controls +(0,.25) and +(0,-.25) .. (2,2);
  \draw [black, very thick]    (1,3) .. controls +(-.1,.5) and +(+.1,.5) .. (0,3);
  \draw [black, very thick]    (2,3) .. controls +(+.1,.5) and +(-.1,.5) .. (3,3);
  \draw [black, very thick]    (0,3) -- (0,1);
  \draw [black, very thick]    (3,3) -- (3,1);
  \draw [black, very thick]    (1,1) .. controls +(-0.1,-0.5) and+(+0.1,-0.5) .. (0,1);
  \draw [black, very thick]    (2,1) .. controls +(0.1,-0.5) and +(-0.1,-0.5) .. (3,1);
  \draw [thick, red, dashed] (1.5,1.5) circle [radius=0.6];
    \end{tikzpicture}   };
 \node (B4) at (3,1) {
        \begin{tikzpicture}[scale=.4]
  \draw [black, very thick]    (1,1) .. controls +(0,.35) and +(0,.35) .. (2,1);
  \draw [black, very thick]    (2,2) .. controls +(0,-.35) and +(0,-.35) .. (1,2);
  \draw [black, very thick]    (1,2) .. controls +(0,.35) and +(0,.35) .. (2,2);
  \draw [black, very thick]    (2,3) .. controls +(0,-.35) and +(0,-.35) .. (1,3);
  \draw [black, very thick]    (1,3) .. controls +(-.1,.5) and +(+.1,.5) .. (0,3);
  \draw [black, very thick]    (2,3) .. controls +(+.1,.5) and +(-.1,.5) .. (3,3);
  \draw [black, very thick]    (0,3) -- (0,1);
  \draw [black, very thick]    (3,3) -- (3,1);
  \draw [black, very thick]    (1,1) .. controls +(-0.1,-0.5) and+(+0.1,-0.5) .. (0,1);
  \draw [black, very thick]    (2,1) .. controls +(0.1,-0.5) and +(-0.1,-0.5) .. (3,1);
    \end{tikzpicture}};
 \node (B3) at ( 1,1) {
        \begin{tikzpicture}[scale=.4]
  \draw [black, very thick]    (1,2) -- (1,1);
  \draw [black, very thick]    (2,2) -- (2,1);
  \draw [black, very thick]    (1,2) .. controls +(0,.35) and +(0,.35) .. (2,2);
  \draw [black, very thick]    (2,3) .. controls +(0,-.35) and +(0,-.35) .. (1,3);
  \draw [black, very thick]    (1,3) .. controls +(-.1,.5) and +(+.1,.5) .. (0,3);
  \draw [black, very thick]    (2,3) .. controls +(+.1,.5) and +(-.1,.5) .. (3,3);
  \draw [black, very thick]    (0,3) -- (0,1);
  \draw [black, very thick]    (3,3) -- (3,1);
  \draw [black, very thick]    (1,1) .. controls +(-0.1,-0.5) and+(+0.1,-0.5) .. (0,1);
  \draw [black, very thick]    (2,1) .. controls +(0.1,-0.5) and +(-0.1,-0.5) .. (3,1);
        \end{tikzpicture}};
 \node (B2) at (-1,1) {
        \begin{tikzpicture}[scale=.4]
  \draw [black, very thick]    (1,1) .. controls +(0,.35) and +(0,.35) .. (2,1);
  \draw [black, very thick]    (2,2) .. controls +(0,-.35) and +(0,-.35) .. (1,2);
  \draw [black, very thick]    (1,2) -- (1,3);
  \draw [black, very thick]    (2,2) -- (2,3);
  \draw [black, very thick]    (1,3) .. controls +(-.1,.5) and +(+.1,.5) .. (0,3);
  \draw [black, very thick]    (2,3) .. controls +(+.1,.5) and +(-.1,.5) .. (3,3);
  \draw [black, very thick]    (0,3) -- (0,1);
  \draw [black, very thick]    (3,3) -- (3,1);
  \draw [black, very thick]    (1,1) .. controls +(-0.1,-0.5) and+(+0.1,-0.5) .. (0,1);
  \draw [black, very thick]    (2,1) .. controls +(0.1,-0.5) and +(-0.1,-0.5) .. (3,1);
    \end{tikzpicture}};
  \node (B1) at (-3,1) {\begin{tikzpicture}[scale=.4]
  \draw [black, very thick]    (1,1) -- (1,3);
  \draw [black, very thick]    (2,1) -- (2,3);
  \draw [black, very thick]    (1,3) .. controls +(-.1,.5) and +(+.1,.5) .. (0,3);
  \draw [black, very thick]    (2,3) .. controls +(+.1,.5) and +(-.1,.5) .. (3,3);
  \draw [black, very thick]    (0,3) -- (0,1);
  \draw [black, very thick]    (3,3) -- (3,1);
  \draw [black, very thick]    (1,1) .. controls +(-0.1,-0.5) and+(+0.1,-0.5) .. (0,1);
  \draw [black, very thick]    (2,1) .. controls +(0.1,-0.5) and +(-0.1,-0.5) .. (3,1);
    \end{tikzpicture}};
    \draw [thick,->] (T) -- node[above]{$0\;$} (ML);
    \draw [thick,->] (T) -- node[above]{$ 1$} (MR);
    \draw [thick,->] (ML) -- node[left]{$ 0$} (B1);
    \draw [thick,->] (ML) -- node[right]{$ 1$} (B2);
    \draw [thick,->] (MR) -- node[left]{$ 0$} (B3);
    \draw [thick,->] (MR) -- node[right]{$ 1$} (B4);
  \draw [thick,->] (B1) --  (-3,0);
  \draw [thick,->] (B2) --  (-1,0);
  \draw [thick,->] (B3) --  (1,0);
  \draw [thick,->] (B4) --  (3,0);
\node at (-3.9,-.5) {$\ket{r}$:};
  \node at (-3,-.5) {$\ket{00}$};
  \node at (-1,-.5) {$\ket{01}$};
  \node at (1,-.5) {$ \ket{10}$};
  \node at (3.2,-.5) {$\ket{11}$};
  \node at (-3.9,-1.2) {$\ell(r):$};
  \node at (-3,-1.2) {$2$};
  \node at (-1,-1.2) {$1$};
  \node at (1,-1.2) {$1$};
  \node at (3.2,-1.2) {$2$};
\end{tikzpicture}

\caption{ Left: Implementing the recursive algorithm that produces the Kauffmann bracket polynomial on a Hopf link with $m=2$ crossings. At each step, a crossing (circled red) is replaced by two possible smoothings. The final step produces a collection of unknotted loops, each contributing a factor of $q+q^{-1}$ to the polynomial. Right: The resulting states $\ket{r}$ and number of loops $\ell(r)$ associated with the resolution $r$. Here we enumerate the crossings from the top of the diagram down.}
\label{fig:Kauffman1}
\end{figure}

The Kauffman bracket is not a knot invariant as defined since it is not invariant under all Reidemeister moves.  However, adding an overall normalization corrects this so that 
\begin{equation}\label{eq:Jones-shift}
     J(K) := (-1)^{n_-} q^{n_+-2n_-} \langle K \rangle 
\end{equation}
is a well defined invariant, called the \emph{Jones polynomial} of $K$. Here, $n_+$ is the number of positive crossings, $n_-$ is the number of negative crossings, both of which are obtained from an orientation of the knot using the following convention for positive and negative crossings:
\begin{equation}
\hackcenter{
\begin{tikzpicture}
\draw[black,very thick,-> ] (-.5,0) -- (.5,1.5);
\draw[black,very thick ] (.5,0) -- (.1,.6);
\draw[black,very thick, -> ] (-.1,.9) -- (-.5,1.5);
 \node at (0,.1) {$+$};
\end{tikzpicture}}
\qquad  \qquad 
\hackcenter{
\begin{tikzpicture}
\draw[black,very thick,-> ] (.5,0) -- (-.5,1.5);
\draw[black,very thick ] (-.5,0) -- (-.1,.6);
\draw[black,very thick, -> ] (.1,.9) -- (.5,1.5);
 \node at (0,.1) {$-$};
\end{tikzpicture}}.  
\end{equation}

\subsection{Chain complexes and homology} \label{sec:chain}

Our focus in this article is quantum mechanical algorithms to compute Betti numbers of (co)homology theories relating to the Jones polynomial.  
Recall that a {\em chain complex} of complex vector spaces is a collection of vector spaces $\mathbf{C} = \{ C_i \mid i \in \mathbb{Z}\}$ together with maps $\partial_i \maps C_i \to C_{i-1}$ called {\em differentials} satisfying $\partial_i \circ \partial_{i+1} = 0$. 
A {\em cochain complex} of objects $\mathbf{C} = \{ C^i \mid i \in \mathbb{Z}\}$ and linear maps $d_i \maps C^i \to C^{i+1}$ also called differentials, satisfying $d_{i} \circ d_{i-1} =0$.

Given a chain complex $(\mathbf{C}, \partial)$, define the $i$th {\em homology groups} $H_i(\mathbf{C})$ and $i$th Betti number $\beta_i$   by
\[
H_i(\mathbf{C}):= \ker(\partial_i)/ \Im (\partial_{i+1}), \qquad \beta_i := \dim H_i (\mathbf{C}) .
\]
Likewise, given a cochain complex $(\mathbf{C}, d)$ define the $ith$ cohomology $H^i(C)$ and $i$th Betti number $\beta^i$   by
\[
H^i(\mathbf{C}):= \ker(d_i)/ \Im (d_{i-1}), \qquad \beta^i := \dim H^i (\mathbf{C}).
\]

Maps between complexes that induce maps between homology groups are called {\em chain maps} (resp. {\em cochain maps}).  Such maps $f \maps \mathbf{C} \to \mathbf{D}$ consist of maps $f_i \maps C_i \to D_i$ (resp. $f_n \maps C^i \to D^i$ for $i \in \mathbb{Z}$ satisfying $ f_{i-1} \circ \partial_i = \partial_i \circ f_i$ (resp. $f_{i+1}d_i = d_i f_i$). They induce well defined maps $f_{\ast} \maps H_i(\mathbf{C}) \to H_i(\mathbf{D})$, resp. $f^{\ast} \maps H^i(\mathbf{C}) \to H^i(\mathbf{D})$, on (co)homology. 

We will be primarily interested in bounded (co)chain complexes where $C_i =0$, respectively $C^i =0$, for all but finitely many values of $i \in \mathbb{Z}$.   We also assume that all vector spaces $C_i$, resp. $C^i$, are finite-dimensional. 
In what follows, we follow the tradition in homological algebra and refer to cochain complexes simply as complexes, and to cohomology as homology of the complex when no confusion is likely to arise. 

\begin{remark}
From a mathematical perspective, there is little difference between a chain complex and a cochain complex since any chain complex $(C_i, \partial_i)$ defines a cochain complex by setting $C^{i}= C_{-i}$ and $d_i = \partial_{-i}$, and vice-versa. Since our primary interest is Khovanov homology, which has traditionally been presented as a cohomology theory where the differential increases the index of chain spaces, we present this material using complexes of this type.   
\end{remark}


\subsection{Categorification} \label{subsec:Categorification}

Categorification was a concept introduced by Crane and Frenkel in their study of algebraic structures in Topological Quantum Field Theory (TQFT)~\cite{Crane-Frenkel}.  By examining the structures needed for TQFTs in various dimensions, they observed an increase in complexity where algebraic objects formalized in the language of sets or vector spaces become enhanced to similar structures built from categories.    Categories are much like sets, except that instead of just having elements, categories have two levels of structure: objects and morphisms.  It is precisely this higher level of structure that gives categorification its power.  
\[\begin{array}{c}
 {\bf Set} \\
  \text{equations $x = y$}
\end{array} \quad
\vcenter{\xy (0,0)*{
 \includegraphics[width=2.2cm]{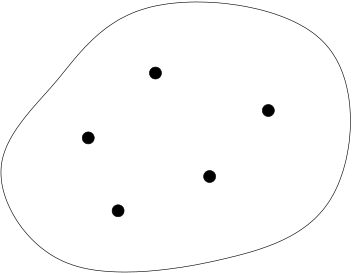}}; \endxy}
  \qquad\qquad 
\begin{array}{c}
{\bf Categories} \\
  \text{isomorphisms $x \cong y$}
\end{array} \quad 
\xy (0,0)*{
 \includegraphics[width=2.2cm]{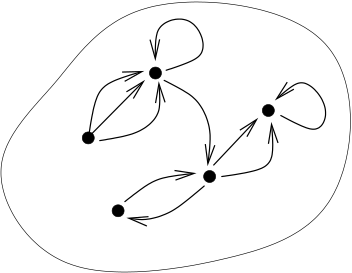}}; \endxy
\]
Equalities are lifted to explicit isomorphisms built from the new higher structure of categories.  The morphisms at the categorical level are a structure not seen at the `decategorified' or set level, but this new level of structure allows for far greater descriptive power.   This enhancement to categorical structures is where the term \emph{categorification} came from.   

After Crane and Frenkel's work, it became clear that categorification was prevalent throughout mathematics, with many well-known constructions naturally fitting into this framework, see ~\cite{Crane-Frenkel,Baez,KMS,Lauda2022AnIT}.  Perhaps the best example to illustrate the key features of categorification might come from thinking about invariants of polyhedra.  It is well known that the Euler characteristic   
of a 3-dimensional polyhedron is a topological invariant.  The quantity $\chi = (\#\text{vertices} - \#\text{edges}+\#\text{faces})$ remains the same for any polyhedra with the same underlying topology.  
\[
\xy
(0,0)*{
\begin{tikzpicture}
    \fill[gray!30] (0,0,0) -- (1,0,0) -- (1,1,0) -- (0,1,0) -- cycle;
    \fill[gray!30] (0,0,0) -- (0,0,1) -- (0,1,1) -- (0,1,0) -- cycle;
    \fill[gray!30] (0,0,0) -- (1,0,0) -- (1,0,1) -- (0,0,1) -- cycle;
    \fill[gray!30] (1,1,1) -- (0,1,1) -- (0,1,0) -- (1,1,0) -- cycle;
    \fill[gray!30] (1,1,1) -- (1,0,1) -- (1,0,0) -- (1,1,0) -- cycle;
    \fill[gray!30] (1,1,1) -- (1,0,1) -- (0,0,1) -- (0,1,1) -- cycle;
    \draw[dotted] (0,0,0) -- (1,0,0);
    \draw[dotted] (0,0,0) -- (0,0,1);
    \draw(1,0,0) -- (1,0,1);
    \draw[dotted] (0,1,0) -- (0,0,0);
    \draw (0,1,0) -- (0,1,1);
    \draw (1,0,1) -- (1,1,1);
    \draw (0,0,1) -- (1,0,1);
    \draw (0,0,1) -- (0,1,1);
    \draw (1,0,0) -- (1,1,0);
    \draw (1,1,0) -- (0,1,0);
    \draw (1,1,1) -- (1,1,0);
    \draw (1,1,1) -- (0,1,1);
    \foreach \x in {0,1}
        \foreach \y in {0,1}
            \foreach \z in {0,1}
                \fill[black] (\x,\y,\z) circle (2pt);
\end{tikzpicture}}; 
(0,-15)*{8-12+6=2};  
\endxy
\qquad  \qquad  \qquad 
\xy 
(0,0)*{
\begin{tikzpicture}[scale=1]
    \fill[gray!30] (1,1,0) -- (2,0,0) -- (0,0,0) -- cycle;
    \fill[gray!30] (1,1,0) -- (2,0,0) -- (1,0,2) -- cycle;
    \fill[gray!30] (1,1,0) -- (0,0,0) -- (1,0,2) -- cycle;
    \fill[gray!30] (0,0,0) -- (2,0,0) -- (1,0,2) -- cycle;
    \draw (1,1,0) -- (2,0,0);
    \draw[dotted] (2,0,0) -- (0,0,0);
    \draw (0,0,0) -- (1,1,0);
    \draw (1,1,0) -- (1,0,2);
    \draw (2,0,0) -- (1,0,2);
    \draw (0,0,0) -- (1,0,2); 
    \draw[dashed] (1,1,0) -- (1,0,2);
    \draw[dashed] (2,0,0) -- (1,0,2);
    \draw[dashed] (0,0,0) -- (1,0,2);
    \fill[black] (1,1,0) circle (2pt);
    \fill[black] (2,0,0) circle (2pt);
    \fill[black] (0,0,0) circle (2pt);
    \fill[black] (1,0,2) circle (2pt);
\end{tikzpicture} 
};
(0,-15)*{4-6+6=2};  
\endxy
\]
Later, of course, it was realized that there is a far more powerful set of invariants that not only apply to polyhedra but any sufficiently nice topological space $X$.  These are the (singular) homology groups, a collection of vector spaces\footnote{More generally abelian groups.} $H_i(X)$ associated to $X$.  We can forget information and simplify the vector spaces $H_i(X)$ into numerical quantities $\beta_i = \dim H_i(X)$ called the $i$-th Betti numbers.  The Euler characteristic is then the alternating sum of the Betti numbers.  

We say that homology groups categorify the Euler characteristic in that we have a well-defined procedure for forgetting information and producing a number.  The homology groups are a collection of vector spaces with the property that
\[
 \chi(X) = \sum_{i \geq 0} (-1)^i \dim H_i(X). 
\]
Finding a collection of vector spaces with this property on its own is not remarkable.  However, homology has a number of key advantages.  
\begin{itemize}
\item Homology is a more powerful invariant:  each $\beta_i$ is an invariant of $X$.  More spaces can be distinguished using homology than using Euler characteristic.    
\item Homology groups provide new insights into the meaning of the numeric quantities used to compute the Euler characteristic.  Betti numbers are dimensions of homology groups which themselves come from some topologically meaningful construction.  
\item Homology is functorial:  given a map $f\maps X \to Y$ between topological spaces, we get a map between the respective homologies $f_* \maps H_i(X) \to H_i(Y)$.  This has important consequences for computing homology and provides an explanation for how continuous maps behave on Euler characteristics.   This level could not have been seen if we had not lifted the Euler characteristic from a numerical invariant to the category of vector spaces and linear maps. 
\end{itemize}

\begin{example}
The Euler characteristic of the 2-sphere $S^2$ is categorified by its homology groups $H_i(S^2)$:
\[
\xy
(0,0)*{
\begin{tikzpicture}[scale=.5]
\node at (.75,0) {$\begin{tikzpicture} [fill opacity=0.2, scale=.45]
	\path [fill=black] (0,0) circle (1);
	\draw (-1,0) .. controls (-1,-.4) and (1,-.4) .. (1,0);
	\draw[dashed] (-1,0) .. controls (-1,.4) and (1,.4) .. (1,0);
	\draw[very thick] (0,0) circle (1);
	\end{tikzpicture}$};
\draw[|->] (2,.5) to (4.7,1.5);
\draw[|->] (2,-.5) to (5,-1.5);
\draw[|->, very thick] (7,.7) to (7,-.7);
\node at (9.5,1.5) {$\mathrm{H}_{\bullet}\left(\sphere[.3]\right) \cong \Z \oplus 0 \oplus \Z$};
\node at (8,-1.5) {$\chi\left(\sphere[.3]\right)$ = 2};
\draw[thick, |->] (11,-1.5) to [out=20,in=-90] (12,1);
  \path [fill=white] (10.35,-.5) rectangle (13.65,.5);
\node at (12,0) {$\text{Categorification}$};
\end{tikzpicture}};
\endxy
\]
\[\chi(S^2) = \sum_{i=0}^{\infty} (-1)^i \dim H_i(S^2) = 1 - 0 +1 =2.\]
\end{example}

In this article, we will be interested in another kind of homology that categorifies the Jones polynomial in a similar sense. To motivate its construction, we start by considering the categorification of natural numbers.

At the most primitive level, a natural categorification of the set $\Z_{\geq 0}$ is the category $\text{Vect}_\mathbb{K}$ of finite-dimensional vector spaces over a field $\mathbb K$. A vector space $V$ is associated with its dimension $n = \operatorname{dim} V$, for which addition and multiplication on $\Z_{\geq 0}$ extend to $\text{Vect}_\mathbb{K}$ via the rules
\begin{equation}
 \operatorname{dim}(V \oplus W)=\operatorname{dim} V+\operatorname{dim} W, \qquad 
 \operatorname{dim}(V \otimes W)=\operatorname{dim} V \cdot \operatorname{dim} W.
\end{equation}
This means all the original structures of $\Z_{\geq 0}$ can be seen as decategorifications of operations on $\text{Vect}_\mathbb{K}$ that are simplified by applying the dimension map. 

At this point, it is again important to note that while any natural number $n$ can be categorified by $\mathbb{K}^n$, this would be an extremely naive categorification that would not produce any of the desired properties we observe in the example of homology groups and the Euler characteristic.  Nevertheless, this is an excellent starting point for understanding the structure that would be needed to categorify the Jones polynomial.

The Jones polynomial is not just a numerical invariant; it is a Laurent polynomial in $\mathbb{Z}[q,q^{-1}]$, which can be thought of as a sequence of numerical invariants corresponding to the coefficient of powers of $q$ appearing in $J(K)$. A Laurent polynomial $f \in$ $\mathbb{Z}_{\geq 0}\left[q, q^{-1}\right]$ with non-negative coefficients can be categorified by a $\mathbb{Z}$-graded vector space with decategorification corresponding to the \emph{graded} dimension $\operatorname{gdim} V$:
\begin{equation}
V=\bigoplus_{n \in \mathbb{Z}} V_n, \quad \operatorname{gdim} V:=\sum_{n \in \mathbb{Z}} q^n \operatorname{dim} V_n .
\end{equation}
In this context, the exponent of the variable $q$ encodes the grading on the vector space $V$. 
Hence, Laurent polynomials with non-negative coefficients are categorified by the category of finite-dimensional $\mathbb{Z}$-graded vector spaces, and decategorification is taking the graded dimension.

The above constructions are not sufficient to categorify polynomials with negative coefficients, such as the Jones polynomial. To introduce the needed minus signs within a categorical setting,  it is natural to pass to chain complexes of graded vector spaces
\begin{equation}
V^{\cdot}=\cdots \longrightarrow V^i \xrightarrow{d} V^{i+1} \xrightarrow{d} V^{i+2} \longrightarrow \cdots ,
\end{equation}
consisting of a sequence of graded vector spaces $V^i$ and maps $d \maps V^i \to V^{i+1}$ satisfying $d^2=0$.   Such chain complexes give rise to homology groups as explained in section~\ref{sec:chain}.  In this context, decategorification is given by taking  the \emph{graded Euler characteristic}
\begin{equation} \label{eq:gradedEuler}
    \chi\left(V^{\cdot}\right):=\sum_{i \in \mathbb{Z}}(-1)^i \operatorname{gdim} V^i  
\end{equation}
where the minus signs appear from odd homological degrees.   

Hence, to categorify the Jones polynomial, we need to construct a chain complex whose graded Euler characteristic recovers the Jones polynomial in such a way that the categorification has similar desired properties as the homology and Euler characteristic of topological spaces.  Naively taking $\mathbb{K}^n$'s with trivial differential would not produce an enhanced invariant or have nice functoriality properties.     For more on the interpretation of (graded) Euler characteristics in the broader context of decategorification, see \cite{Beliakova}.
 
\subsection{Khovanov Homology}
\label{sec:bg_khovanov}
 
Khovanov homology \cite{Kh1} is precisely a homology theory whose Euler characteristic is the Jones polynomial of a knot or link $K$. Each homology group $Kh^i(K)$ is a graded vector space $Kh^i(K)=\bigoplus_{j}Kh^{i,j}(K)$ and taking the graded Euler characteristic by the alternating sum of the graded dimensions of these groups, as in the previous section, gives the Jones polynomial,
\[
\mathrm{J}(K) = \sum_{i,j} (-1)^i q^j \dim( \mathrm{Kh}^{i,j}(K) ) .
\]
Each group $Kh^{i,j}(K)$ is a topological invariant of $K$.
\subsubsection{Why categorify? } \label{subsubsec-Why}
The power of Khovanov homology comes from the fact that it contains more information than the Jones polynomial.   Like the example of singular homology groups and Euler characteristics in the previous section, Khovanov homology has key advantages over the Jones polynomial.  

\begin{itemize}
    \item Khovanov homology groups $Kh^{i,j}(K)$ are strictly stronger knot invariants than $J(K)$.  Many knots with the same Jones polynomial have different Khovanov homologies~\cite{BN1}. 

\item Khovanov homology detects the unknot~\cite{KM}  and more generally the $n$- component unlink~\cite{Hedden}, trefoils~\cite{MR4393789}, and the Hopf link~\cite{MR4049809}, and other specific knots~\cite{MR4275096}.  Analogous results for the Jones polynomial are not known. 
    
    \item Khovanov homology is functorial:  functoriality in the context of knots and links requires a higher degree of topological sophistication.   In this context, a map between two knots $K_1$ and $K_2$ is a surface $\Sigma$ embedded in the unit interval cross $\R^3$ whose boundary consists of the two knots. Such an embedded surface is called a \emph{knot cobordism}
    \[
\includegraphics{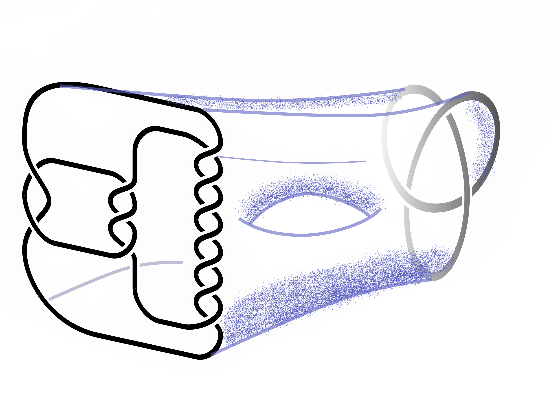}
\] 
    
    While it is difficult to draw meaningful illustrations\footnote{Thanks to Peter Kronheimer for the use of this knot cobordism illustration.}, 
    knot cobordisms belong to a part of knotted structures living in 4-dimensions.    Khovanov homology is functorial in the sense that to any  link cobordism $K_1 \xrightarrow{\Sigma} K2$ it associates a map on homology:
\[
\xy
(0,0)*{
\begin{tikzpicture} [rotate=90, scale=.4]
\node at (0,5) {$
\xy
(0,0)*{
\begin{tikzpicture} [scale=.35]
\draw[very thick] (0,0) to [out=90,in=270] (.5,1);
\draw[very thick] (.2,.6) to [out=135,in=270] (0,1);
\draw[very thick] (.5,0) to [out=90,in=315] (.3,.4);
\draw[very thick] (0,-1) to [out=90,in=270] (.5,0);
\draw[very thick] (.2,-.4) to [out=135,in=270] (0,0);
\draw[very thick] (.5,-1) to [out=90,in=315] (.3,-.6);
\draw[very thick] (0,-2) to [out=90,in=270] (.5,-1);
\draw[very thick] (.2,-1.4) to [out=135,in=270] (0,-1);
\draw[very thick] (.5,-2) to [out=90,in=315] (.3,-1.6);
\draw[very thick, directed=.5] (.5,1) to [out=90,in=180] (.75,1.4) to [out=0,in=90] (1,1) to
	(1,-2) to [out=270,in=0] (.75,-2.4) to [out=180,in=270] (.5,-2);
\draw[very thick, directed=.5] (0,1) to [out=90,in=0] (-.25,1.4) to [out=180,in=90] (-.5,1) to
	(-.5,-2) to [out=270,in=180] (-.25,-2.4) to [out=0,in=270] (0,-2);
\end{tikzpicture}
}
\endxy
$};
\node at (0,0) {$
\xy
(0,0)*{
\begin{tikzpicture} [scale=.35]
\draw[very thick] (0,0) to [out=90,in=270] (.5,1);
\draw[very thick] (.2,.6) to [out=135,in=270] (0,1);
\draw[very thick] (.5,0) to [out=90,in=315] (.3,.4);
\draw[very thick] (0,-1) to [out=90,in=270] (.5,0);
\draw[very thick] (.2,-.4) to [out=135,in=270] (0,0);
\draw[very thick] (.5,-1) to [out=90,in=315] (.3,-.6);
\draw[very thick, directed=.5] (.5,1) to [out=90,in=180] (.75,1.4) to [out=0,in=90] (1,1) to
	(1,-1) to [out=270,in=0] (.75,-1.4) to [out=180,in=270] (.5,-1);
\draw[very thick, directed=.5] (0,1) to [out=90,in=0] (-.25,1.4) to [out=180,in=90] (-.5,1) to
	(-.5,-1) to [out=270,in=180] (-.25,-1.4) to [out=0,in=270] (0,-1);
\end{tikzpicture}
}
\endxy
$};
	\draw[thick] (-1.5,4.375) to [out=225,in=90] (-1.25,2.5) to [out=270,in=135] (-1.25,.625);
	\draw[thick] (1.5,4.375) to [out=315,in=90] (1.25,2.5) to [out=270,in=45] (1.25,.625);
	\fill[red,opacity=.2] (-1.5,4.375) to [out=225,in=90] (-1.25,2.5) to [out=270,in=135] (-1.25,.625) to
		(1.25,.625) to [out=45,in=270] (1.25,2.5) to [out=90,in=315] (1.5,4.375);
	\fill[white] (0,3) to [out=300,in=60] (0,2) to [out=120,in=240] (0,3);
	\draw[thick] (0,3) to [out=300,in=60] (0,2);
	\draw[thick] (0.1,1.8) to [out=125,in=235] (0.1,3.2);
\end{tikzpicture}
};
\endxy
\quad
\leadsto
\quad
\mathrm{Kh}\left( \;
\xy
(0,0)*{
\begin{tikzpicture} [scale=.35]
\draw[very thick] (0,0) to [out=90,in=270] (.5,1);
\draw[very thick] (.2,.6) to [out=135,in=270] (0,1);
\draw[very thick] (.5,0) to [out=90,in=315] (.3,.4);
\draw[very thick] (0,-1) to [out=90,in=270] (.5,0);
\draw[very thick] (.2,-.4) to [out=135,in=270] (0,0);
\draw[very thick] (.5,-1) to [out=90,in=315] (.3,-.6);
\draw[very thick] (0,-2) to [out=90,in=270] (.5,-1);
\draw[very thick] (.2,-1.4) to [out=135,in=270] (0,-1);
\draw[very thick] (.5,-2) to [out=90,in=315] (.3,-1.6);
\draw[very thick, directed=.5] (.5,1) to [out=90,in=180] (.75,1.4) to [out=0,in=90] (1,1) to
	(1,-2) to [out=270,in=0] (.75,-2.4) to [out=180,in=270] (.5,-2);
\draw[very thick, directed=.5] (0,1) to [out=90,in=0] (-.25,1.4) to [out=180,in=90] (-.5,1) to
	(-.5,-2) to [out=270,in=180] (-.25,-2.4) to [out=0,in=270] (0,-2);
\end{tikzpicture}
}
\endxy
\; \right)
\xrightarrow{\mathrm{Kh}(\Sigma)}
\mathrm{Kh}\left( \;
\xy
(0,0)*{
\begin{tikzpicture} [scale=.35]
\draw[very thick] (0,0) to [out=90,in=270] (.5,1);
\draw[very thick] (.2,.6) to [out=135,in=270] (0,1);
\draw[very thick] (.5,0) to [out=90,in=315] (.3,.4);
\draw[very thick] (0,-1) to [out=90,in=270] (.5,0);
\draw[very thick] (.2,-.4) to [out=135,in=270] (0,0);
\draw[very thick] (.5,-1) to [out=90,in=315] (.3,-.6);
\draw[very thick, directed=.5] (.5,1) to [out=90,in=180] (.75,1.4) to [out=0,in=90] (1,1) to
	(1,-1) to [out=270,in=0] (.75,-1.4) to [out=180,in=270] (.5,-1);
\draw[very thick, directed=.5] (0,1) to [out=90,in=0] (-.25,1.4) to [out=180,in=90] (-.5,1) to
	(-.5,-1) to [out=270,in=180] (-.25,-1.4) to [out=0,in=270] (0,-1);
\end{tikzpicture}
}
\endxy
\; \right)
\]
\end{itemize}

\begin{remark}
The functoriality of Khovanov homology encodes deep structures that can probe the smooth topology of 4-manifolds~\cite{Rasmussen}.  These structures have been used to prove results previously only accessible using gauge theory~\cite{Rasmussen}, as well as provide new constructions of exotic smooth 4-manifolds~\cite{Ren-WIllis}.  
\end{remark}

\subsubsection{Lifting the Kauffman bracket} \label{subsec:lifting-Kauffman}

We now describe the construction of Khovanov homology given a planar representation of a knot $K$.  We first lift the Kauffman bracket $\langle K \rangle$ to the \emph{Khovanov bracket}  $\llangle K  \rrangle$.  The Khovanov bracket is a chain complex whose homology is the Khovanov homology up to some overall shifts explained in \Cref{subsec:Kh-bracket-homology}.  For more details see ~\cite{Kh1,BN1}.

Fix an ordering of the $m$ crossings in the knot diagram for $K$.  
By choosing at each of the $m$ crossings of $K$ either a 0-smoothing $\langle\KPC\rangle$ or 1-smoothing $\langle\KPD\rangle$ (cf. \Cref{sec:Jones_polynomial}), we arrive at a crossingless diagram which we call a resolution (or state), denoted by a $m$-bit string $r$ whose $j$-th bit specifies which smoothing was applied to the crossing labeled $j$.  The Hamming weight of the state $\ket{r}$ is denoted $|r|$ and equals the number of 1-smoothings in the resolution.

Each resolution ${r}$ is associated with a number $\ell(r)$ of disjoint loops resulting from the resolution, see the right-hand side of Figure~\ref{fig:Kauffman1}. In this notation, the Kauffmann bracket can be written in its ``state sum'' representation 
\begin{equation} \label{eq:Kauff}
    \langle K \rangle = \sum_{r\in \{0,1\}^m}(-q)^{|r|}(q+q^{-1})^{\ell(r)}.
\end{equation}
It is easy to verify that this expression agrees with the recursive definition in \Cref{sec:Jones_polynomial}.

The central idea of Khovanov homology is to take the alternating sum in \cref{eq:Kauff} and interpret it as an Euler characteristic of some homology theory.  To account for the $q$ factors that appear, our homology theory will carry an additional grading as in \cref{eq:gradedEuler}, and the Kauffman bracket will be categorified by a graded homology theory whose graded Euler characteristic recovers the Kauffman bracket. 
 
Since a loop gets replaced by $(q+q^{-1})$ in the Kauffman bracket,  to construct the Khovanov bracket $\llangle K \rrangle$ this Laurent polynomial associated with a loop is categorified to a graded vector space $V=\mathbb{C}_{-1} \oplus \mathbb{C}_{+1}$ that is one-dimensional in degrees $\pm 1$ and zero in all other degrees, so that $\operatorname{gdim}V = (q+q^{-1})$.  We fix a graded basis for $V$ that we denote by  $\1$ and $X$, with $\operatorname{gdim}(\1)=1$ and $\operatorname{gdim}(X)=-1$ following the conventions from \cite{Kh1}
and think of these as labels on the loop.  

More generally, we can categorify all of the terms in \cref{eq:Kauff}.  The space $V^{\otimes \ell(r)}$ will have graded dimension $(q+q^{-1})^{\ell(r)}$.   Define the $i$th chain space $\llangle K\rrangle^i$ of the Khovanov bracket $\llangle K\rrangle$ chain complex    by setting 
\begin{equation} \label{eq:Kh-C}
    \llangle K\rrangle^i = \bigoplus_{ |r|=i} V^{\otimes \ell(r)} \langle i \rangle,
\end{equation}
where $\langle i \rangle$ is the grading shift operation that shifts the overall grading by $i$ such that
\begin{equation}
 \operatorname{gdim} V \langle i \rangle = q^i  \operatorname{gdim} V.
\end{equation}
The number $i$ is called the \emph{(co)homological degree} and counts how many $1$-smoothings appear in the resolution. Observe that with this grading shift, $\operatorname{gdim} \llangle K\rrangle^i = \sum_{|r|=i} q^{i}(q+q^{-1})^{\ell(r)}$.

Choosing a basis vector for each of the $\ell(r)$ loops appearing in a resolution $r$ of the knot $K$ assigns a label that is either $\1$ or $X$. The possible labeling of the loops are thus described by a $\ell(r)$-bit string $s$, where a 0 encodes the state $\1$, and a 1 encodes the state $X$. The tuples $(r,s)$ that specify both the choice of smoothings and the labels of the resulting loops are called \emph{enhanced states}.  
The Hamming weight $|s|$ denotes the number of loops in the state $X$, while $\ell(r)-|s|$ is the number of loops in state $\1$.    

Each enhanced state $(r,s)$ carries two gradings, the \emph{homological degree} $i:= |r|$ given by the number of 1-resolutions in $r$, and the \emph{quantum grading} $j$ associated with the labels in the enhanced state $s$.  The quantum grading is defined as $j= i  + \ell(r)-2|s|$, which comes from the overall shift by $i$, plus the number of loops labeled $\1$, and minus the number of loops labeled $X$.  Hence, each $\llangle K\rrangle^i$ further splits into $\llangle K\rrangle^i=\bigoplus_{j}\llangle K\rrangle^{i,j}$. 
 
The graded Euler characteristic is the alternating sum of the dimensions of the homology groups, but if the differential preserves the grading and all the chain groups are finite dimensional then one can show that the Euler characteristic is also the alternating sum of the graded dimensions of the chain spaces,
\[
 \chi(\llangle K\rrangle) = \sum_{i} (-1)^i {\rm gdim} H^i(\llangle K\rrangle) = \sum_i (-1)^i {\rm gdim} (\llangle K\rrangle^i), 
\] which now matches \cref{eq:Kauff}.
Hence, we could build a naive categorification of the  Kauffman bracket from our complex $\llangle K\rrangle$ with chain spaces as in \cref{eq:Kh-C} and trivial differentials.  
\[
 \chi(\llangle K\rrangle) 
 = \sum_i (-1)^i {\rm gdim} (\llangle K\rrangle^i) 
  = \sum_{i,j} (-1)^i q^j{\rm  dim} (\llangle K\rrangle^{i,j}) 
 = \langle K \rangle.
\]
However, this naive categorification would not have any of the properties described in \cref{subsubsec-Why} and would contain no more information than the Kauffman bracket.

\subsubsection{Khovanov boundary operator} \label{subsec:Kh-boundary-intro}

The nontrivial content of Khovanov homology is specified by the differential on the complex $\llangle K\rrangle$   associated to a knot. To specify the differential in Khovanov homology, we make use of the following maps on the vector space $V=\langle \1, X\rangle$ associated with loops in a resolution
\begin{alignat}{3} \label{eq:mdelta}
 m\maps V\otimes V&\to V   \qquad \qquad    && \delta \maps V \to V\otimes V 
 \\   \nonumber
 \ket{\1\1} &\mapsto \ket{\1} \qquad \qquad    && \quad\; \ket{\1} \mapsto \ket{\1X} + \ket{X\1}  
 \\ \nonumber
 \ket{\1X} &\mapsto \ket{X} \qquad \qquad    && \quad\; \ket{X} \mapsto \ket{XX}  
 \\ \nonumber
\ket{X\1}  &\mapsto \ket{X} 
\\ \nonumber
\ket{XX}  &\mapsto 0  \nonumber
\end{alignat}
With the gradings given by ${\rm gdim}(\1)=1$ and ${\rm gdim}(X)=-1$, one can check that both maps have degree -1, meaning they decrease the grading by one.  If we add an additional grading shift of $\langle1\rangle$ on the targets so that $m \maps V\otimes V \to V\langle1\rangle$ and $\delta\maps V \to V \otimes V\langle1\rangle$, then these maps will have degree zero, meaning they preserve the quantum degree $j$.

The differential $d \maps C^i \to C^{i+1}$ in the Khovanov bracket $\llangle K\rrangle$
decomposes as a sum of maps $d_{r,v}$ over pairs $(r,v)$ consisting of a resolution $r$ with $|r|=i$ and a resolution $v$  obtained from $r$ by swapping a single $0$ to a $1$ in $r$.  The number of loops in the resolutions for such pairs of $r$ and $v$ will always have $|\ell(r)-\ell(v)|=1$, and the map $d_{r,v}$ can be thought of as either merging two loops into one or splitting one loop into two.  See Figure~\ref{fig:Kh-Hopf} for a simple example, or Example~\ref{example:trefoil-zero} for a more complex illustration.

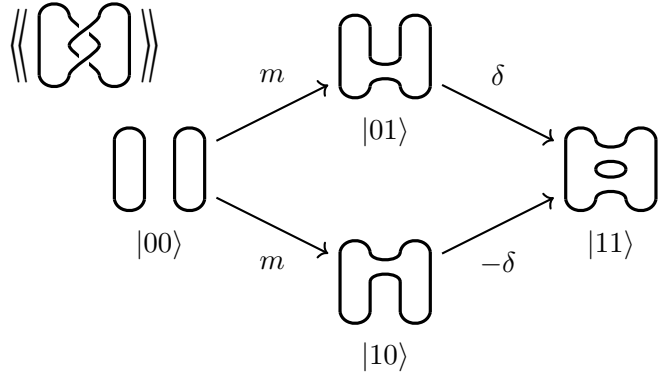
\begin{figure}
    \centering
$
\begin{tikzpicture}
 \node (T) at (-4,1.65) {  $\left\llangle \; \hackcenter{
    \begin{tikzpicture}[scale=.4]
  \draw [black, very thick]    (2,1) .. controls +(0,.25) and +(0,-.25) .. (1,2);
  \draw [black, very thick]    (2,2) .. controls +(0,.25) and +(0,-.25) .. (1,3);
  \path [fill=white] (1.35,1) rectangle (1.65,3);
  \draw [black, very thick]    (1,1) .. controls +(0,.25) and +(0,-.25) .. (2,2);
  \draw [black, very thick]    (1,2) .. controls +(0,.25) and +(0,-.25) .. (2,3);
  \draw [black, very thick]    (1,3) .. controls +(-.1,.5) and +(+.1,.5) .. (0,3);
  \draw [black, very thick]    (2,3) .. controls +(+.1,.5) and +(-.1,.5) .. (3,3);
  \draw [black, very thick]    (0,3) -- (0,1);
  \draw [black, very thick]    (3,3) -- (3,1);
  \draw [black, very thick]    (1,1) .. controls +(-0.1,-0.5) and+(+0.1,-0.5) .. (0,1);
  \draw [black, very thick]    (2,1) .. controls +(0.1,-0.5) and +(-0.1,-0.5) .. (3,1);
    \end{tikzpicture} } \; \right\rrangle$  };
 \node (B4) at (3,0) {
        \begin{tikzpicture}[scale=.4]
  \draw [black, very thick]    (1,1) .. controls +(0,.35) and +(0,.35) .. (2,1);
  \draw [black, very thick]    (2,2) .. controls +(0,-.35) and +(0,-.35) .. (1,2);
  \draw [black, very thick]    (1,2) .. controls +(0,.35) and +(0,.35) .. (2,2);
  \draw [black, very thick]    (2,3) .. controls +(0,-.35) and +(0,-.35) .. (1,3);
  \draw [black, very thick]    (1,3) .. controls +(-.1,.5) and +(+.1,.5) .. (0,3);
  \draw [black, very thick]    (2,3) .. controls +(+.1,.5) and +(-.1,.5) .. (3,3);
  \draw [black, very thick]    (0,3) -- (0,1);
  \draw [black, very thick]    (3,3) -- (3,1);
  \draw [black, very thick]    (1,1) .. controls +(-0.1,-0.5) and+(+0.1,-0.5) .. (0,1);
  \draw [black, very thick]    (2,1) .. controls +(0.1,-0.5) and +(-0.1,-0.5) .. (3,1);
    \end{tikzpicture}};
 \node (B3) at (0,-1.5) {
        \begin{tikzpicture}[scale=.4]
  \draw [black, very thick]    (1,2) -- (1,1);
  \draw [black, very thick]    (2,2) -- (2,1);
  \draw [black, very thick]    (1,2) .. controls +(0,.35) and +(0,.35) .. (2,2);
  \draw [black, very thick]    (2,3) .. controls +(0,-.35) and +(0,-.35) .. (1,3);
  \draw [black, very thick]    (1,3) .. controls +(-.1,.5) and +(+.1,.5) .. (0,3);
  \draw [black, very thick]    (2,3) .. controls +(+.1,.5) and +(-.1,.5) .. (3,3);
  \draw [black, very thick]    (0,3) -- (0,1);
  \draw [black, very thick]    (3,3) -- (3,1);
  \draw [black, very thick]    (1,1) .. controls +(-0.1,-0.5) and+(+0.1,-0.5) .. (0,1);
  \draw [black, very thick]    (2,1) .. controls +(0.1,-0.5) and +(-0.1,-0.5) .. (3,1);
        \end{tikzpicture}};
 \node (B2) at (0,1.5) {
        \begin{tikzpicture}[scale=.4]
  \draw [black, very thick]    (1,1) .. controls +(0,.35) and +(0,.35) .. (2,1);
  \draw [black, very thick]    (2,2) .. controls +(0,-.35) and +(0,-.35) .. (1,2);
  \draw [black, very thick]    (1,2) -- (1,3);
  \draw [black, very thick]    (2,2) -- (2,3);
  \draw [black, very thick]    (1,3) .. controls +(-.1,.5) and +(+.1,.5) .. (0,3);
  \draw [black, very thick]    (2,3) .. controls +(+.1,.5) and +(-.1,.5) .. (3,3);
  \draw [black, very thick]    (0,3) -- (0,1);
  \draw [black, very thick]    (3,3) -- (3,1);
  \draw [black, very thick]    (1,1) .. controls +(-0.1,-0.5) and+(+0.1,-0.5) .. (0,1);
  \draw [black, very thick]    (2,1) .. controls +(0.1,-0.5) and +(-0.1,-0.5) .. (3,1);
    \end{tikzpicture}};
  \node (B1) at (-3,0) {\begin{tikzpicture}[scale=.4]
  \draw [black, very thick]    (1,1) -- (1,3);
  \draw [black, very thick]    (2,1) -- (2,3);
  \draw [black, very thick]    (1,3) .. controls +(-.1,.5) and +(+.1,.5) .. (0,3);
  \draw [black, very thick]    (2,3) .. controls +(+.1,.5) and +(-.1,.5) .. (3,3);
  \draw [black, very thick]    (0,3) -- (0,1);
  \draw [black, very thick]    (3,3) -- (3,1);
  \draw [black, very thick]    (1,1) .. controls +(-0.1,-0.5) and+(+0.1,-0.5) .. (0,1);
  \draw [black, very thick]    (2,1) .. controls +(0.1,-0.5) and +(-0.1,-0.5) .. (3,1);
    \end{tikzpicture}};
  \draw [thick,->] (B1) -- (B2);
  \draw [thick,->] (B1) -- (B3);
  \draw [thick,->] (B2) -- (B4);
  \draw [thick,->] (B3) -- (B4);
 \node at (-3,-1) {$\ket{00}$ };
 \node at (0,-2.5) {$\ket{10}$ };
\node at (0,.5) {$\ket{01}$ };
 \node at (3,-1) {$\ket{11}$ };
  \node at (-1.5,-1.25) {$m$};
  \node at (-1.5,1.25) {$m$};
 \node at (1.5,-1.25) {$-\delta$};
 \node at (1.5,1.25) {$\delta$};
\end{tikzpicture}
$
    \caption{An illustration of the Khovanov complex associated with the Hopf link.  Each resolution is arranged in a column by Hamming weight.  The differential connects edges related by swapping a single 0 to a 1 and has the effect of merging or splitting two circles.  Each merge of two circles is associated with the map $m$, and each split is associated the map $\delta$.  }
    \label{fig:Kh-Hopf} 
\end{figure}

If $r=r_1 \dots r_{a-1}0r_{a+1} \dots r_m$ and $v = v_1 \dots v_{a-1} 1 v_{a+1} \dots v_m$, then
\begin{equation} \label{eq:Kh-d}
    d^i = \sum_{r,v :  |r|=i} (-1)^{v_1 + \dots v_{a-1}} d_{r,v},
\end{equation}
where $d_{r,v}$ acts trivially on all copies of $V$ not involved in the split or merge and 
\begin{itemize}
    \item if $d_{r,v}$ merges two loops into one, then $d_{r,v}$ applies the map $m\maps V \otimes V \to V$ on the copies of V associated with the loops being merged, and 
    \item if $d_{r,v}$ splits one loops into two, then $d_{r,v}$ acts by the map $\delta\maps V \to V \otimes V$ on the loop being split into two. 
\end{itemize}
Note that the grading shift of $\langle i \rangle$ on $\llangle K\rrangle^i$ and $\langle i +1\rangle$ on $\llangle K\rrangle^{i+1}$ from \cref{eq:Kh-C} ensures that these differentials are (quantum) degree preserving.  
One can check that with this definition $d^{i+1}\circ d^i = 0$.   This differential on $\llangle K\rrangle$ gives rise to the highly nontrivial nature of Khovanov's categorification of the Jones polynomial.

In Section~\ref{subsec:Kh-boundary} we give a quantum interpretation of the Khovanov differential, where the signs in \cref{eq:Kh-d} have a natural interpretation via Jordan Wigner creation operators.   In section~\ref{subsubsec:homzero-structure} we will delve into the structure of Khovanov homology in homological degree zero and examine these differentials more closely.

\subsubsection{From the Khovanov bracket to Khovanov homology} \label{subsec:Kh-bracket-homology}

Just as the Jones polynomial is obtained from the Kauffman bracket via an overall rescaling by $(-1)^{n_-} q^{n_+ -2n_-}$ as in \cref{eq:Jones-shift} to fix Reidemeister  invariance, the complex $\mathbf{C}$ used to compute Khovanov homology is obtained by an overall shift  in homological degree  by $n_-$ and quantum degree shift by $\langle n_+ -2n_-\rangle$, so that $\mathbf{C}(K)= \llangle K \rrangle [-n_-]\langle n_+ -2n_-\rangle$ meaning that the bigraded chain spaces in Khovanov homology are given by
\[
C(K)^{i,j}  : = \llangle K \rrangle^{i+n_-, j-n_+ +2n_-}.
\]
Here $[s]$ is the homological height shift operation that shifts the index of a complex $\mathbf{C}$ so that $C[s]^i := C^{i-s}$ with the differentials shifted accordingly.  

With these new shifts, we have the following:
\begin{theorem}[\cite{Kh1}]
    The graded Euler characteristic of $\mathbf{C}(K)$ is the unnormalized Jones polynomial of $K$.
    \[
\chi ( \mathbf{C}(K) ) = J(K).
\]
\end{theorem}

The Khovanov homology of a knot $K$ is the homology of the complex $\mathbf{C}(K)$.  
Let $H^i(\mathbf{C}(K))$ be the $i$th homology of the complex $\mathbf{C}(K)$.  Since all the differentials in this complex have (quantum) degree zero, the homology splits as a direct sum of quantum degrees 
\[
H^i(\mathbf{C}(K))= \bigoplus_j H^{i,j}\mathbf{C}(K),
\]
and we define  $Kh^{i,j}(K):= H^{i,j}(\mathbf{C}(K))$.   Our goal in this article is to compute the graded Betti numbers $\beta_{i,j}(K):= \dim Kh^{i,j}(K)$.

For the purposes of this article, computing the homology of the Khovanov bracket $\llangle K\rrangle$ has the same difficulty as computing Khovanov homology since the two only differ by overall homological and quantum degree shifts.  For this reason, we will often streamline our presentation and ignore these overall degree shifts.

\subsection{The unknotting problem}
\label{sec:intro_unknot}
An additional motivation for studying Khovanov homology arises from the \emph{unknotting problem:}
\begin{quote}
    Given a knot diagram $D$, determine whether $D$ represents the unknot (the trivial knot with no crossings).
\end{quote}
The Jones polynomial of the unknot is $J(\KPA) = 1$. It is not known whether the Jones polynomial is a strong enough invariant to recognize the unknot, that is, whether the trivial knot is the only knot whose Jones polynomial is 1. 

On the other hand, Khovanov homology is known to detect the unknot due to the seminal result by Kronheimer and Mrowka~\cite{KM}. Kronheimer and Mrowka show that a given representation of a knot corresponds to the unknot if and only if its Khovanov homology has the form  
\begin{equation}
Kh^{i,j}(\KPA) = \begin{cases} 
\mathbb{C} & \text{if } (i,j) = (0,1) \text{ or } (0,-1), \\
0 & \text{otherwise}. 
\end{cases}
\end{equation}
This means that any algorithm that can certify that the only nonzero Betti numbers of a knot are $\beta_{0,1}(K) = \beta_{0,-1}(K) = 1$ detects the unknot. 

It is a major unresolved challenge to determine whether there exists an efficient (polynomial time) algorithm for the unknotting problem. After Haken \cite{haken1961theorie} showed in 1961 that knot-triviality is decidable, a sequence of works established increasingly tight upper bounds on the complexity class of the unknotting problem. For a definition of these complexity classes, see \Cref{sec:complex_prelim}. Hass et al. \cite{HLP} showed in 1999 that the unknotting problem is contained in $\NP$. Kuperberg \cite{kuperberg2014knottedness} showed that the unknotting problem is also contained in $\mathsf{co}$-$\NP$, provided the generalized Riemann hypothesis holds. In 2016, Lackenby \cite{lackenby2015polynomial} improved this result to an unconditional proof of $\mathsf{co}$-$\NP$ containment. This places the unknotting problem in the complexity class $\NP$ $\cap$ $\mathsf{co}$-$\NP$, meaning it is believed to be an \emph{$\NP$-intermediate} problem. An efficient quantum algorithm for a $\NP$-intermediate problem does not clash with the widely held belief that $\NP \not \subseteq \BQP$. A well-studied NP-intermediate problem is the factoring problem, which is famously solved efficiently on a quantum computer via Shor's algorithm.


\section{Quantum algorithms for general homologies}
\label{sec:general_homology}
 
Quantum algorithms for computing the ranks of homology groups were first developed in \cite{lloyd2016quantum} in the context of clique homology, which is the basis of Topological Data Analysis. Recently, these quantum algorithms have experienced a surge of attention, see e.g. \cite{lloyd2016quantum,quantumAdvantage,persistent,IBMresult,GoogleBettiBerry,AmazonBettiMcArdle,schmidhuber2023complexity,crichigno2022clique,king2023promise,hayakawa2024quantum,gyurik_schmidhuber_2024quantum}. A small fraction of these works have considered extending the techniques used in qTDA to homologies that are more general than simplicial. Cade and Crichigno \cite{cade2021complexity} formally discussed sufficient requirements for efficiently estimating so-called \emph{normalized quasi-Betti numbers} on general complexes with log-local or sparse boundary operators. In this section, we summarize sufficient requirements for estimating \emph{Betti numbers} on arbitrary complexes with arbitrary boundary operators. For simplicity, we keep our discussion in this section informal, but the next section discusses each requirement in detail for the special case of Khovanov homology. 

Recall from \Cref{sec:chain} that in the most general setting, a homology is specified by a chain complex $C$, which is a (possibly infinite) sequence of abelian groups $C_k$ connected by homomorphisms $\del_k$, such that 
\begin{equation}
\hdots \xrightarrow{\hspace{2pt}\del_{j+2}\hspace{2pt} }  C_{j+1} \xrightarrow{\hspace{2pt}\del_{j+1}\hspace{2pt} }  C_{j} \xrightarrow{\hspace{5pt}\del_{j}\hspace{5pt} } C_{j-1} \xrightarrow{\hspace{2pt}\del_{j-1}\hspace{2pt} } C_{j-2}  \xrightarrow{\hspace{2pt}\del_{j-2}\hspace{2pt} } \hdots
\end{equation}
and \begin{equation}
    \del_{k-1} \del_{k} = 0
\end{equation} for all $k$. Because of the above equation,  $\del_k$ is commonly referred to as \emph{boundary operator}.
Associated to this chain complex are the homology groups  \begin{equation}
    H_k = \ \bigslant{Ker $\del_k$}{Im $\del_{k+1}$}.
\end{equation}
Elements of $\Ker \del_k$ and $\Im \del_{k+1}$ are called $k$-cycles and $k$-boundaries, respectively. Throughout this paper, we will utilize a different characterization of homology groups that is more natural from the perspective of quantum computation. This characterization is based on Hodge theory, which states that there exists a linear operator (called the Laplacian) $\Delta_k = \del_k^*\del_k + \del_{k+1}\del_{k+1}^*$ such that \begin{equation}
    H_k \cong \text{Ker } \Delta_k. 
\end{equation}  The adjoint $\del_k^*$ is defined via an inner product that depends on the specific chain complex. While the derivation of the full Hodge relation is mathematically involved, it can be understood in this context intuitively as follows: Both $\del_{k}^*\del_{k}$ and $\del_{k+1}\del_{k+1}^*$ are non-negative Hermitian operators, thus any chain annihilated by the Laplacian lies in the intersection of the kernels of $\del_k$ and $\del_{k+1}^*$. The first operator enforces that the chain is a cycle, while the second operator fixes a certain representative per equivalence class. Two cycles are considered equivalent if they can be continuously deformed to each other within the chain complex, i.e., if they differ by a boundary.

In anticipation of representing such a chain space $C_i$ on a quantum mechanical Hilbert space made out of qubits, we consider the case where $C$ consists of finitely many chain spaces, each being a vector space over the complex numbers $\mathbb C$. 
\begin{equation}
0 \xrightarrow{\hspace{2.5pt}\del_{n}\hspace{2.5pt} }  C_{n-1} \xrightarrow{\del_{n-1}}\hdots  \xrightarrow{\hspace{4pt}\del_{1}\hspace{4pt} } C_0  \xrightarrow{\hspace{3pt}\del_{0}\hspace{3pt} } 0.
\end{equation}
With a slight abuse of notation, such a chain complex is compactly specified by a graded vector space $C$ and a boundary operator $\del$, 
\begin{equation}
    C = \bigoplus_{k=0}^m C_k, \quad\quad \del_k : C_k \to C_{k-1}.
\end{equation} 
The Laplacian $\Delta_k$ is a self-adjoint linear operator acting on the complex vector space $C$, and can therefore be interpreted as the Hamiltonian of a quantum mechanical system. Estimating the Betti numbers
\begin{equation}
    \beta_k = \dim H_k = \dim \left( \bigslant{Ker $\del_k$}{Im $\del_{k+1}$} \right) = \dim\text{Ker} \left(\del_k^{\dagger} \del_k + \del_{k+1}\del_{k+1}^{\dagger}\right) = \dim \text{Ker} \Delta_k
\end{equation} 
then corresponds to estimating the dimension of the ground state space of this Hamiltonian. Finding the ground state of a quantum mechanical system is a canonical problem in quantum information theory that is well-suited for quantum algorithms. 
 
Specifically, we describe a quantum algorithm for approximating Betti numbers that is efficient under the following four requirements:
\begin{enumerate}
    \item The Hilbert space $C$ has to have an efficient description, for example, it could be the subspace of a full $n-$qubit Hilbert space described by a polynomial (in $n$) number of constraints.
    \item The Hodge Laplacian $\Delta_k = \left(\del_k^\dagger \del_k + \del_{k+1}\del_{k+1}^\dagger\right)$ has to be efficiently exponentiable, i.e., the unitary $U = e^{i\Delta_k t}$ has to be described by a polynomially large quantum circuit that can be efficiently computed. This is the case, for example, if $\del$ is local or sparse. 
    \item The Laplacian $\Delta_k$ has to thermalize efficiently, meaning there is an efficient quantum circuit that prepares a state $\rho_k$ that has inverse-polynomial overlap with the uniform mixture over the ground state space of $\Delta_k$.  For details, see the discussion below and \Cref{sec:thermalization}. 
    \item The Hodge Laplacian has to have a spectral gap at least inverse polynomial in $n$: If $\delta$ is the smallest non-zero eigenvalue of $\Delta_k$, then \begin{equation}
        \delta^{-1} = \mathcal{O}\left(\text{poly}(n)\right).
    \end{equation}
\end{enumerate}

The third requirement is an improvement over the original quantum homology algorithm that is crucial in the context of Khovanov homology. The original quantum homology algorithm \cite{lloyd2016quantum} operates by preparing a uniform mixture state \begin{equation}
        \rho_k = \frac{1}{\dim{C_k}}\sum_{s \in C_k} \ket{s}\bra{s}, \quad \text{where $\{s\}_s$ is some orthogonal basis of $C_k$,}
    \end{equation}  and by using the quantum phase estimation algorithm to project onto the kernel of the Hodge Laplacian $\Delta_k$. The probability that this projection succeeds is given by the normalized Betti number $\beta_k/\dim{C_k}$. The resulting state is then a mixture over the representatives of the homology, which can now be sampled and their properties measured.  This method succeeds if the Betti numbers are large, such that the ratio between the Betti number $\beta_k$ and the dimension of $C_k$ is not exponentially small.

If the Betti numbers are small -- as is empirically the case in Khovanov homology, see \Cref{sec:numerics_gap} -- then we employ an alternative approach based on \cite{abrams1999quantum}. We employ a thermalizing Lindblad equation to ``cool'' the system whose Hamiltonian is the Hodge Laplacian to a low-temperature thermal state \cite{scali2024topology} with non-negligible overlap with the ground state space. We then use quantum phase estimation to project onto the kernel, yielding a uniform mixture over the representatives of the homology. Applying a SWAP test to multiple copies of this state then reveals the corresponding Betti number. We provide more details of this approach in \Cref{sec:thermalization}.

The caveat with the latter method is that implementing a thermalizing Lindblad equation may be hard. Indeed, from the computational complexity results in \Cref{sec:complexity}, we do not expect the thermalization to succeed in all instances. In the next section, we will check under which conditions the above requirements hold for Khovanov homology.

\section{Quantum algorithm for Khovanov homology}
\label{sec:algorithm}

\subsection{Overview}

In this section, we give a quantum algorithm for producing an additive approximation to the Betti numbers of Khovanov homology, given a description of a planar projection of a knot.   The algorithm is described in three subsections. 
\begin{itemize}
    \item In section~\ref{subsec:Kh-encoding}, we describe our encoding of the Khovanov complex.   This section closely follows Bar-Natan's encoding used in his original classical algorithm for computing Khovanov homology~\cite{BN1}, though we show that this encoding can be reversibly implemented.

    \item In sections~\ref{sec:enhanced_state} and~\ref{subsec:Kh-boundary}, we depart from Bar-Natan's work by giving a new quantum description and implementation of the boundary operator. Here, we show that the boundary operator and the corresponding Laplacian can be exponentiated efficiently on a quantum computer.

    \item In section~\ref{subsec:Kh-algorithm} we describe the quantum algorithm.  To our knowledge, this is the first quantum algorithm for computing Betti numbers of a concrete non-simplicial complex. Our quantum algorithm also combines several recent techniques for efficiently implementing the LGZ algorithm.

    \item Finally, in section~\ref{subsec:Kh-performance} we discuss the performance of our algorithm and ways to further improve its scaling.  
\end{itemize}
 
The algorithm we describe uses an encoding of the knot that is by no means maximally efficient.  Many steps involved in encoding the data for the Khovanov algorithm could be outsourced to a classical computer.  We have chosen here to use an encoding that illuminates all relevant aspects involved in computing the Khovanov complex boundary operators, instead of favoring efficiency. 

\subsection{Encoding the Khovanov complex}  \label{subsec:Kh-encoding}
Given a knot diagram $K$ with $m$ crossings, we label all its $2m$ edges by integers $[2m] = {1,\dots,2m}$ using an arbitrary orientation and the following procedure.   Ignore the over and under-crossing information of the knot diagram and regard each crossing as a four-valent vertex, each knot diagram contains $2m$ edges between the $m$ vertices.  We label each edge by choosing an arbitrary starting point, then following the orientation, each segment is labeled in order as we traverse the knot getting back to the original starting point.

The orientation on $K$ allows us to read off a 4-tuple of edge labels for each crossing $X_k$ by reading the labels counterclockwise according to the convention 
\begin{equation} \label{eq:crossing-labels}
X_k \;\; = \;\; \hackcenter{
\begin{tikzpicture}
\draw[very thick, ] (.5,0) -- (-.5,1.5);
\draw[very thick ] (-.5,0) -- (-.1,.6);
\draw[very thick, -> ] (.1,.9) -- (.5,1.5);
\node at (-.55,-.3) {\small$\underline{a_k}$};
\node at (.55,-.3) {\small$\underline{b_k}$};
\node at (-.55,1.8) {\small$\underline{d_k}$};
\node at (.55,1.8) {\small$\underline{c_k}$};
\end{tikzpicture}}    
\end{equation}
The Knot $K$ can then be encoded by a 4$m$-tuple of edge labels\footnote{Bar-Natan would use the notation $X_{a_1,b_1,c_1,d_1} \dots X_{a_m,b_m,c_m,d_m}$} 
\begin{equation}
\ket{K} = \ket{a_1, b_1, c_1 ,d_1 \mid \dots \mid  a_m, b_m, c_m ,d_m}
\end{equation}
where each 4-tuple represents the edges incident to the crossing as in \cref{eq:crossing-labels}.
See Figure~\ref{fig:trefoil} for an example of this encoding for a trefoil. 

\begin{figure}    \centering
\begin{equation} \label{hopfpic}
K \;\; = \;\; 
\hackcenter{   \begin{tikzpicture}[scale=.7]
      \draw [ very thick]    (1,1) .. controls +(0,.25) and +(0,-.25) .. (2,2);
  \draw [ very thick]    (1,2) .. controls +(0,.25) and +(0,-.25) .. (2,3);
   \draw [ very thick]    (1,3) .. controls +(0,.25) and +(0,-.25) .. (2,4);
  \path [fill=white] (1.35,1) rectangle (1.65,4);
  \draw [ very thick]    (2,1) .. controls +(0,.25) and +(0,-.25) .. (1,2);
  \draw [ very thick]    (2,2) .. controls +(0,.25) and +(0,-.25) .. (1,3);
    \draw [ very thick]    (2,3) .. controls +(0,.25) and +(0,-.25) .. (1,4);
  \draw [ very thick]    (1,4) .. controls +(-.1,.35) and +(+.1,.5) .. (0,4);
  \draw [ very thick]    (2,4) .. controls +(0,.35) and +(-.1,.5) .. (3,4);
  \draw [ very thick]    (0,4) -- (0,1);
  \draw [ very thick, <-]    (3,4) -- (3,1);
  \draw [ very thick]    (1,1) .. controls +(-0.1,-0.35) and+(+0.1,-0.5) .. (0,1);
  \draw [ very thick]    (2,1) .. controls +(0.1,-0.35) and +(-0.1,-0.5) .. (3,1);
  \node at (.8,3.9) {$\scriptstyle\underline{1}$};
   \node at (2.2,3) {$\scriptstyle\underline{2}$};
   \node at (2.2,2) {$\scriptstyle\underline{6}$};
   \node at (.8,3) {$\scriptstyle\underline{5}$};
   \node at (2.2,1.1) {$\scriptstyle\underline{4}$};
   \node at (.8,1.1) {$\scriptstyle\underline{1}$};
   \node at (.8,2) {$\scriptstyle\underline{3}$};
    \end{tikzpicture}  }
   \qquad 
  \ket{K} = \ket{5,2,4,1 \mid 3,6,2,5 \mid 1,4,6,3} 
\end{equation}
    \caption{A trefoil knot with $m=3$ and $2m=6$ edge labels.  This knot is encoded by $\ket{K}$.}
    \label{fig:trefoil}
\end{figure}

As explained in section ~\ref{subsec:lifting-Kauffman}, the computation of Khovanov homology requires one to compute all resolutions $r$ of $K$ and determine the number $\ell(r)$ of resulting loops in the resolution.   
The input of our encoding algorithm requires the following data.
\medskip

\noindent \textbf{INPUT:}  An oriented knot diagram $K$ with $m$ crossings presented as a $4m$-tuple of edge labels associated with each of the $m$ crossings
  \[   \ket{K} = \ket{a_1, b_1, c_1 ,d_1 \mid \dots \mid  a_m, b_m, c_m ,d_m} .\]
This can be encoded using $4m\log(2m)$ qubits. 
\smallskip

Then, to encode the boundary operator, we will make use of the following:
\begin{itemize}
    \item Resolution register: $\ket{r} = \ket{r_1 \dots r_m}$ is a string of qubits that specifies a choice of resolution for each crossing.

    \item Resolved knot register: $\ket{K_r}$ contains the list of crossings resolved by $r$.

    \item Loop enumeration register: $\ket{L_r} = \ket{k_1 \dots k_{2m} }$ keeps track of the loops arising in the resolved knot. It is a string of qubits with a $1$ in position $i$, if $i$ is the minimal edge label appearing in a closed loop arising in the resolution $\ket{r}$.   Then the Hamming weight $|\ket{L_r}|$ is the number of loops $\ell(r)$ in the resolution $\ket{r}$.  

    \item Enhanced states: $\ket{s} = s_1 \dots ... s_{2m}$ is a sequence of qubits denoting the enhanced state on the set of loops from $L_r$. Only the $s_i$ with $k_i=1$ carry information relevant to the encoding.  The other $s_i$ carries no information relevant to the computation.  
\end{itemize}

\subsubsection{Resolution register} 
 Arbitrarily enumerate all the $m$ crossings in the knot $K$.  Then any bitstring $\ket{r}=r_1 \dots r_m$ gives rise to a complete resolution $K_r$ of the knot $K$ into a diagram containing no crossings, where the $i$th crossing has been resolved by the $0$-resolution if $r_i=0$ and the 1-resolution if $r_i=1$.  In this way, the resolution register $\ket{r}$ encodes the $2^m$ possible complete resolutions of the $m$ crossing knot $K$.  See Figure~\ref{fig:hr} for examples of resolutions resulting from the trefoil in Figure~\ref{fig:trefoil} where the crossings have been enumerated from top to bottom. 

\subsubsection{Resolved knot register}  \label{subsec:resolution-reg}

Each of the right-hand crossings $X_k$ is specified by a tuple of four edges $(a_k,b_k,c_k,d_k)$. Applying a 0-smoothing at this crossing connects $(a_k,b_k)$ and $(c_k,d_k)$. The 1-smoothing would connect $(a_k,d_k)$ and $(c_k,b_k)$.  
\[  
\hackcenter{
\begin{tikzpicture}
\draw[very thick, ] (-.5,0) .. controls ++(0,.5) and ++(0,.5) ..  (.5,0);
\draw[very thick, ] (-.5,1.5) .. controls ++(0,-.5) and ++(0,-.5) ..  (.5,1.5);
\node at (-.55,-.3) {\small$\underline{a_k}$};
\node at (.55,-.3) {\small$\underline{b_k}$};
\node at (-.55,1.8) {\small$\underline{d_k}$};
\node at (.55,1.8) {\small$\underline{c_k}$};
\end{tikzpicture}} 
\qquad \qquad 
\hackcenter{
\begin{tikzpicture}
\draw[very thick, ] (-.5,0) .. controls ++(0,.25 ) and ++(.0,-.45) ..  (-0.25,.75) ..                   controls ++(0,.45) and ++(0,-.25 ) ..(-.5,1.5);
\draw[very thick, ] (.5,0) .. controls ++(0,.25 ) and ++(.0,-.45) ..  (0.25,.75) ..                   controls ++(0,.45) and ++(0,-.25 ) ..(.5,1.5);
\node at (-.55,-.3) {\small$\underline{a_k}$};
\node at (.55,-.3) {\small$\underline{b_k}$};
\node at (-.55,1.8) {\small$\underline{d_k}$};
\node at (.55,1.8) {\small$\underline{c_k}$};
\end{tikzpicture}} 
\] 

To compute the differential in Khovanov homology, we make use of a $4m$-tuple of knot edge labels $\ket{K_r}$ that contains a list of the values $\{a_k,b_k,c_k,d_k\}$ of each crossing $X_k$, reordered so that for each crossing, the first and second labels are connected, and the third and fourth labels are connected as illustrated below:
\[
\xy
 (-30,0)*+{ \{a_k, b_k, c_k, d_k \} }="L"; 
  (20,6)*+{ \{ (a_k, b_k), (c_k, d_k) \} }="T"; 
   (20,-6)*+{ \{ (a_k, d_k),  (c_k, b_k)  \} }="B"; 
    {\ar^{\text{0-resolution} } "L" ;"T"};
    {\ar_{\text{1-resolution} } "L" ;"B"};
\endxy 
\]
The relative order of the pairs of connected edges will not be important in what follows, so we do not distinguish between $\{ (a_k, d_k),  (c_k, b_k)  \}$ and $\{ (d_k, a_k),  (b_k,c_k )  \}$, and we omit the parenthesis indicating the two tuples.

The register $\ket{K_r}$ can easily be extracted from the knot input $\ket{K}$ and the resolution register $\ket{r}$ by doing nothing to the $4$-tuple $a_k, b_k, c_k, d_k$ associated to the crossing $X_k$ if the $k$th bit of $\ket{r}$ is 0, and performing a swap on the second and fourth edge label if the $k$th bit of $\ket{r}$ is 1.  In particular, the resolution $\ket{00\dots} = \ket{0}^{\tens m}$ is 
\begin{equation}
   \ket{r}\ket{K_r} =  \ket{00\dots}\ket{K_{(00\dots)}} = \ket{00\dots}\ket{a_1,b_1,c_1,d_1,a_2,b_2,c_2,d_2,\dots,a_m,b_m,c_m,d_m},
\end{equation} 
so that $\ket{K_{ (00\dots) }}$  is just the input $\ket{K_r}$. Any other resolution $\ket{r}$ with $r_k=1$ can be described by swapping the positions of $b_k$ and $d_k$, for example, 
\begin{equation}
    \ket{00\dots01}\ket{K_{(00\dots01)}} = \ket{00\dots 01}\ket{a_1,b_1,c_1,d_1,a_2,b_2,c_2,d_2,\dots,a_m,d_m,c_m,b_m},
\end{equation} where $b_m$ and $d_m$ have been swapped. 

More generally, introduce a unitary operator on our $4mlog(2m)$ register encoding $\ket{K}$ that  operates on edge labels in positions $4k-3,4k-2,4k-1,4k$ as 
\begin{equation} \label{eq:SWAP}
    \text{SWAP(k)} = \mathbb{1}_{4k-3} \tens \text{SWAP}_{4k-2,4k}\tens \mathbb{1}_{4k-1}.
\end{equation} 
This SWAP operator can be efficiently implemented. An arbitrary resolution register $\ket{K_r}$ can be obtained from $\ket{K}$ using the bits of the resolution register $r$ as a control input into the swap operator.   We make a copy of $\ket{K}$ and use it to construct  
 \begin{equation}
    \ket{K_r} = \text{SWAP}^{\tens r} \ket{K} =\left(\bigotimes_{k=1}^m \text{SWAP(k)}^{r_k} \right)\ket{K}.
\end{equation}
See Figure~\ref{fig:hr} for an example with the trefoil.

 \begin{figure}
     \centering
 \begin{align} 
 \label{hopfpic-000}
\ket{r} &= \ket{000}  = \;\;  
\hackcenter{   \begin{tikzpicture}[scale=.6]
      \draw [ very thick]    (1,1) .. controls +(0,.45) and +(0, .45) .. (2,1);
  \draw [ very thick]    (1,2) .. controls +(0,.45) and +(0,.45) .. (2,2);
   \draw [ very thick]    (1,3) .. controls +(0,.45) and +(0,.45) .. (2,3);
     \draw [ very thick]    (1,4) .. controls +(0,-.45) and +(0,-.45) .. (2,4);
  \draw [ very thick]    (2,2) .. controls +(0,-.45) and +(0,-.45) .. (1,2);
  \draw [ very thick]    (2,3) .. controls +(0,-.45) and +(0,-.45) .. (1,3); 
  \draw [ very thick]    (1,4) .. controls +(-.1,.35) and +(+.1,.5) .. (0,4);
  \draw [ very thick]    (2,4) .. controls +(0,.35) and +(-.1,.5) .. (3,4);
  \draw [ very thick]    (0,4) -- (0,1);
  \draw [ very thick]    (3,4) -- (3,1);
  \draw [ very thick]    (1,1) .. controls +(-0.1,-0.35) and+(+0.1,-0.5) .. (0,1);
  \draw [ very thick]    (2,1) .. controls +(0.1,-0.35) and +(-0.1,-0.5) .. (3,1);
  \node at (.8,3.9) {$\scriptstyle\underline{1}$};
   \node at (2.2,3) {$\scriptstyle\underline{2}$};
   \node at (2.2,2) {$\scriptstyle\underline{6}$};
   \node at (.8,3) {$\scriptstyle\underline{5}$};
   \node at (2.2,1) {$\scriptstyle\underline{4}$};
   \node at (.8,2) {$\scriptstyle\underline{3}$};
    \end{tikzpicture}  }
    \qquad 
   \ket{ K_{000} }  = \ket{ 5,2, 4,1 \mid 3,6,2,5 \mid 1,4,6,3 }
   \\
\label{hopfpic-110}
\ket{r} &= \ket{110}  = \;\;  
\hackcenter{   \begin{tikzpicture}[scale=.6]
      \draw [ very thick]    (1,1) .. controls +(0,.45) and +(0, .45) .. (2,1);
  \draw [ very thick]    (1,2) .. controls +(0,-.45) and +(0,-.45) .. (2,2);
   \draw [ very thick]    (1,3) .. controls ++(0,.25) and ++(0,-.25) ..  (1.15,3.5)
                                .. controls ++(0,.25) and ++(0,-.25) .. (1,4);
  \draw [ very thick]    (2,2) .. controls ++(0,.25) and ++(0,-.25) ..  (1.85,2.5)
                                .. controls ++(0,.25) and ++(0,-.25) .. (2,3);
  \draw [ very thick]    (1,2) .. controls ++(0,.25) and ++(0,-.25) ..  (1.15,2.5)
                                .. controls ++(0,.25) and ++(0,-.25) ..(1,3);
    \draw [ very thick]    (2,3) .. controls ++(0,.25) and ++(0,-.25) ..  (1.85,3.5)
                                .. controls ++(0,.25) and ++(0,-.25) .. (2,4);
  \draw [ very thick]    (1,4) .. controls +(-.1,.35) and +(+.1,.5) .. (0,4);
  \draw [ very thick]    (2,4) .. controls +(0,.35) and +(-.1,.5) .. (3,4);
  \draw [ very thick]    (0,4) -- (0,1);
  \draw [ very thick]    (3,4) -- (3,1);
  \draw [ very thick]    (1,1) .. controls +(-0.1,-0.35) and+(+0.1,-0.5) .. (0,1);
  \draw [ very thick]    (2,1) .. controls +(0.1,-0.35) and +(-0.1,-0.5) .. (3,1);
  \node at (.8,3.9) {$\scriptstyle\underline{1}$};
   \node at (2.2,3) {$\scriptstyle\underline{2}$};
   \node at (2.2,2) {$\scriptstyle\underline{6}$};
   \node at (.8,3) {$\scriptstyle\underline{5}$};
   \node at (2.2,1) {$\scriptstyle\underline{4}$};
   \node at (.8,2) {$\scriptstyle\underline{3}$};
    \end{tikzpicture}  }
    \qquad
\ket{K_{110}} =  \text{SWAP}(1)\otimes\text{SWAP}(2) \ket{K_{000}}
= \ket{ 5,1, 4,2 \mid 3,5,2,6 \mid 1,4,6,3 }
\end{align} 

    \caption{Two examples of resolutions $\ket{r}$ of the trefoil knot from Figure~\ref{fig:trefoil}.   The resolution encoding register $\ket{ K_{000} }$ is identical to the knot input $\ket{K}$ we used to specify the knot in Figure~\ref{fig:trefoil}. Observe that each subsequent pair of edge labels $\ket{ (5,2), (4,1) \mid (3,6), (2,5) \mid (1,4),(6,3)}$ are connected in the resolution $\ket{000}$.  The resolution encoding register $\ket{K_{110}}$ for the resolution $\ket{r}=\ket{110}$ is obtained from $\ket{K_{000}}$ via the rule \cref{eq:SWAP}, and again, each two-tuple of edge labels are connected in the resolution $\ket{r}=\ket{110}$.   }
     \label{fig:hr}
 \end{figure}

\subsubsection{Loop counting and enumeration register} 
Implementing the boundary operator requires a determination of the number of loops $\ell(r)$ for each resolution $\ket{r}$, as well as a label for each loop by the smallest edge label that appears in the resolution $\ket{r}$. 
This association of a minimal edge label to each loop is needed to specify which tensor factors are acted on by the maps \cref{eq:mdelta} comprising the differential.

For example, the trefoil knot from Figure~\ref{fig:trefoil} would have the two loops of the resolution $\ket{100}$ labelled 1 and 3, respectively. 
\begin{equation} \label{hopfpic-100}
\ket{r} = \ket{100} \;\; = \;\; 
\hackcenter{   \begin{tikzpicture}[scale=.7]
      \draw [ very thick]    (1,1) .. controls +(0,.45) and +(0, .45) .. (2,1);
  \draw [ very thick]    (1,2) .. controls +(0,.45) and +(0,.45) .. (2,2);
   \draw [ very thick]    (1,3) .. controls ++(0,.25) and ++(0,-.25) ..  (1.15,3.5)
                                .. controls ++(0,.25) and ++(0,-.25) .. (1,4);
  \draw [ very thick]    (2,2) .. controls +(0,-.45) and +(0,-.45) .. (1,2);
  \draw [ very thick]    (2,3) .. controls +(0,-.45) and +(0,-.45) .. (1,3);
    \draw [ very thick]    (2,3) .. controls ++(0,.25) and ++(0,-.25) ..  (1.85,3.5)
                               .. controls ++(0,.25) and ++(0,-.25) .. (2,4);
  \draw [ very thick]    (1,4) .. controls +(-.1,.35) and +(+.1,.5) .. (0,4);
  \draw [ very thick]    (2,4) .. controls +(0,.35) and +(-.1,.5) .. (3,4);
  \draw [ very thick]    (0,4) -- (0,1);
  \draw [ very thick]    (3,4) -- (3,1);
  \draw [ very thick]    (1,1) .. controls +(-0.1,-0.35) and+(+0.1,-0.5) .. (0,1);
  \draw [ very thick]    (2,1) .. controls +(0.1,-0.35) and +(-0.1,-0.5) .. (3,1);
  \node at (.8,3.9) {$\scriptstyle\underline{1}$};
   \node at (2.2,3) {$\scriptstyle\underline{2}$};
   \node at (2.2,2) {$\scriptstyle\underline{6}$};
   \node at (.8,3) {$\scriptstyle\underline{5}$};
   \node at (2.2,1) {$\scriptstyle\underline{4}$};
  \node at (.8,2) {$\scriptstyle\underline{3}$};
    \end{tikzpicture}  }
\qquad
\ket{K_{100}} = \text{SWAP}(1)\ket{K_{000}} = \ket{ 5,1, 4,2 \mid 3,6,2,5 \mid 1,4,6,3 }
\end{equation}
The resolution $\ket{000}$ has $\ell(000)=3$, with the loops labelled $1$, $2$ and $3$.  The resolution $\ket{110}$ has $\ell(110)= 1$ with minimal edge label $1$.   See Example~\ref{example:trefoil-label} for a more complex example.

Given $\ket{K_r}$, we produce a bit string  $\ket{L_r} = \ket{L_1 \dots L_{2m}}$ that lists the loops appearing in the resolution $r$ of $K$, together with their labeling by the minimal edge appearing in the loop. This encoding has the $j$-th bit of $L_r$ set to 1 if a loop with minimal label $j$ is in the resolution $r$ and 0 otherwise\footnote{Since every loop appearing in a resolution $r$ has at least two edge labels appearing, we will never need the label $2m$ as a minimal label of a loop in a resolution.}. 
For example, the resolution from \cref{hopfpic-100} of the trefoil  from Figure~\ref{fig:trefoil} would have 
\[
\ket{L_r } = \ket{L(100)} = \ket{101000}
\]
Similarly, the resolution $\ket{r} = \ket{000}$ has $\ket{L(000)} = \ket{111000}$ and the resolution $\ket{r} =\ket{110}$ has $\ket{L_r}=\ket{100000}$.

The list $\ket{K_r}$ of $4m$ edge labels can be regarded as $2m$ unordered two-sets $(x,y)$, describing two edges $x,y$ that became connected in the resolution $\ket{r}$.  By construction, the edges $(x,y)$ are connected in the smoothing $r$ of the knot $K$. 
  
Two unordered pairs $(x_i,y_i)$ and $(x_j,y_j)$ of $\ket{K_r}$ are said to have a contraction if $y_i = x_j$ or $x_i = x_j$ or $y_i = y_j$ for some $1 \leq i \not = j \leq 2m$.  In this case, these edges are connected in the resolution $K_r$.  

\subsubsection*{Loop counting algorithm}
The algorithm that takes as input $\ket{K_r}$ and outputs $\ket{L_r}$ is the following:
\smallskip

Start from the first resolved crossing in $\ket{K_r}$.    This contains two pairs of edges.   Use the connectivity information contained in $\ket{K_r}$  to follow the first edge from crossing to crossing until one returns to the original edge, thereby closing the loop.  While following the loop, keep track of the smallest edge index contained in the loop.  Place a $1$ at the corresponding entry of $\ket{L_r}$.   If the second pair of edges in the crossing is not in the loop created from the first edge, do the same for the second pair of edges: construct the loop and place a $1$ in the corresponding entry of $\ket{L_r}$.

Continue in the same fashion.  Move on to the next edge pair in $\ket{K_r}$; if the loop generated by that edge pair has not already been listed, then place a $1$ in the corresponding entry of $\ket{L_r}$.  Else move on to the next edge pair.

When all loops have been found, and their smallest edge indices recorded in $\ket{L_r}$, reversibly uncompute the list of edges in each loop.   The result is the state $\ket{K_r}\ket{L_r}$.

\begin{example}
The state $\ket{110}\ket{K_{110}} = \ket{110}\ket{5,1,4,2 \mid 3,5, 2,6 \mid 1,4,6,3}$ from Figure~\ref{fig:hr}  produces two-sets 
 \[
(5,1), (4,2), (3,5) ,(2,6) , (1,4) ,(6,3) 
 \]
The first resolved crossing of $K$ produces edges (5,1) and (4,2).  Starting with (5,1) and following along the resolution, we obtain the encounter the edge (5,3), then (3,6), then (6,2),   (2,4), then (4,1), and back to the initial edge (1,5).  Here, we have written the unordered 2-tuples in the order in which the label is encountered while traversing the knot.  The lowest edge label encountered is 1, so we put a 1 in the first bit of $\ket{L_{110}}$.  

Since the second edge (4,2) from the first crossing already appeared in the list, we move on to the next crossing.   In this case, all edges from the remaining crossings are already in the first list, so we are done, and $\ket{L_{110}}=\ket{100000}$.  
\end{example}

\begin{example}
The state $\ket{100}\ket{K_{100}} = \ket{100}\ket{ 5,1, 4,2 \mid 3,6,2,5 \mid 1,4,6,3 }$ from \cref{hopfpic-100} produces two-sets
\[
(5,1), (4,2),  (3,6), (2,5), (1,4), (6,3)
\]
Starting with the first edge of the first crossing (1,5) we have a contraction with (5,2), then with (2,4), then with (4,1) and we are back to the initial edge (1,5), where we have again listed the unordered 2-tuples in the order that the edge labels are encountered while traversing the resolution.  Since 1 is the lowest edge label encountered, we add a 1 to the first bit of $\ket{L_{100}}$.  The edge (4,2) appeared in this first loop so we move on to the edges in the next crossing.  The edge (3,6) contracts with (6,3) to produce another closed loop with minimal edge label 3.  All edges have now appeared so the output is $\ket{L_{100}}=101000$.
 
\end{example}

This is an efficient reversible classical algorithm, thus, in particular, also an efficient quantum algorithm.

\subsection{Enhanced state register}

\label{sec:enhanced_state}
 
The final data needed to encode the enhanced state $\ket{e}$ is a $2m$-qutrit register $\ket{s}=\ket{s_1 \dots s_{\ell(r)}}$ that encodes the enhanced state of each loop. We now describe an algorithm that computes $\ket{s}$ given $\ket{K_r}$ and $\ket{L_r}$, and thus encodes an enhanced state as
\begin{equation}
    \ket{e} = \ket{r}\ket{K_r}\ket{L_r}\ket{s},
\end{equation} where $r$ has length $m$, $K_r$ has length $4m\log(2m)$, $L_r$ has length $2m$  and determines how many loops are in the resolution and the minimal edge label of each loop. 

Let $j_1$, \dots, $j_{\ell(r)}$ indicate the positions of the $\ell(r)$ entries of $\ket{L_r}$ that are equal to 1. For the $\ell(r)$ qutrits corresponding to the loops, we let the value of $s_{j_i} = 0,1$ encode the state of the loop with minimal edge label $j_i$, where $0$ corresponds to what is called $\1$ in the notation of conventional Khovanov homology and $1$ corresponds to $X$ in that notation (not to be confused with Pauli $\1$ and $X$). The remaining $2m - \ell(r)$ qutrits are placed in the state $\ket{\perp}$.

  The registers $\ket{K_r}$ and $\ket{L_r}$ were used simply to compute the form of $\ket{s}$.  At this point we uncompute $\ket{K_r}$ and $\ket{L_r}$ such that an enhanced state is succinctly encoded by \begin{equation}
    \ket{e} = \ket{r}\ket{s}.
\end{equation}

\subsection{Boundary operators for Khovanov homology} \label{subsec:Kh-boundary}

We now study the (co-)boundary operator $\del_{ij}: C_{i,j} \to C_{i+1,j}$. To recapitulate, we describe an enhanced state $\ket{e}$ by the encoding \begin{equation}
    \ket{e} = \ket{r}\ket{s},
\end{equation} 
where $r$ describes the resolutions,  
and $s$ encodes the labeling of each loop. For simplicity, we will first study the full boundary operator \begin{equation}
    \del = \bigoplus_{ij} \del_{ij},
\end{equation} acting on $\bigoplus_{ij} C_{ij}$. The partial boundary operator can be recovered by  
a projection on the subspace with Hamming weight (=homological degree) $i$ and quantum degree $j$ (cf. \Cref{sec:bg_khovanov}). 

To keep track of the alternating signs in the boundary operator, we use the `fermionization' trick introduced in \cite{IBMresult}. This trick is based on the observation that the homology boundary operator can be succinctly written in terms of fermionic annihilation operators. To this end, let us introduce familiar notation from fermionic quantum mechanics. The \emph{k-th Jordan Wigner creation operator} $a_k^\dagger$ is an operator on $m$ qubits acting as \begin{equation}
a_k^\dagger = \sigma^z_1 \tens \dots \tens \sigma^z_{k-1} \tens \ket{1}\bra{0}_k \tens \mathbb{1}_{k+1} \tens \dots \tens \mathbb{1}_m.
\end{equation} The $\sigma_z$ operators are there to account for the correct alternating sign of fermionic annihilation operators. We need the same alternating sign for the Khovanov boundary operator here. With this definition, we can write the full boundary operator as \begin{equation}
    \del = \sum_{k=1}^m a_k^\dagger \tens
\Gamma_{k}.
\end{equation} 
The operator $a_k^\dagger$ flips the k-th bit of $\ket{r}$ from $0$ to a $1$ and introduces a corresponding alternating minus sign. $\Gamma_{k}$ is responsible for mapping the part $\ket{s}$ to the correct $\ket{s'}$, and defined via the following rules: 
\begin{itemize}
 
    \item If the action of $a^\dagger_k$ on $\ket{r}$ results in the merging of two loops labeled by $i,j$ into one loop labeled by $\min(i,j)$, then $\Gamma_k$ acts on $\ket{s}$ as \begin{equation}
        \Gamma_k^-(i,j) = \ket{1}_{\min(i,j)}\ket{\perp}_{\max(i,j)}\bra{1}_i\bra{1}_j +  \ket{0}_{\min(i,j)}\ket{\perp}_{\max(i,j)}\bra{1}_i\bra{0}_j +  \ket{0}_{\min(i,j)}\ket{\perp}_{\max(i,j)}\bra{0}_i\bra{1}_j.  
    \end{equation}

    \item If the action of $a^\dagger_k$ on $\ket{r}$ splits a loop labeled by  $\min(i,j)$ into two loops labeled by $i,j$, then $\Gamma_k$ acts on $\ket{s}$ as \begin{equation}
        \Gamma_k^+(i,j) = \left(\ket{1}_i\ket{0}_j + \ket{0}_i\ket{1}_j\right)\bra{1}_{\min(i,j)} \bra{\perp}_{\max(i,j)}+\ket{0}_i\ket{0}_j\bra{0}_{\min(i,j)} \bra{\perp}_{\max(i,j)}.
    \end{equation} Note that $\Gamma_k^+ = X^{\tens\ell_{max}} (\Gamma_k^-)^\dagger X^{\tens\ell_{max}},$ where $X$ is the Pauli-X operator actings on the first two states $\ket{0}, \ket{1}$ of the qutrit. 
\end{itemize}
The boundary operator above changes only one qubit in $\ket{r}$, 
and one qutrit in $\ket{s}$. It is thus local and can be efficiently exponentiated using standard techniques. In fact, we only need to exponentiate the operator $B=(\del+\del^\dagger)$, since $\Delta = B^2$, where \begin{equation}
    B = \sum_{k=1}^m B_k = \sum_{k=1}^m \left(a_k^\dagger \tens \Gamma_{k} + a_k\tens\Gamma_k^\dagger\right).
\end{equation} 
We can now make use of the fact that $\Gamma_k^+ = X^{\tens\ell_{max}} (\Gamma_k^-)^\dagger X^{\tens\ell_{max}}$
. Define $G = \Gamma_k^-.$ Then $B_k$ swaps the $k$-th bit $r_k$ of $\ket{r}$, 
and acts on $\ket{s}$ up to minus signs as 
\begin{itemize}
    \item  $G$ if $r_k = 0$ and two loops are merged.
    \item $X^{\tens\ell_{max}} G^\dagger X^{\tens\ell_{max}}$ if $r_k = 0$ and one loop is split. 
    \item $X^{\tens\ell_{max}} G X^{\tens\ell_{max}}$ if $r_k = 1$ and two loops are merged.
    \item $G^\dagger$ if $r_k = 1$ and one loop is split. 
\end{itemize} 
The $X$ operations can be efficiently implemented as controlled $X$ gates conditioned on $r_k$.

\subsection{The algorithm} \label{subsec:Kh-algorithm}
All the pieces are now in place to extend the quantum homology algorithm \cite{lloyd2016quantum} to Khovanov homology.  The previous sections have shown how to implement the boundary operator $B$ for a knot in the form needed to apply the Hamiltonian evolutions $e^{-iBt}$ and $e^{-i\Delta t}$ efficiently. In particular, this ability allows us to perform the quantum phase estimation algorithm \cite{kitaev1995quantum} to project an initial state onto the kernel of $B, \Delta$.  

The original homology algorithm for TDA works by preparing a maximally mixed state and then performing the projection onto the kernel of $B,\Delta$. Since in the case of Khovanov homology, the dimension of $\Delta$ grows exponentially in the number of crossings in the knot, the method of \cite{lloyd2016quantum} will only succeed if the Betti numbers are also exponentially large in the number of crossings. We explain this in more detail in \Cref{sec:thermalization}. However, we have found -- by exhaustive calculation of the Betti numbers for knots of up to eleven crossings -- that the Betti numbers tend to remain small.   Accordingly, we need to alter the algorithm so that it functions for low Betti numbers.

To cope with low Betti numbers, we treat  $\Delta$ as a physical Hamiltonian and invoke quantum thermalization and Gibbs sampling algorithms to obtain an approximate low-temperature thermal state \cite{BrandaoThermal21, SommaThermal22, GilyenThermal23a, GilyenThermal23, GilyenThermal24, LinLinThermal24}.    As long as the thermalization results in a state with inverse polynomially large overlap with the kernel of $B,\Delta$, then the projection onto the kernel succeeds in polynomial time.   
Because of the computational complexity results derived in the next section, we do not expect such thermalization to succeed for all knots (because we do not expect to be able to solve $\sharpP$ hard problems); but when we look at randomly selected knots of up to eleven crossings, we have not been able to identify obvious obstructions to thermalization via the methods proposed in \cite{BrandaoThermal21, SommaThermal22, GilyenThermal23a, GilyenThermal23, GilyenThermal24, LinLinThermal24}.   In particular, the Khovanov Hodge Laplacian does not obviously exhibit the bottlenecks to thermalization recently derived in \cite{ZlokapaThermal24}.

To compute $\beta_{ij}$ we do the following: 
\begin{itemize}
    \item Using the ability to perform $e^{-i\Delta t}$ combined with the ability to implement arbitrary sparse thermalizing Lindbladians, we prepare an approximate low temperature Gibbs state $e^{- \Delta/kT}$ over the Hilbert space of all enhanced states with bi-grading $(i,j)$.  
    
    \noindent{\em Comment:} In order to have at least polynomially small overlap with the kernel, we require $kT$ to on the order of the gap of $\Delta$\footnote{The precise value of $T$ that is required depends on the density of states of the Hodge Laplacian. See our discussion and numerics at the end of section \Cref{sec:numerics_gap} for more details.}. In subsection 6 below, we give evidence by exhaustive computation for all knots up to 10 crossings, that the gap of $\Delta$ for random knots does not shrink with the number of crossings $m$.  In addition, the minimum gap of $\Delta$ over all knots with $m$ crossings seems to decrease only polynomially in $m$.  
    
    \item Exponentiate the Laplacian, perform quantum phase estimation, measure, and check whether you found a zero eigenvalue or not.   
    This projection succeeds after a number of trials equal to one over the overlap of the approximate thermal state with the kernel.   Repeat this procedure multiple times to obtain multiple copies of the state $I/\beta$, the fully mixed state on the kernel of $\Delta$.   

    \noindent{\em Comment:} To resolve the zero eigenvalue from the next larger eigenvalue of $\Delta$, we must run the quantum phase estimation algorithm (using the improved version given in \cite{HHL}) for a time
    of $O(1/{\rm gap}(\Delta))$.   

    \item Perform a {\em SWAP} test on pairs of copies of $I/\beta$.   The {\em SWAP} test succeeds with probability $1/2 + 1/2\beta$.    As long as $\beta$ is polynomially large, then the time it takes to estimate $\beta$ to the desired accuracy grows only polynomially with $m$.

    \noindent{\em Comment:} Note that the pre-thermalized quantum homology algorithm described here works only when the Betti numbers are not exponentially large, whereas the original quantum homology algorithm, which started from an infinite temperature thermal state, worked only when the Betti numbers {\em were} exponentially large.
\end{itemize}
At the end of this procedure, we have arrived at an additive approximation to $\beta_{ij} = \dim Kh^{i,j}(K)$.

\subsection{Performance analysis } \label{subsec:Kh-performance}
\noindent The performance and error scaling of the above Khovanov homology algorithm can be calculated using optimal existing scalings for the different components of the algorithm.    

First, we analyze the computational complexity of exponentiating the boundary map/combinatorial Laplacian.     To perform the unitary $e^{-iBt}$ to accuracy $\epsilon$ we use the optimal state-of-the-art algorithm of Low and Chuang \cite{LowChuang} for sparse density matrix exponentiation based on quantum signal processing: given an $s$-sparse matrix, $H$, one can implement $e^{-iHt}$ to accuracy $\epsilon$ with query complexity $O(t s \|H\|_{max} + \log(1/\epsilon)/ \log \log (1/\epsilon))$.

The Hermitian boundary map $B = \partial + \partial^\dagger$ is $O(m)$ sparse.   We showed in the previous section how, given a state with a row index, we can construct a list of the locations of the $O(m)$ non-zero entries in that row, and their values $\pm 1$.    This subroutine enacts the query call in \cite{LowChuang}: each query requires $O(m)$ operations.  $B$ is a sparse matrix with $O(m)$ entries $\pm 1$ in each row/column.  Accordingly, $\|B\|_{max} = O(m)$.   Applying the Low and Chuang method we see that we can perform $e^{-iBt}$ to accuracy $\epsilon$ in time
$O(t m^2 + \log(1/\epsilon)/ \log \log (1/\epsilon))$.

As discussed in section 6 below, we assume that we are able to use the ability to apply $B$ as a Hamiltonian, together with the ability to couple this Hamiltonian system with a thermal environment to enact an efficiently thermalizing Lindbladian.   We assume that this thermalization process  efficiently drives the system to an approximate Gibbs state with temperature
$T = O(1/\delta)$, where $\delta$ is the gap of $\Delta$, in time $t_{therm}$. We stress that we do not expect that efficient thermalization is possible for all knots, but we have not been able to identify obvious obstructions for typical knots. 

To project onto the kernel of $\Delta$ we must run the quantum phase estimation algorithm for a time $t > 1/\delta$,   For the purposes of matrix exponentiation and projecting onto the kernel, it is more efficient to use the Hermitian boundary operator $B$ instead of $\Delta = B^2$.   In this case, we need to run the quantum phase estimation for time $O(1/\sqrt \delta)$ to resolve the kernel of $B$.

Combining these results, we find that the computational complexity of performing the quantum phase estimation to project onto the kernel of $B,\Delta$ is $O(m^3 \delta^{-1/2} t_{therm})$, plus additive terms depending on $\log \epsilon$  (We note that the results of the IBM South Africa group \cite{IBMresult} might improve the $m$ scaling here to $m^{3/2}$.)    

Having prepared multiple copies of the fully mixed state on the kernel, we now perform {\em SWAP} tests to estimate the Betti numbers $\beta$.   Each {\em SWAP} test takes $O(m)$ operations.  As noted above, the {\em SWAP} test succeeds with probability 
$1/2 + 1/2\beta$.    Repeating the {\em SWAP} test enough times to resolve $1/\beta$ to accuracy $\epsilon$ takes $O(m/\epsilon^2)$ operations. (Quantum counting might be able to reduce this scaling to 
$O(m/\epsilon)$). To exactly compute $\beta$ we need to estimate $1/\beta$ to accuracy $\epsilon = \Theta( \frac{1}{\beta^2})$. 

Combining the computational complexity of Hamiltonian simulation, quantum phase estimation, and the {\em SWAP} test, we obtain an overall computational complexity for the algorithm of
\begin{equation}
O(m^4 \beta^4 t_{therm}/ \delta^{1/2} ),
\label{eq:runtime}
\end{equation}
where the exponents might be reduced by the methods mentioned above.  
We provide numerical and analytical evidence in \Cref{sec:spectral_gaps,sec:numerics_gap,sec:graph} that the spectral gap only scales polynomially as $\delta^{-1} = \text{poly}(n)$. The dominating contribution to the runtime of this Khovanov homology quantum algorithm is thus the thermalization time $t_{therm}$.

\section{The computational complexity of Khovanov homology}
\label{sec:complexity}
In this section, we study the computational complexity of additively approximating the Betti number $\beta_{i,j}(K)$ of Khovanov homology at bidegree $(i,j)$ of a knot $K$, given as input an efficient description of the knot (such as a braid diagram) and the integers $(i,j)$. 

The fact that Khovanov homology categorifies the Jones polynomial, and that the Jones polynomial is only specified by polynomially many Betti numbers, implies that estimating Khovanov homology is at least as hard as estimating the Jones polynomial. We now make this connection explicit to state the hardness of several increasingly tight approximations to the ranks of Khovanov homology. We then discuss whether these lower-bounds are tight and several other open questions.
\subsection{Preliminaries}
\label{sec:complex_prelim}
For a detailed explanation of the complexity classes used in this article, see e.g. \cite{bernstein1993quantum,watrous2009quantum}. In brief, though, recall that the complexity class $\mathsf{P}$ is (informally) the set of decision problems that can be solved in polynomial time by a classical computer, and $\NP$ is the set of decision problems for which a solution can be classically verified in polynomial time. The complexity class $\sharpP$ is the set of counting problems associated to $\NP$. That is, whereas a problem in
$\NP$ asks whether a given instance has a solution or not,
the corresponding problem in $\sharpP$ asks how many solutions
the instance has. If a problem is at least as hard as any problem in a complexity class $C$, we say the problem is $C$-hard. If the problem is additionally contained in $C$, we call it $C$-complete.

These classical complexity classes have quantum analogues. The quantum version of $\mathsf{P}$ is $\BQP$, the class of all decision problems that can be solved in polynomial time by a quantum computer. $\QMA$ is the class of decision problems for which a quantum computer can efficiently verify a solution. We will also refer to the class $\DQC$ (or \textbf{d}eterministic \textbf{q}uantum \textbf{c}omputation with \textbf{one} clean
qubit) which is a restricted model of quantum computation believed to be strictly in-between $\mathsf{P}$ and $\BQP$. 

No known quantum algorithm solves an $\NP$-complete problem in polynomial time and it is conjectured that $\NP \not \subset \BQP$, i.e., that $\NP$-hard problems are not accessible to quantum computers. The same
holds for $\sharpP$-hard problems.
\subsection{Hardness results}
\label{sec:complexity_lower}
Throughout this subsection, we always consider a braid diagram $b$ with $n$ strands and $m$ crossings, where $m$ is polynomial in $n$. 
We consider the problem of additively approximating the Betti number $\beta_{ij}(K)$ of Khovanov homology at bidegree $(i,j)$ of a knot $K$, given a description of the knot and the integers $(i,j)$ as input. 
The following theorems are corollaries of corresponding hardness results for estimating the Jones polynomial due to Freedman et al. \cite{freedman2002simulation}, Aharonov et al. \cite{Aharonov-Jones-Landau}, Kuperberg \cite{kuperberg2015hard}, Shor et al. \cite{shor2007estimating}, and Aharonov et al. \cite{aharonov2011bqp}. 
\begin{theorem}
    Let $b$ be a braid diagram with $n$ strands and $m = \text{poly}(n)$ crossings. It is $\DQC$-hard to estimate the Betti numbers of Khovanov homology of the trace closure of $b$ up to additive error $\epsilon |d|^{n-1}$ for $\epsilon = \frac{1}{\text{poly}(n)}$, where $d = 2 \cos(\pi/5) \approx 1.62$.
    \label{thm:dqc1hard}
\end{theorem}
\begin{proof} Our proof is a reduction from estimating the Jones polynomial at a root of unity to the estimation of Betti numbers of Khovanov homology. Recall from \Cref{sec:bg_khovanov} that the Kauffman bracket $\langle  K  \rangle$ of the trace closure of $b$ is 
\begin{equation}
     \langle  K  \rangle =\sum_j q^j \sum_{i}(-1)^i \beta_{ij},
\end{equation}
where $i$ and $j$ are the homological and quantum degrees, respectively. 
An $\epsilon$-approximation to the Betti numbers of Khovanov homology thus yields a $O(m^2\epsilon)$-approximation to the Kauffman bracket at a root of unity, since \begin{equation}
\label{eq:kauffman_approx}
    |\langle \Tilde K  \rangle - \langle K \rangle | =|\sum_j q^j \sum_{i}(-1)^i (\Tilde{\beta}_{ij} -\beta_{ij}) | \leq \sum_{i,j}|\Tilde{\beta}_{ij} -\beta_{ij} | \leq (2m^2+5m+3) \epsilon.
\end{equation}
Here we have used that the homological degree $i$ is bounded by $[0,m]$ and the quantum degree $j$ by $[-m-1,+m+1]$, and that $q$ is a root of unity whenever $t$ is. Up to polynomial factors in $m$, estimating Khovanov homology is thus at least as hard as estimating the Kauffman bracket at a root of unity. 

Jordan and Shor \cite{shor2007estimating} showed that it is $\DQC$-hard to estimate the Jones polynomial to additive accuracy $\epsilon d^{n-1}$ for $\epsilon = \text{poly}(n,m,k)$ at the fifth root of unity $t= \exp^{2\pi i/5}$.  The theorem then follows from the fact that at a root of unity, the Kauffman bracket differs from the Jones polynomial only by a phase.
\end{proof}

Next, we show that this task becomes $\BQP$-hard if we target a smaller additive error $\epsilon |d|^{\frac{n}{2}-1}$ for the \emph{plat closure} of a braid.

\begin{theorem}
    Let $b$ be a braid diagram with $n$ strands and $m=\text{poly}(n)$ crossings. It is $\BQP$-hard to estimate the Betti numbers of Khovanov homology of the plat closure of $b$ up to additive error $\epsilon |d|^{\frac{n}{2}-1}$ for $\epsilon = \frac{1}{\text{poly}(n)}$, where $d = 2 \cos(\pi/5) \approx 1.62$.
    \label{thm:bqphard}
\end{theorem}
\begin{proof}
    As in the proof of \Cref{thm:dqc1hard}, our proof is a reduction from the task of estimating the Jones polynomial at a root of unity. To allow comparison with \Cref{thm:dqc1hard}, we fix again $t$ to be the fifth root of unity. In \cite{aharonov2011bqp}, it was shown that estimating the Jones polynomial of the plat closure of $b$ at the fifth root of unity up to additive accuracy  $\epsilon |d|^{\frac{n}{2}-1}$ is $\BQP$-hard, for $\epsilon = \frac{1}{\text{poly}(n)}$. The theorem statement then follows from \cref{eq:kauffman_approx} and the fact that for $t$ a root of unity, $q$ is a root of unity, and moreover, the Kauffman bracket differs from the Jones polynomial only by a phase. 
\end{proof}

Note that the trace closure of a braid on $n$ strands can always be written as the plat closure of a braid on $2n$ strands. It is also easy to see that estimating Khovanov homology exactly is $\sharpP$-hard, since the exact computation of the Jones polynomial reduces to it. The $\sharpP$ hardness of the exact Jones polynomial computation is shown in \cite{kuperberg2015hard, Aharonov-Jones-Landau}. Hence increasingly tight additive approximations to the Betti numbers of Khovanov homology are $\DQC$-hard, $\BQP$-hard, and $\sharpP$-hard, respectively. Via our quantum algorithm for Khovanov homology and its runtime in \cref{eq:runtime}, this implies corresponding hardness results for the preparation of Gibbs states of the Hodge Laplacian of Khovanov homology.

\subsection{Open questions}
 
The above results highlight the hardness of approximating Khovanov homology in various regimes. The hardness of Betti numbers of Khovanov homology seems analogous to previously derived results for the case of simplicial complexes. There, certain additive approximations are also known to be $\DQC$-hard \cite{quantumAdvantage,cade2021complexity}, while multiplicative approximations are $\NP$-hard \cite{adamaszek2016,schmidhuber2023complexity} and in fact even $\QMA_1$-hard \cite{crichigno2022clique,king2023promise}, and exact computation is $\sharpP$-hard \cite{schmidhuber2023complexity,crichigno2022clique}. However, the hardness results derived so far leave several important questions open.

\paragraph{$\QMA$-hardness} Quantum algorithms for homology generally encode the generators of homology in the ground state space of a sparse Hamiltonian (the Hodge Laplacian). Estimating the ground state energy of a sparse or local Hamiltonian is a canonical $\QMA$-complete problem \cite{kitaev2002classical,bookatz2012qma}. Indeed, even for the special kind of Hamiltonians arising as combinatorial Laplacians of (weighted) clique complexes, deciding whether its Betti numbers are nonzero is known to be $\QMA_1$-hard and contained in $\QMA$ \cite{crichigno2022clique,king2023promise}. We expect that the same hardness results extend to the setting of Khovanov complexes in certain parameter regimes, but proving this rigorously seems to require the development of new perturbative gadgets that encode arbitrary $k$-local Hamiltonians in the ground state of the Hodge Laplacian of Khovanov homology, which is an interesting direction for further research.

\paragraph{The complexity of normalized Betti numbers}
In the absence of other issues (like a small spectral gap), the natural output of the original quantum algorithm for homology (without thermalization) is an approximation of \emph{normalized} Betti numbers, that is, they estimate $\beta_{ij}/ |C_{ij}|$ up to additive accuracy $\epsilon$ in time poly$(m,n,\frac{1}{\epsilon})$. Here, $|C_{ij}|$ is the dimension of the chain space at bidegree $(i,j)$. Our hardness results above also apply to the approximation of normalized Betti numbers, however, there the normalization constant is a power of $|d| = |q+q^{-1}|$. It is an interesting open question to directly derive hardness results for the estimation of normalized Betti numbers of Khovanov homology, with normalization factor  
\begin{equation}
  \dim C_{ij} =   \sum_{r : |r|=i}  {\ell(r) \choose \frac{j-i+\ell(r)}{2}},
\end{equation} which is also exponential in $m$ in general. Here, $\ell(r)$ is the loop number at resolution $r$. Indeed, the related problem for general simplicial complexes is known to be $\DQC$-hard \cite{quantumAdvantage,cade2021complexity}, but it is an open question to establish analogue results for the Khovanov complex of a knot.

\paragraph{Containment}
We have so-far only discussed lower bounds on the hardness of Khovanov homology. In \Cref{sec:algorithm}, we show that estimating Khovanov homology is in $\BQP$ \emph{provided} the Hodge Laplacian thermalizes efficiently and the spectral gap is sufficiently large. We have good numerical and analytical evidence for the latter, but do not expect that thermalization is efficient in general. Moreover, assuming the spectral gap of the Hodge Laplacian is at least inverse-polynomial in the number of crossing, it is easy to see that checking whether a specific Betti number of Khovanov homology is zero or non-zero is in $\QMA$: If the Betti number is non-zero then any quantum state in the kernel of the Hodge Laplacian acts as a witness. This witness can be efficiently verified by running quantum phase estimation up to inverse-polynomial accuracy. 

More generally, one can ask whether the lower bounds as stated in \Cref{sec:complexity_lower} are tight. We expect that they can indeed be strengthened. This is because they were derived by reducing to the Jones polynomial. Many knots are distinguished by Khovanov homology that are indistinguishable by their Jones polynomial, and it is reasonable to expect that Khovanov homology is thus harder to approximate. More generally, Khovanov homology is encoded in the ground state space of a Laplacian, whereas the Jones polynomial is encoded in the expectation value of local observables, and it is well known that estimating the latter can be often much easier than estimating the former.

\paragraph{Persistent Khovanov homology} Another interesting open question concerns homological persistance. 
Recently, the notion of persistent Khovanov homology has been introduced \cite{liu2024persistent}. Even more recently, it has been shown \cite{gyurik_schmidhuber_2024quantum} that a problem closely related to determining the persistence of clique homology is tightly linked to quantum mechanics: The problem is $\BQP_1$-hard and contained in $\BQP$, implying an exponential quantum speedup for this task. It is an interesting direction for further research to see whether this result extends to the case of Khovanov homology. 

\section{Spectral gaps and homological perturbation theory}
\label{sec:spectral_gaps}
 
A key contribution to the runtime of the quantum algorithm for Khovanov homology is the spectral gap $\delta$ of the Hodge Laplacian $\Delta$. Importantly, the spectral gap is a quantity associated with a knot diagram and not directly with a knot itself, and therefore not topologically invariant. 

In this section, we investigate how changing a knot diagram can impact the spectral gap in computations of Khovanov homology. We start by reviewing some homological algebra and collecting some general results on combinatorial Hodge theory in \Cref{subsec:Results-comb}.  Then, in \Cref{subsec:Hom-pert}, we introduce the important notion of homological perturbation theory that explains how maps between chain complexes induce maps between harmonic chains of their corresponding Laplacians.   

Knot diagrams that differ by applying a Reidemeister move are related in Khovanov homology by chain maps between their corresponding complexes.  In \Cref{prop:gap-inequality}, we examine special kinds of chain maps where the spectral gap is guaranteed not to increase.  In Proposition~\ref{prop:R1} we apply this result to study how the spectral gaps of Khovanov homology can change under the first Reidemeister move.  This analysis shows that adding trivial twists to a complex can decrease the spectral gap.   In \Cref{subsec:twisted}, we investigate a case study of twisted unknots to examine how adding topologically trivial twists to the unknot can decrease the spectral gap.  Importantly, our numerical investigations suggest that while the gap for twisted unknots keeps decreasing the spectral gap with an increase in the number of twists, the decrease in the spectral gap is only inverse polynomial in the number of crossings.  In \Cref{sec:graph}, we make these numerical observations rigorous through analytic bounds on the spectral gap for this class of examples.  

Finally, in \Cref{subsec:increase}, we show how the technique of Gaussian elimination from homological algebra can be used to dramatically increase the spectral gap.

\subsection{Results on combinatorial Hodge theory} \label{subsec:Results-comb}

Here, we collect some general results on combinatorial Hodge theory expressed using cohomological notation.  Recall our conventions for (co)chain complexes from \Cref{sec:chain}.  The results in this section hold over any field of characteristic 0, but since we are interested in quantum algorithms, we present them over the complex numbers $\mathbb{C}$.

Assume that the complex $\{ C^i , d_i\}$ of finite-dimensional complex vector spaces
is such that  each chain space $C^i$ is endowed with an inner product $\langle,\rangle$.  We will sometimes call $c$ an $i$-form if $c\in C^i$. Then the adjoint $d_i^{\dagger} \maps C^{i+1} \to C^{i}$ is defined by 
\[
\langle d_{i}(c), c' \rangle = \langle c, d_i^{\dagger}(c')\rangle 
\]
for $c\in C^{i}$ and $c' \in C^{i+1}$.    The differentials and their adjoints give rise to a (combinatorial) Laplacian
\begin{equation}
\Delta_i := d_i^{\dagger} d_i +  d_{i-1} d_{i-1}^{\dagger}.
\end{equation}
We call an $i$ form \emph{harmonic} if $ \Delta_ic = 0$.  
Define for each $i$ a vector subspace of $C^i$ consisting of \emph{harmonic $i$-forms}  
\begin{equation}
    \mathcal{H}^i  := \{ c \in C^i \mid \Delta_ic = 0\}. 
\end{equation}

It is easy to see that 
\begin{itemize}
 \item $d_i^{\dagger}d_i$ annihilates $\mathcal{H}^i$ and $\Im(d_{i-1})$ and preserves the subspace $\Im(d_i^{\dagger})$ of $C^i$,  
 
\item $d_{i-1} d_{i-1}^{\dagger}$ annihilates  $\mathcal{H}^i$ and  $\Im(d_i^{\dagger})$ and preserves the subspace $\Im(d_{i-1})$ of $C^i$,
\end{itemize}

We can then define a Hodge theory associated with this Laplacian following  \cite[Proposition 2.1 \& 2.2]{MR1622290}.  
\begin{theorem}[{\bf Hodge Theory}]  \label{thm_Hodge}
Given a complex $\mathbf{C}$ equipped with an inner product, each cochain space $C^i$ decomposes as 
\begin{equation} \label{eq:Cidecomp}
    C^i \cong \mathcal{H}^i \oplus \Im(d_{i-1}) \oplus \Im(d_i^{\dagger}).
\end{equation}
Furthermore, there is an isomorphism $\mathcal{H}^i \cong H^i(\mathbf{C})$ so that each class in homology is represented by a unique harmonic chain in $\mathcal{H}^i$. 
\end{theorem} 

In particular, Hodge theory implies $\Delta_i$ is positive definite on the space
\[
 (\mathcal{H}^i)^{\perp} =  (H^i(\mathbf{C}))^{\perp} =  \Im(d_{i-1}) \oplus \Im(d_i^{\dagger})  ,
\]
and the smallest eigenvalue of $\Delta_i$ on  $(\mathcal{H}^i)^{\perp}$ is the \emph{spectral gap} $\delta$.

\begin{proposition} 
 The decomposition in \cref{eq:Cidecomp} has the following properties:
 \begin{itemize}
\item The set of (non-zero) eigenvalues of $\Delta_i$ on  $(\mathcal{H}^i)^{\perp}$ is a subset of the union of those of $\Delta_{i-1}$ on $(\mathcal{H}^{i-1})^{\perp}$ and of $\Delta_{i+1}$ on $(\mathcal{H}^{i+1})^{\perp}$.

     \item Every eigenvalue of of $\Delta_i$ occurs as an eigenvalue of either $d_i^{\dagger}d_i$ or $d_{i-1} d_{i-1}^{\dagger}$.
\end{itemize}
\end{proposition}

\begin{proof}
The first claim is proven in \cite[Proposition   2.2]{MR1622290}.  For the second observe that the positivity of the operator $\Delta_k$, implies that $\Delta_k\ket{v} = 0$ if and only if we have both $d_k^{\dagger}d_k\ket{v}=0$ and $d_{k+1} d_{k+1}^{\dagger}\ket{v}=0$.
Suppose then that for $\lambda\neq 0$ we have $\Delta_k \ket{v} = d_k^{\dagger}d_k\ket{v} + 
d_{k-1} d_{k-1}^{\dagger}\ket{v}=\lambda \ket{v}$.  Then 
\begin{align}
d_k^{\dagger}d_k \left(\ket{v} - \frac{1}{\lambda} d_{k-1} d_{k-1}^{\dagger}\ket{v} \right)
&= 
d_k^{\dagger}d_k  \ket{v} - \frac{1}{\lambda} d_k^{\dagger}d_k d_{k-1} d_{k-1}^{\dagger}\ket{v}  \nonumber \\
&=
d_k^{\dagger}d_k  \ket{v} + 0\nonumber \\
&=
   \lambda \ket{v}  -d_{k-1} d_{k-1}^{\dagger}\ket{v}\nonumber \\
   &=
   \lambda \left( \ket{v}  - \frac{1}{\lambda}d_{k+1} d_{k+1}^{\dagger}\ket{v}\right)
\end{align}
so that either $\lambda$ is an eigenvalue of $d_k^{\dagger}d_k$, or $\left( \ket{v}  - \frac{1}{\lambda}d_{k-1} d_{k-1}^{\dagger}\ket{v}\right)$ is the zero vector.
In the latter case, $d_{k-1} d_{k-1}^{\dagger}\ket{v} =\lambda\ket{v}$.
\end{proof}

Note that chain maps need not preserve the Hodge decomposition of a complex $\mathbf{C}$.  For a concrete example, see \cite[Remark 3.1]{Liu} or Proposition~\ref{prop:R1} below and the discussion in \Cref{subsec:twisted}.   

\subsection{Homological perturbation theory} \label{subsec:Hom-pert}
The aim of this section is to explain how chain maps between complexes induce maps on harmonic chains.  
First, we introduce some homological algebras motivated by constructions in homotopy theory. 

 Given $f,f' \maps \mathbf{C}\to \mathbf{D}$, we say that $f$ and $f'$ are \textit{chain homotopic} if there are maps $h_i \maps C_i \to D_{i+1}$ satisfying 
\[
f_i - f_i' = \partial_{i+1} \circ h_i + h_{i-1} \circ \partial_i.
\]
If $f$ and $f'$ are chain homotopic, written $f\simeq f'$, then they induce the same map on homology $f_{\ast} = f_{\ast}' \maps H_i(\mathbf{C}) \to H_i(\mathbf{D})$.  A similarly defined notion can be formulated for cochain maps. From the perspective of homology, two chain complexes $\mathbf{C}$ and $\mathbf{D}$ are equivalent if there are chain maps $f \maps \mathbf{C} \to \mathbf{D}$ and $g \maps \mathbf{D} \to \mathbf{C}$ so that $fg\simeq \Id_{\mathbf{D}}$ and $gf\simeq \Id_{\mathbf{C}}$.  In this case,  $\mathbf{C}$ and $\mathbf{D}$ will have isomorphic homology groups.

\begin{remark}
 If two knot diagrams of the same knot differ by a Reidemeister move~\cref{eq:Reidemeistere}, then the chain complexes assigned to the knot diagrams in Khovanov homology are chain homotopic.  In particular, the complex assigned to a knot in Khovanov homology is not an invariant of a knot; only the homology is a topological invariant. 
\end{remark}
 
A morphism of chain complexes $\pi \maps \mathbf{D} \to \mathbf{C}$ is called a \textit{strong deformation retract} if there is a chain map $i \maps \mathbf{D}  \to \mathbf{C}$ and a chain homotopy $K \maps \mathbf{C}  \to \mathbf{C}$ so that $ \pi i=I_{\mathbf{D}}$, $I_{\mathbf{C}}-i \pi=dK + Kd$, and $Ki=0$.  In this case, we say that $i$ is the \textit{inclusion in a strong deformation retract. } 
Given the data of a strong deformation retraction $(\mathbf{C},\mathbf{D},\pi, i, K)$, it suffices to assume that $Ki = \pi K = K^2 =0$.

Consider the chain complex $\mathcal{H}=(\mathcal{H}^i,0)$ consisting of the harmonic $i$-forms with trivial differential.
The Hodge decomposition
\[
 C^i = \Im(d_{i-1}) \oplus \Im(d_i^{\dagger}) \oplus \mathcal{H}^i
\]
applied to the chain spaces of the complex $\mathbf{C}=(C^i, d_i)$ gives rise to an inclusion chain map $i \maps \mathcal{H} \to \mathbf{C}$. 
By Theorem~\ref{thm_Hodge},  this is a quasi-isomorphism, meaning that the chain map $i \maps \mathcal{H} \to \mathbf{C}$  induces an isomorphism on cohomology $$i^{\ast} \maps H^i(\mathcal{H})=\mathcal{H}^i \to H^i(\mathbf{C}) .$$

Over a field, any quasi-isomorphism extends to a strong deformation retraction.  To see this define a projection chain map $p \maps \mathbf{C} \to \mathcal{H} $ by
\[
p_i( \Im(d_{i-1}))  =    p_i( \Im(d_i^{\dagger})) = 0, \quad   p_i(x) = x, \quad x \in \mathcal{H}^i. 
\]
Define a linear maps $K \maps C^i \to C^{i-1}$ that will be part of a chain homotopy $K \maps \mathbf{C} \to \mathbf{C}$ by
\[
K_i(\mathcal{H}^i) = K_i(\Im(d_i^{\dagger}) = 0, \quad K_i( d_{i-1} d_{i-1}^{\dagger}x) = d_{i-1}^{\dagger}x.
\]
Note that this completely determines $K$ since any nonzero element of $\Im d_{i-1}$ must be of the form $d_{i-1}^{\dagger}x$.  To simplify notation, we will omit indices when clear.  

In what follows, define $h = ip$, so that $h^2=h$.  The chain map $h$ acts by zero on all non-harmonic chains and the identity on harmonic chains. 

\begin{proposition}[Section 3 \cite{Liu}] \label{prop:defret}
The data $(\mathbf{C},\mathcal{H},i,p,K )$ defines a strong deformation retraction of $\mathbf{C}$ onto $\mathcal{H}$.    In particular, $pi= {\rm Id}_{\mathcal{H}}$, $K^2=Ki = pK =0$, and ${\rm Id}_{\mathbf{C}}-ip = Kd + dK$. 
\end{proposition}

\begin{proof}
This is proven in \cite[Section 3]{Liu} in the context of differential graded inner product spaces, but the proof carries over verbatim. 
\end{proof}

\begin{proposition}[Proposition 3.3 \cite{Liu}] \label{prop:Kdecomp}
There is a decomposition of the chain space of $\mathbf{C}$ given by ${\rm Id}_{\mathbf{C}}  = Kd + dK+ip$ so that 
\[
\mathbf{C} = dK(\mathbf{C})\oplus Kd(\mathbf{C}) \oplus h(\mathbf{C})
\]
where $h(\mathbf{C})  = \mathcal{H} = \ker d \cap \ker d^{\dagger}$ and this agrees with the Hodge decomposition, so that $dK(\mathbf{C})=d(\mathbf{C})$ and $Kd(\mathbf{C}) =  d^{\dagger}(\mathbf{C})$.
\end{proposition}

While it is not true that chain maps preserve the Hodge decomposition, it is shown that a modification using the structure in Proposition~\ref{prop:defret} makes it possible to get well-defined maps on harmonic chains from chain maps.

\begin{proposition} \label{prop:h-functor}
Let $h(\mathbf{C})  = \mathcal{H} = \ker d \cap \ker d^{\dagger}$ be defined as in Proposition~\ref{prop:Kdecomp}.  Given any chain map $f \maps \mathbf{C} \to \mathbf{D}$, then the map 
\begin{align}
h(f)\maps  h(\mathbf{C}) &\to h(\mathbf{D}) \\
x & \mapsto hf(x)
\end{align}
is well-defined and compatible with composition, in the sense that $h(g)\circ h(f) = h(g\circ f)$. 
\end{proposition}

\begin{proof}
The proof is identical to the proof in \cite[Lemma 3.4]{Liu} proven in the context of differential graded inner product spaces. 
\end{proof}

\begin{remark}
Proposition~\ref{prop:h-functor} can be interpreted as saying that $h$ is a functor from the category of chain complexes and chain maps into the category of vector spaces and linear maps.  In \cite[Proposition 3.5]{Liu}, it is shown that this functor is naturally isomorphic to the usual homology functor, giving a tight connection between the maps induced on homology by chain maps and those induced on harmonic chains via the above construction. 
\end{remark}

Chain maps do not generally preserve the Hodge decomposition even with the modifications above, but they do preserve the complement of harmonic chains.  Let $h^{\perp}=({\rm Id}_{\mathbf{C}} - h)$ and observe that $\Delta$ is positive definite on the space $h^{\perp}(\mathbf{C})$.

\begin{proposition}
Define 
\[
 h^{\perp}(\mathbf{C}) := d(\mathbf{C}) \oplus d^{\dagger}(\mathbf{C}).
\]
Given a map of chain complexes $f\maps \mathbf{C} \to \mathbf{D}$ define 
\[
h^{\perp}(f) \maps  h^{\perp}(\mathbf{C}) \to  h^{\perp}(\mathbf{D}).
\]
Then $h^{\perp}(f)$ is well defined and compatible with composition.  In other words, $h^{\perp}$ defines a functor from chain complexes to the category of vector spaces. 
\end{proposition}

\begin{proof}
This is similar to the proof of Proposition~\ref{prop:h-functor}.
\end{proof}

In the next section, we apply this theory to the study of Khovanov homology.  

\subsection{Behavior of the spectral gap under change of knot diagram} \label{sec:change-diagram}

An important part of the discussion of the spectral gap in Khovanov homology is that it is a quantity associated with a knot diagram rather than a knot itself. In particular, it is not a topological invariant of the knot. Applying Reidemeister moves \cref{eq:Reidemeistere} can alter the spectral gap. Here, we analyze how the spectral gap can change under Reidemeister moves using homological perturbation theory.

As we observed above,  a map of chain complexes does not need to preserve the eigenvectors of the Laplacian. However, we have the following.

\begin{proposition} \label{prop:gap-inequality}
Let $\mathbf{C}$ and $\mathbf{D}$ be chain complexes of finite dimensional Hilbert spaces. 
Suppose that $f \maps \mathbf{C} \to\mathbf{D}$ is the inclusion in a strong deformation retraction and that $f^{\dagger}$ is also a chain map ($d f^{\dagger} = f^{\dagger} d$.)  Then for any eigenvector $x$ of $\Delta^C$ with eigenvalue $\lambda$, $f(X)$ is an eigenvector of $\Delta^D$ with eigenvalue $\lambda$ and 
\begin{equation}
    {\rm gap}(\Delta^C) \geq {\rm gap}(\Delta^D).
\end{equation}
\end{proposition}

\begin{proof}
The first claim is immediate; since $f$ is a chain map, it commutes with the differentials in $\mathbf{C}$ and $\mathbf{D}$, and the assumption that $f^{\dagger}$ is a chain map implies $f$ also commutes with $d^{\dagger}$, and hence also with $\Delta$ for $\mathbf{C}$ and $\mathbf{D}$.  Since $f$ is the inclusion in a strong deformation retract, $f$ is injective on each chain space so that if $\Delta x = \lambda x$ for nonzero $x$, we have $\Delta f x = f\Delta x = \lambda f x$ with $fx\neq 0$.  

The last claim follows since every eigenvalue of $\Delta^C$ is an eigenvalue of $\Delta^D$, so that the gap for $\Delta^D$ is at most as large as that of $\Delta^C$.
\end{proof}

To apply these results, it is helpful to work in the local model of Khovanov homology described in \cite{BN2}, where one can consider portions of the Khovanov chain complex in a small neighborhood around a crossing.  Bar-Natan formulates this local version of Khovanov homology using the so-called picture world, which may seem quite abstract at first reading.  However, it is straightforward to convert these diagrammatically defined complexes into concrete complexes of vector spaces, as explained in \cite[Section 7]{BN2}.

In Bar-Natan's local proof of invariance Reidemeister move from \cite[Section 4.2]{BN2}, he defines chain maps between any Khovanov complexes obtained for knot diagrams that are identical except in a small neighborhood that looks like one portion of a Reidemeister move.  There he defines a chain map  
\begin{equation} \label{eq:F1}
F_1\maps    \left[
  \hackcenter{ \;\; \begin{tikzpicture}[scale=.5]
  \draw [black,very thick]    (1,1) to (1,2);
  \end{tikzpicture}}  \;\; \right]
  \to
 \left[\;\hackcenter{\begin{tikzpicture}[scale=.5]
 \draw [black, very thick]    (2,2) .. controls ++(.5,.25) and ++(.5,-.25) .. (2,1) .. controls +(-.25,.25) and +(0,-.25) .. (1,2);
  \path [fill=white] (1.25,1) rectangle (1.55,2);
  \draw [black,very thick]    (1,1) .. controls +(0,.25) and +(-.25,-.25) .. (2,2);
  \end{tikzpicture}} \;\right]
\end{equation}
that is an inclusions in a strong deformation retraction.  

\begin{proposition} \label{prop:R1}
The chain map $F_1$ 
from \cref{eq:F1} 
satisfy the hypothesis of Proposition~\ref{prop:gap-inequality} so that adding topologically trivial crossing via the Reidemeister one move to a knot diagram will never increase the spectral gap of the Laplacian, but may decrease the gap. 
\end{proposition}

\begin{proof}
Bar-Natan establishes that $F_1$ is an inclusion in a strong deformation retraction.  It remains to establish that $\partial^{\dagger} F_1= F_1 \partial^{\dagger}$, but this is trivially satisfied for $F_1$ since the source is a one-term complex and the target is a two complex. 
\end{proof}

\subsection{Twisted unknots} \label{subsec:twisted}
For low crossing numbers ($n <8$), the smallest spectral gaps we encountered come from trivial twists of the unknot.  We write $TU_n$ for the unknot with $n$-crossings obtained by repeatedly applying the Reidemeister one (R1) move to an unknot.  
\begin{equation} \label{eq:twists}
TU_n \;\; := \;\; 
\hackcenter{   \begin{tikzpicture}[scale=.6]
      \draw [ very thick]    (1,1) .. controls +(0,.35) and +(0,-.35) .. (2,2);
  \draw [ very thick]    (1,2) .. controls +(0,.35) and +(0,-.35) .. (2,3);
   \draw [ very thick]    (1,4) .. controls +(0,.35) and +(0,-.35) .. (2,5);
  \path [fill=white] (1.35,1) rectangle (1.65,5);
  \draw [ very thick]    (2,1) .. controls +(0,.35) and +(0,-.35) .. (1,2);
  \draw [ very thick]    (2,2) .. controls +(0,.35) and +(0,-.35) .. (1,3);
    \draw [ very thick]    (2,4) .. controls +(0,.35) and +(0,-.35) .. (1,5);
     \draw [ very thick]    (2,5) .. controls +(0,.5) and +(0,.5) .. (1,5);
     \draw [ very thick]    (2,1) .. controls +(0,-.5) and +(0,-.5) .. (1,1);
      \node at (1.5,3.75) {$\vdots$};
    \end{tikzpicture}  } 
    \qquad \qquad 
    \sigma^n \;\; := \;\; 
\hackcenter{   \begin{tikzpicture}[scale=.6]
      \draw [ very thick]    (1,1) .. controls +(0,.35) and +(0,-.35) .. (2,2);
  \draw [ very thick]    (1,2) .. controls +(0,.35) and +(0,-.35) .. (2,3);
   \draw [ very thick]    (1,4) .. controls +(0,.35) and +(0,-.35) .. (2,5);
  \path [fill=white] (1.35,1) rectangle (1.65,5);
  \draw [ very thick]    (2,1) .. controls +(0,.35) and +(0,-.35) .. (1,2);
  \draw [ very thick]    (2,2) .. controls +(0,.35) and +(0,-.35) .. (1,3);
    \draw [ very thick]    (2,4) .. controls +(0,.35) and +(0,-.35) .. (1,5); 
      \node at (1.5,3.75) {$\vdots$};
    \end{tikzpicture}  }
\end{equation}
Proposition~\ref{prop:R1} shows that applying an R1 move can reduce the spectral gap, and in the case of the twisted unknot, this is indeed the case.  Calculations of the spectral gap for $TU_n$ are among the most computationally intensive.  This is due to the fact that these knots saturate the upper bound for the maximal number of loops appearing in the resolutions.  Indeed, the all 0-resolution of $TU_n$ will have $n+1$ circles giving the maximum dimension for the Hilbert space among all $n$-crossing knots in all homological degrees.  

\begin{table}[ht]
\centering
\begin{tabular}{cc}
\hline
\(n\) & Min Gap of $TU_n$ \\ \hline
1 & 1 \\
2 & 0.585786 \\
3 & 0.381966 \\
4 & 0.267949 \\
5 & 0.198062 \\
6 & 0.152241 \\
7 & 0.120615 \\
8 & 0.097887 \\
9 & 0.0810141 \\
10 & 0.0681483 \\ \hline
\end{tabular}
\caption{The minimum spectral gap of twisted unknots $TU_n$ for different values of \(n\).  All of these minimal gaps occurred in bidgree $(0,3-n)$.  }
\label{table:min_gap}
\end{table}

Inspection of the minimal occurring spectral gap for $n\leq 10$ indicates that it appears in homological degree 0 and $q$ degree $3-n$.  In this bidegree, the form of the Laplacian can be computed directly.  In homological degree zero, $C^i(TU_n) = V^{\otimes (n+1)}$ and the $q$-degree $3-n$ subspace is spanned by $v_1 \otimes \dots v_{n+1}$ with all but two $v_{i}=X$ and exactly two terms equal to $\1$. 

Enumerate the basis elements of $V^{\otimes (n+1)}$ in $q$-degreee $3-n$ by two-tuples $\{a,b\}$ with $1\leq a<b \leq n+1$ indicating the location of the $\1$'s in the tensor product, so that the state $v_1 \otimes \dots v_{n+1}$ has all entries $X$, except at position $a$ and $b$.

\begin{proposition} \label{prop:Delta03n}
The Laplacian of $TU_n$ in bidegree $(0,3-n)$ is $\binom{n+1}{2}$-dimensional corresponding to all of the $n+1$ resolution circles having state $X$ except for two and has the form
\begin{equation} \label{eq:abDelata}
    \Delta_{(0,3-n)} \{a,b\} = (4-\delta_{a,1} -\delta_{b,n+1} - \delta_{b,a+1})\{a,b\} +
\+\{a-1,b\}  +\{a+1,b\} +\{a,b-1\}+\{a,b+1\}
\end{equation}
where we use the convention that $\{a,b\}=0$ if $a<1$ or $b>n+1$ or $a=b$.
\end{proposition}

\begin{proof}
Recall the form of the Khovanov differential from \Cref{subsec:Kh-boundary-intro}.  All differentials leaving the 0th homological degree of $Kh(TU_n)$ are merges, so that the Laplacian $\Delta_{(0,3-n)}$ is determined as a sum of local operators  
\begin{alignat*}{1}
 m_i^{\dagger} m_i \maps V \otimes V&\to V \otimes V      \\  
 \ket{\1\1} &\mapsto \ket{\1\1}    \\
 \ket{\1X} &\mapsto \ket{\1X} + \ket{X\1} \\   
\ket{X\1}  &\mapsto \ket{\1X} + \ket{X\1}  \\
\ket{XX}  &\mapsto 0  
\end{alignat*}
acting on the $i$th and $i+1$st tensor factors of $V^{\otimes (n+1)}$.  That is, 
\[
\Delta_{(0,3-n)} = \sum_{i=1}^n m_i^{\dagger} m_i .
\]
   Note that  
\[
    m_i^{\dagger} m_i \{a,b\} =
 \{a,b\} - \delta_{a,i}\delta_{b,i+1} \{a,b\} 
    + \delta_{i,a-1}\{a-1,b\}
    + \delta_{i,a}\{a+1,b\}
    +\delta_{i,b-1}\{a,b-1\}
    +\delta_{i,b}\{a,b+1\}
\]
so summing over $i$, the result follows. 
\end{proof}

\begin{conjecture} \label{conj-twist}
The smallest gap for $TU_n$ occurs in homological degree $0$ and $q$-degree $3-n$. 
\end{conjecture}

\Cref{conj-twist} has been numerically verified (cf. \Cref{sec:numerics_gap}) for $n$ up to 10.  For larger $n$, we can't prove that the smallest gap always occurs in $q$-degree $3-n$.  However, using the explicit form of the Laplacian $\Delta_{0,3-n}(TU_n)$ from Proposition~\ref{prop:Delta03n}, we can directly compute its spectral gap for larger values of $n$.  Figure~\ref{fig:unknot-plot} shows that while the minimal gap for  $\Delta_{0,3-n}(TU_n)$ gets small, the rate of decrease is Lorentzian and not exponentially decreasing.   We give a proof of this fact in Corollary~\ref{cor:bound}.
 \begin{figure}
  \centering
  \includegraphics[width=5in]{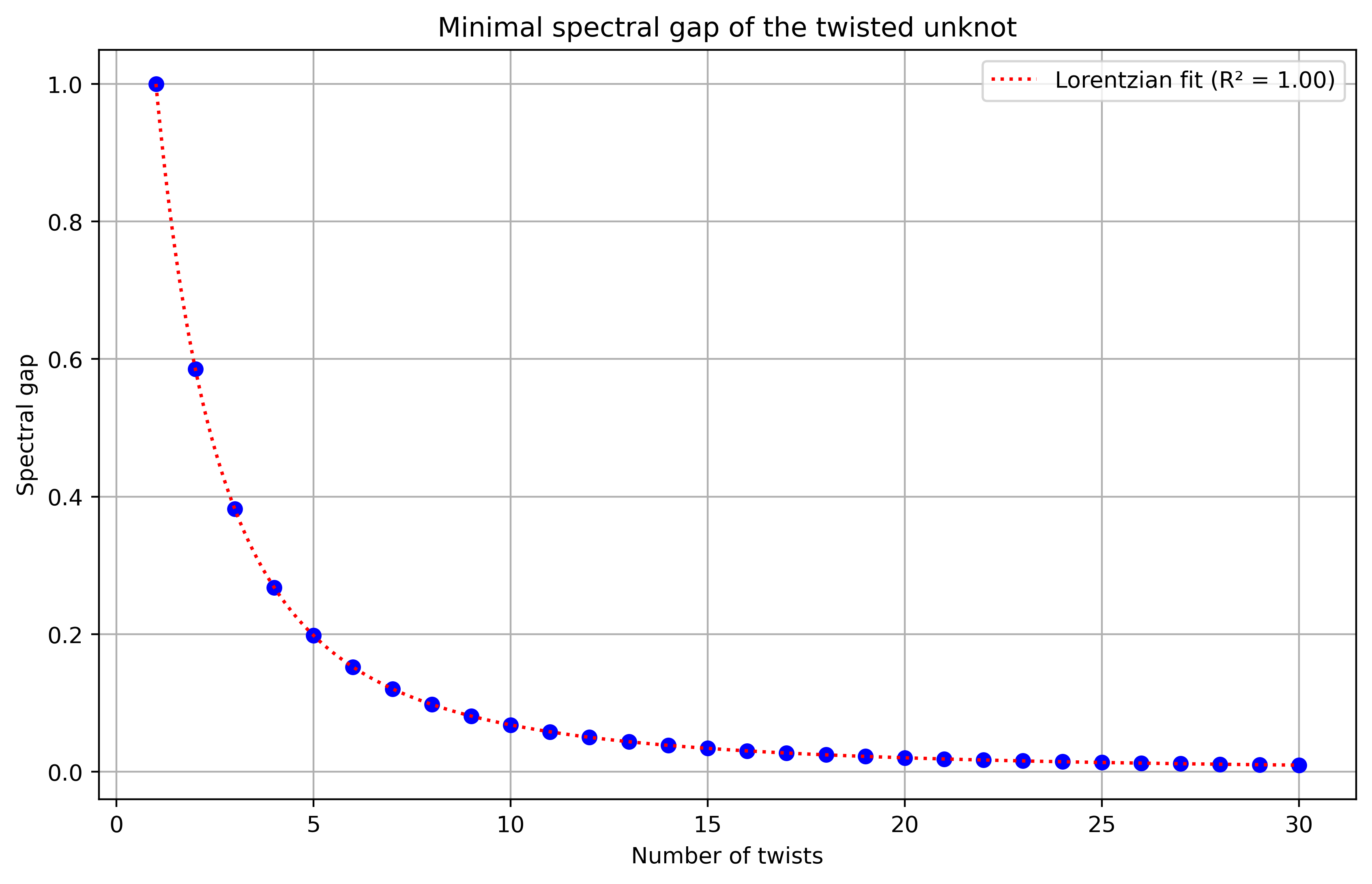}
  \caption{A graph of the spectral gap of the twisted unknot $TU_n$ in bidegree $(0,3-n)$ as a function of $n$.  The data perfectly follows a Lorentzian fit $a/\left( 1+\frac{(x-b)^2}{c^2}\right)$ with $a\to 10.3732,b\to -1.98306,c\to 0.974362$.}\label{fig:unknot-plot}
\end{figure}

In section~\ref{sec:graph}, we give upper and lower bounds for the spectral gap of $TU_n$ in homological degree zero and various $q$-degrees.  We show that the spectral gap of $\Delta_{(0,3-n)}(TU_n)$ can never be exponentially small.  Hence, though Proposition~\ref{prop:gap-inequality} indicates that adding topologically trivial twists can decrease the spectral gap, for $TU_n$ these twists will never create an exponentially small gap.

\subsection{Increasing the spectral gap} \label{subsec:increase}
Analysis of the twisted unknot suggests that powers of a single crossing $\sigma^n$ as in \cref{eq:twists} can be a significant driver for low spectral gaps.   However, the complex associated with $\sigma^n$ is known to have a dramatic simplification.    The naive approach for producing the Khovanov complex of $\sigma^n$ gives a complex of length $n+1$ containing a direct sum of $2^n$ terms.  Specifically, the chain space in homological degree $k$ will be a direct sum of the spaces $V_{\ket{r}}$ of hamming weight $k$.   

However, the powers of a single braid generator are known to greatly simplify to a complex of length $n+1$ with just a single summand in each homological degree.  This observation goes back to the work of Cooper and Krushkal~\cite{MR2901969}.  
To match with our conventions (and ignoring overall grading and homological degree shifts), then $\sigma^n$ is homotopy equivalent to the complex
\begin{equation} \label{eq:reduced}
\left[ \sigma^n \right] =     
    q^{1-n}
   \left[\; 
   \hackcenter{   \begin{tikzpicture}[scale=.6]
      \draw [ very thick]    (1,1) .. controls +(0,-.45) and +(0,-.45) .. (1.8,1);
       \draw [ very thick]    (1,0) .. controls +(0,.45) and +(0,.45) .. (1.8,0);
    \end{tikzpicture}  } \; \right] 
        \stackrel{d_1}{\longrightarrow}
          q^{3-n}
      \left[\; 
   \hackcenter{   \begin{tikzpicture}[scale=.6]
      \draw [ very thick]    (1,1) .. controls +(0,-.45) and +(0,-.45) .. (1.8,1);
       \draw [ very thick]    (1,0) .. controls +(0,.45) and +(0,.45) .. (1.8,0);
    \end{tikzpicture}  } \; \right]   
    \rightarrow\dots    \rightarrow  
    q^{n-3}
          \left[\; 
   \hackcenter{   \begin{tikzpicture}[scale=.6]
      \draw [ very thick]    (1,1) .. controls +(0,-.45) and +(0,-.45) .. (1.8,1);
       \draw [ very thick]    (1,0) .. controls +(0,.45) and +(0,.45) .. (1.8,0);
    \end{tikzpicture}  } \; \right] 
        \stackrel{d_{n-1}}{\longrightarrow}
              q^{n-1}\left[\; 
   \hackcenter{   \begin{tikzpicture}[scale=.6]
      \draw [ very thick]    (1,1) .. controls +(0,-.45) and +(0,-.45) .. (1.8,1);
       \draw [ very thick]    (1,0) .. controls +(0,.45) and +(0,.45) .. (1.8,0);
    \end{tikzpicture}  } \; \right]   
      \stackrel{d_n}{\longrightarrow}
      q^{n}
     \left[\; \hackcenter{   \begin{tikzpicture}[scale=.6]
      \draw [ very thick]    (1,1) to  (1,0);
       \draw [ very thick]    (1.8,1) to  (1.8,0);
    \end{tikzpicture}  }  \;\right] 
\end{equation}
where $d_n$ is the saddle map and
\[
d_{n-i} = 
\left\{
  \begin{array}{ll}
\hackcenter{   \begin{tikzpicture}[scale=.6]
      \draw [ very thick]    (1,1) .. controls +(0,-.45) and +(0,-.45) .. (1.8,1);
       \draw [ very thick]    (1,0) .. controls +(0,.45) and +(0,.45) .. (1.8,0);
       \node at (1.4,.35) {$\bullet$};
    \end{tikzpicture}  } -   \hackcenter{  
   \begin{tikzpicture}[scale=.6]
      \draw [ very thick]    (1,1) .. controls +(0,-.45) and +(0,-.45) .. (1.8,1);
       \draw [ very thick]    (1,0) .. controls +(0,.45) and +(0,.45) .. (1.8,0);
       \node at (1.4,.65) {$\bullet$};
    \end{tikzpicture}  }
        \quad  & \hbox{if $i$ is even;} \\ \\ 
      \hackcenter{   \begin{tikzpicture}[scale=.6]
      \draw [ very thick]    (1,1) .. controls +(0,-.45) and +(0,-.45) .. (1.8,1);
       \draw [ very thick]    (1,0) .. controls +(0,.45) and +(0,.45) .. (1.8,0);
       \node at (1.4,.35) {$\bullet$};
    \end{tikzpicture}  }  +
    \hackcenter{ \begin{tikzpicture}[scale=.6]
      \draw [ very thick]    (1,1) .. controls +(0,-.45) and +(0,-.45) .. (1.8,1);
       \draw [ very thick]    (1,0) .. controls +(0,.45) and +(0,.45) .. (1.8,0);
       \node at (1.4,.65) {$\bullet$};
    \end{tikzpicture} }
        & \hbox{if $i$ is odd.}
  \end{array}
\right.
\]
 These pictures are best understood in Bar-Natan's picture world \cite{BN2}.  However, one can also interpret this as applying to any knot diagram that has a $\sigma^n$ in a small neighborhood of the disc.  Each of the diagrams depicted above is connected to some larger knot diagram where the diagram is identical outside the small disc outside the disc.  Fortunately, without understanding the picture-world interpretation, this simplified complex can be directly applied to the twisted unknot $TU_n$ to show that the Khovanov complex for $TU_n$ is homotopy equivalent to a much simpler complex.  

\begin{proposition}[Gap Reduction]
The Khovanov chain complex associated to $TU_n$ is homotopy equivalent to the complex
\[
\xy
(-90,0)*+{\scriptstyle V^{\otimes 2} \{1-n\} }="R5";
 (-75,0)*++{\dots }="R4";
 (-40,0)*+{\scriptstyle V^{\otimes 2} \{n-5\} }="R3";
 (-4,0)*+{\scriptstyle V^{\otimes 2} \{n-3\} }="R2";
 (40,0)*+{\scriptstyle V^{\otimes 2} \{n-1\} }="R1";
 (62,0)*+{\scriptstyle V\{n\} }="R";
 {\ar^-{   \left(
  \begin{array}{cc}
        0 & 0 \\
        1 & 0 \\
        1 & 0 \\
         0 & 1 \\
     \end{array}
    \right)^T}   "R1";"R"}; 
 {\ar^-{ \left( 
  \begin{array}{cccc}
    0 & 1 & -1 & 0 \\
    0 & 0 & 0 & -1 \\
    0& 0 & 0 & 1 \\
    0 & 0 & 0 & 0 \\
  \end{array}
\right)}   "R2";"R1"}; 
 {\ar^-{ \left( 
  \begin{array}{cccc}
    0 & 1 & 1 & 0 \\
    0 & 0 & 0 & 1 \\
    0& 0 & 0 & 1 \\
    0 & 0 & 0 & 0 \\
  \end{array}
\right)}   "R3";"R2"}; 
{\ar^-{ \left( 
  \begin{array}{cccc}
    0 & 1 & -1 & 0 \\
    0 & 0 & 0 & -1 \\
    0& 0 & 0 & 1 \\
    0 & 0 & 0 & 0 \\
  \end{array}
\right)} "R4";"R3"};
{\ar "R5";"R4"};
\endxy 
\]
where $V=span(\1,X)$ and we order the basis $V\otimes V = \{ \ket{\1\1}, \ket{\1X}, \ket{X\1} , \ket{XX} \}$ and the differentials alternate from right to left between the two 4x4 matrices.  Furthermore, the Laplacian's in each homological degree take the form
\[
\Delta_0 = {\rm Diag}(2,2,0,0), \quad 
\Delta_{n-j} = {\rm Diag}(2,2,2,2), \quad 
\Delta_{n-1} = {\rm Diag}(2,2,2,1), \quad 
\Delta_n= {\rm Diag}(2,1)
\]
for $1<j<n$.  
\end{proposition}
 
\begin{proof}
The form of the complex is an immediate application of the simplified form of the complex $\sigma^n$.  A direct computation gives the Laplacians.  
\end{proof}

\begin{remark}
Replacing powers of crossings with the reduced complex from \cref{eq:reduced} can be used to decrease the spectral gap of the twisted unknot $TU_n$ from the decreasing sequence in Table~\ref{table:min_gap} to have a spectral gap greater than 1 for all $n$.  It is likely that adding a preprocessing step to identify powers of crossings will greatly reduce the spectral gaps encountered while computing spectral gaps in Khovanov homology. 
\end{remark}

\section{Numerical data on the spectral gap}
\label{sec:numerics_gap}
In this section, we present extensive numerical computations of the spectral gap of the Hodge Laplacian $\Delta_{i,j}$. In \Cref{table:min_gap} and \Cref{fig:unknot-plot} of the previous section, we have already summarized numerical data for the spectral gap of twisted forms of the unknot in specific bidegrees. Here, we systematically compute the spectral gap in all bidegrees for all links with $9$ or less crossings and knots with $11$ or less crossings. A list of such knots and links can be found, for example, at Knot Atlas \cite{Knotatlas}. Since the spectral gap is not a knot invariant, we have to choose a specific representation for each knot. We used the standard representation given in the KnotTheory package~\cite{Knotatlas}. These calculations were done using the Center for Advanced Research Computing (CARC) at the University of Southern California.

Note that the quantity that enters the runtime of our quantum algorithm is the spectral gap divided by the spectral norm of $\Delta$, but since $||\Delta||$ is polynomial in $m$ by Gershgorin's circle theorem (in fact, $||\Delta|| \leq 2m$ for our implementation) this normalization does not affect whether the scaling of the spectral gap is polynomial or superpolynomial. We thus directly compute the spectral gap without normalization.

\begin{figure}
  \includegraphics[width=\linewidth]{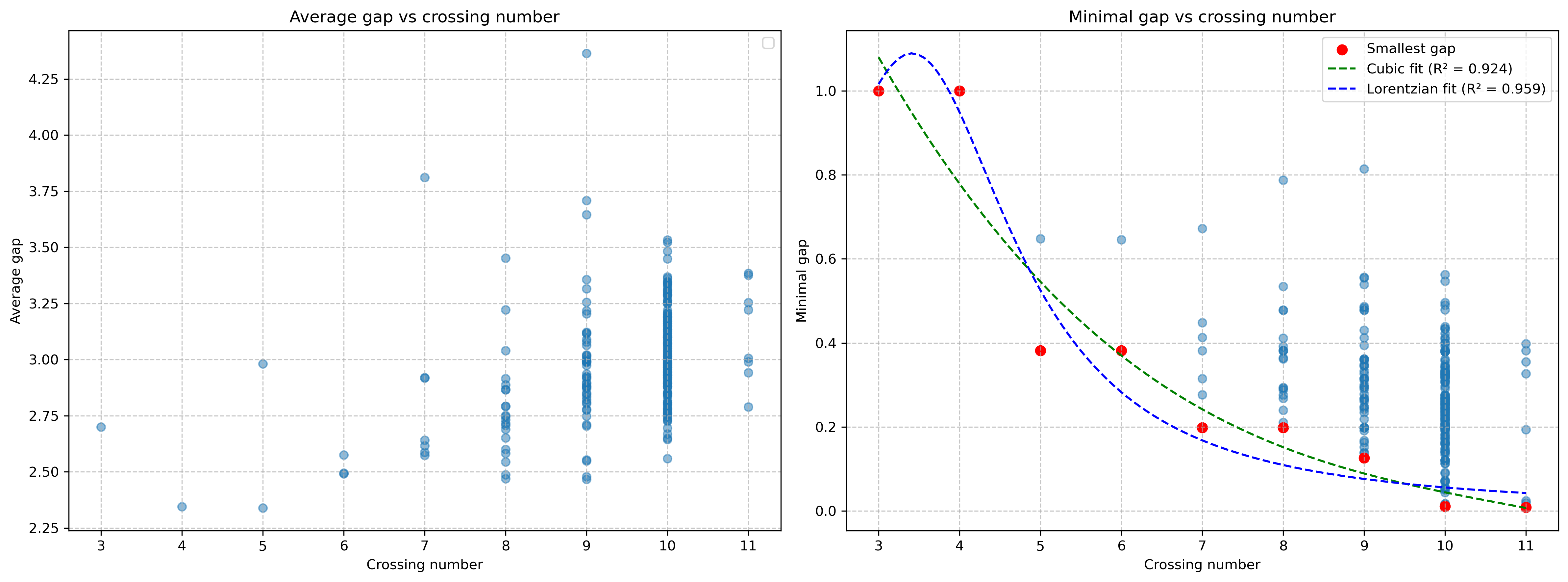}
  \caption{Numerical computations of the spectral gap of the Hodge Laplacian as a function of the crossing number. Each point represents a knot, where the $x$-axis corresponds to the crossing number of a minimal representation of the knot. \textbf{(a)} Average spectral gap for each knot with up to 11 crossings. The gap is averaged over homological and quantum degree $(i,j)$. \textbf{(b)}  Minimal spectral gap for each knot with up to 11 crossings. The gap is minimized over all homological and quantum degrees $(i,j)$. The minimal gap over all knots with a given crossing number is highlighted in red. We fit exponential, cubic, and Lorentzian curves to the minimal spectral gap over all knots with a given crossing number (red points). The cubic and Lorentzian both outperform the exponential fit and are plotted in the graph. The Lorentzian fit achieves the best $R^2$ value. }
  \label{fig:knot_plots}
\end{figure}
\begin{figure}
  \includegraphics[width=\linewidth]{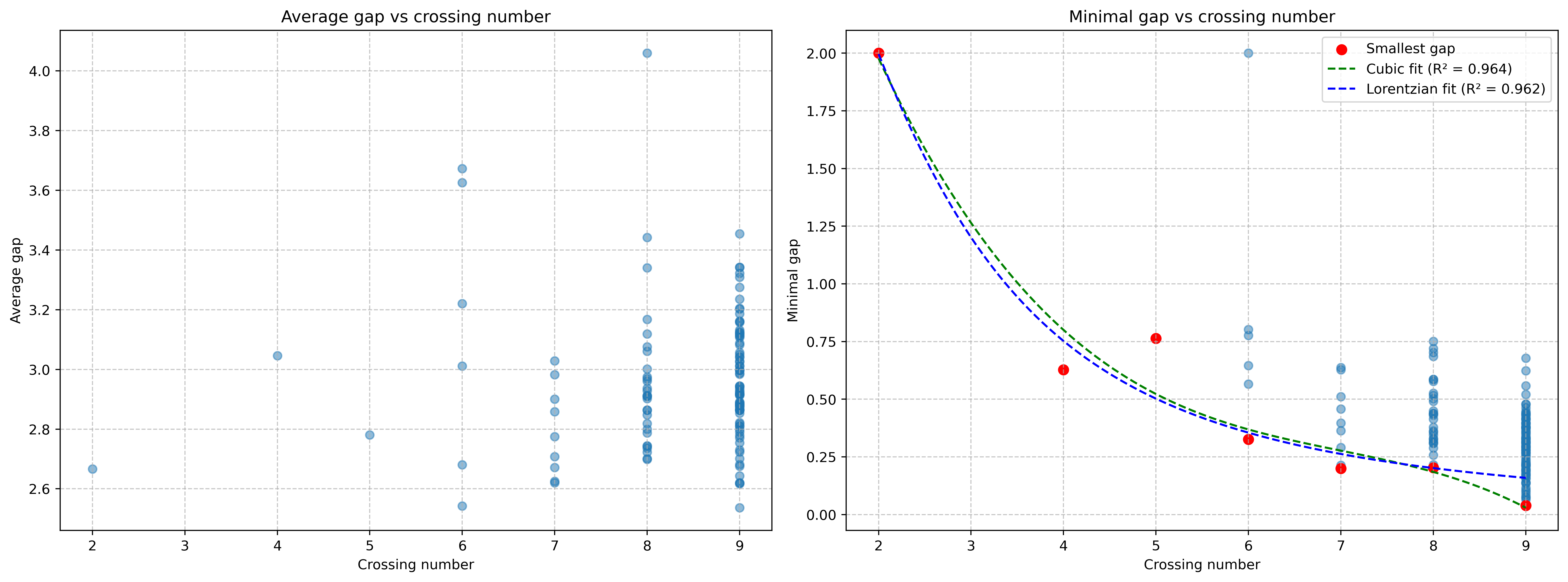}
  \caption{Numerical computations of the spectral gap of the Hodge Laplacian as a function of the crossing number. \textbf{(a)} Average spectral gap for each link with up to 9 crossings. The gap is averaged over homological and quantum degree $(i,j)$. \textbf{(b)}  Minimal spectral gap for each link with up to 9 crossings. The gap is minimized over all homological and quantum degrees $(i,j)$. The red points further denote the minimum over all knots with a given crossing number. The cubic and Lorentzian both outperform the exponential fit and achieve similar $R^2$ value.  }
  \label{fig:link_plots}
\end{figure}
Our data is summarized by the following key points.
\begin{itemize}
    \item For each knot (\Cref{fig:knot_plots}a) and link (\Cref{fig:link_plots}a), we compute the average spectral gap (averaged over quantum and homological degree) of the Hodge Laplacian. We find that this average gap is very large (bigger than one) for all observed knots and seems to increase with the number of crossings. 
    \item For each knot (\Cref{fig:knot_plots}b) and link (\Cref{fig:link_plots}b), we compute the minimal gap (over quantum and homological degree). We find that it decays with the number of crossings, but seems to do so only inverse polynomially.
    \item The minimal spectral gap per crossing number (minimized over quantum and homological degree \textbf{and} all knots with the same number of crossings) is highlighted in red in figures \Cref{fig:knot_plots}b and \Cref{fig:link_plots}b. It decays with the number of crossings. We fit exponential, cubic, and Lorentzian curves to this data. The cubic and Lorentzian both outperform the exponential fit, and the best $R^2$ value for knots is achieved by the Lorentzian fit. However, more data points at higher crossing numbers would be needed to confidently reason about the scaling of the spectral gap. 
    \item Our numerical computations for knots in \Cref{fig:knot_plots} include all knots with 10 or fewer crossings and some with 11 crossings. For 11 crossings, we uniformly at random sample 8 out of the 552 existing knots, because computing the spectral gaps for each one takes over a week using our code and hardware.  Our numerical computations for links in \Cref{fig:link_plots} include all links with 9 or fewer crossings. 
    
    \item It is important to note that all of these computations are for a specific knot diagram of each knot.  As explained in Section~\ref{sec:change-diagram}, changing a knot diagram by a Reidemeister move can change the spectral gap.  We used the representation for knots given in the KnotTheory package~\cite{Knotatlas} for Mathematica. The   `knot PD' presentations for links were obtained from the Knotinfo~\cite{linkinfo}. 
    \item Our numerical calculations moreover reveal that the Betti numbers of Khovanov homology do \emph{not} scale exponentially in the number of crossings. This is also evident from existing databases on Khovanov homology, such as \cite{knotinfo2025}, giving the Betti numbers at each bidegree for all knots up to 13 crossings. 
\end{itemize}

The spectral gap is not the only quantity of interest to us. To efficiently estimate the Betti numbers of Khovanov homology, our algorithm needs to prepare a low-temperature Gibbs state with inverse-polynomial overlap with the ground state. This motivates the study of the number $N(dE)$ of eigenstates of the Hodge Laplacian of Khovanov homology with energy (eigenvalue) in the interval $dE$. Up to normalization, this number is also sometimes called the density of states. In order for a Gibbs state at temperature $\beta$ to achieve inverse-polynomial overlap with the ground state, the count $N(dE)$ should grow less than $e^{\beta E}$ with the energy $E$. \Cref{fig:DoS} shows an example of the full spectrum of the Khovanov Hodge Laplacian. The graph displays the number of eigenstates as a function of their eigenvalue, summed over all quantum and homological degrees and all knots with 10 crossings. Through this plot, we find empirically that $N(dE)$ grows asymptotically as at most $\sim e^{0.36 E}$ in the case of 10 crossings. We also find that $N(dE)$ grows as $\sim e^{0.27 E}$ in the case of 9 crossings, as $\sim e^{0.24 E}$ in the case of 8 crossings, and as $\sim e^{0.19 E}$ in the case of 7 crossings. This suggests that cooling down the Hodge Laplacian to inverse temperature $\beta = 0.19,0.24,0.27, 0.36$ is enough to find the generators of Khovanov homology at the respective crossing number. Determining the scaling of the required temperature for general knots as a function of their crossing number is an interesting direction for further research.
 \begin{figure}
  \centering
  \includegraphics[width=5in]{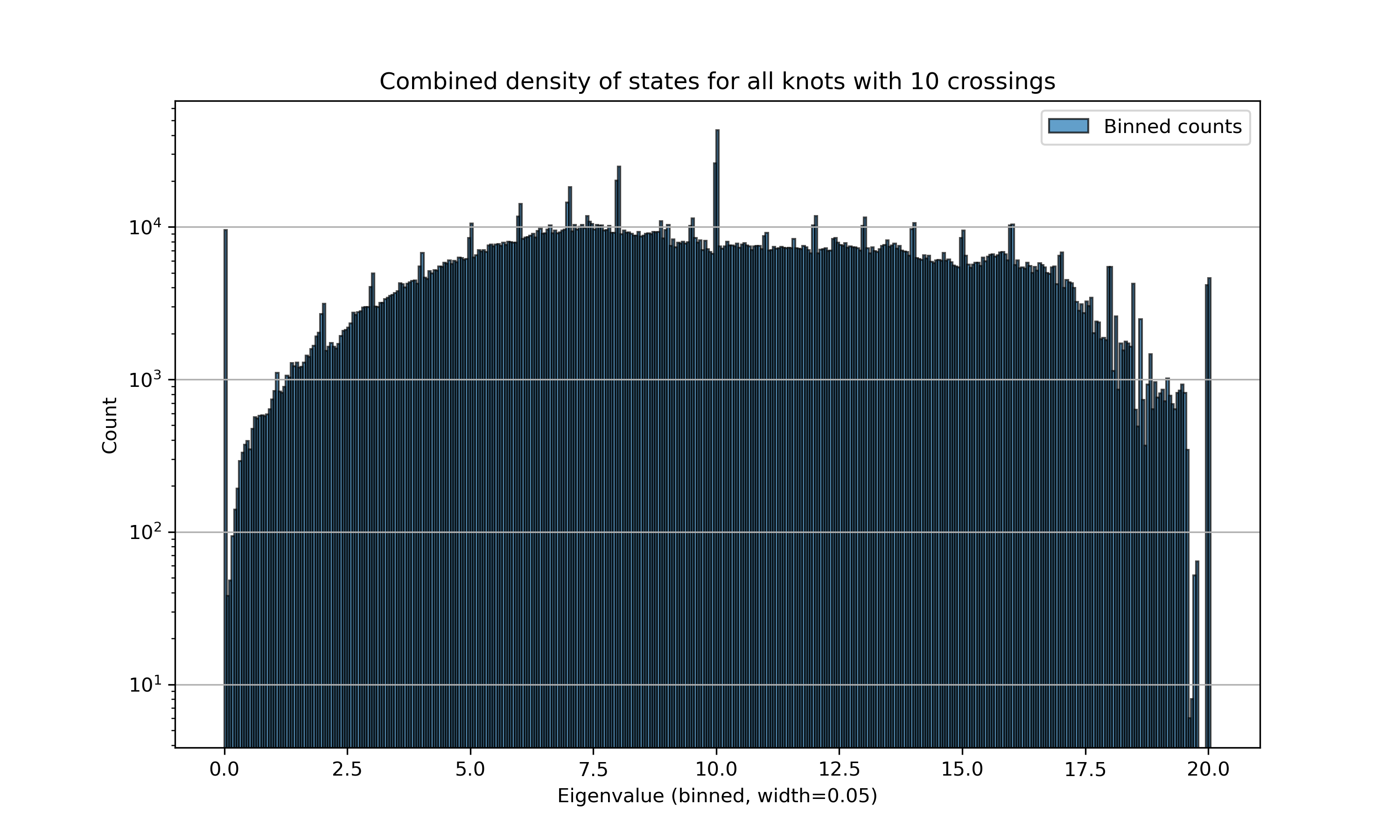}
  \caption{Example of the density of states for Khovanov homology. The figure displays the number of eigenstates (in log-scale) as a function of the eigenvalue. The x-axis is binned into intervals of width $dE = 0.05$. The y-axis (count) is summed over all homological and quantum degrees and all knots with 10 crossings. The first peak at eigenvalue 0 represents the kernel of the Hodge Laplacian, which corresponds to the generators of Khovanov homology. }\label{fig:DoS}
\end{figure}

\subsection{Some phenomenological observations} 
There appears to be a correlation between sequences of twists on the same two strands as in $\sigma^n$ in \cref{eq:twists} and lower spectral gaps.  This phenomenon occurred already in our examination of simple twists of unknots in Section~\ref{subsec:twisted} where subsequence twists had the effect of further decreasing the spectral gap.   All knot diagrams in this section have been taken from the Rolfsen Knot Table on Knot Atlas~\cite{KnotAtlasRolf}.

Among the two 5 crossing knots, $5_1$ has the smaller gap and also has a long sequence of 5 consecutive twists.    For six crossing knots, $6_1$ and $6_2$ have smaller gaps than $6_3$, and again these two appear to have more consecutive twists of the same two strands. 
\[
\stackrel{\includegraphics[width=0.7in]{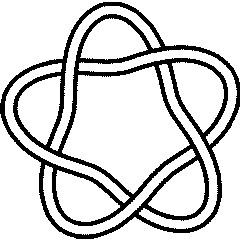}}{5_1}
\qquad 
\stackrel{\includegraphics[width=0.7in]{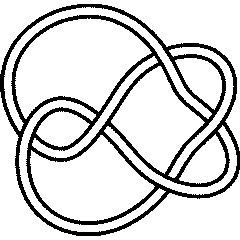}}{5_2}
\qquad 
\stackrel{\includegraphics[width=0.7in]{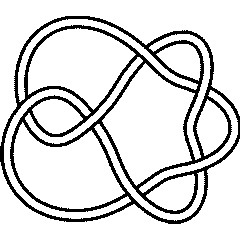}}{6_1}
\qquad 
\stackrel{\includegraphics[width=0.7in]{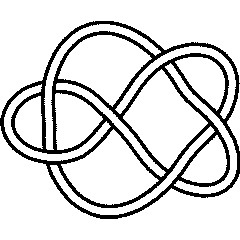}}{6_2}
\qquad 
\stackrel{\includegraphics[width=0.7in]{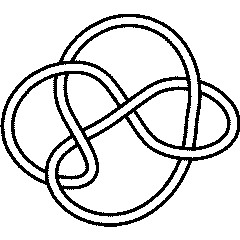}}{6_3}
\]

For 7-crossing knots, $7_1$ has the smallest spectral gap, and again its knot diagram is a sequence of 7 consecutive twists.  
\[
\stackrel{\includegraphics[width=0.7in]{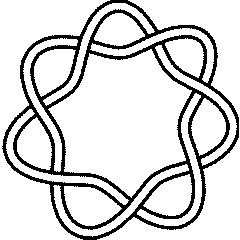}}{7_1}
\quad 
\stackrel{\includegraphics[width=0.7in]{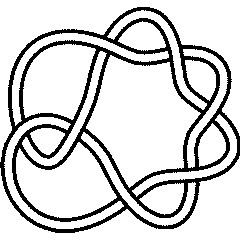}}{7_2}
\quad 
\stackrel{\includegraphics[width=0.7in]{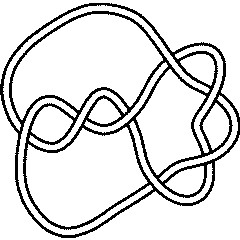}}{7_3}
\quad 
\stackrel{\includegraphics[width=0.7in]{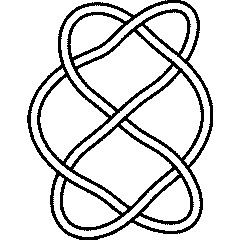}}{7_4}
\quad 
\stackrel{\includegraphics[width=0.7in]{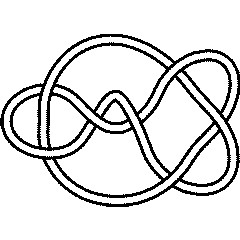}}{7_5}
\quad 
\stackrel{\includegraphics[width=0.7in]{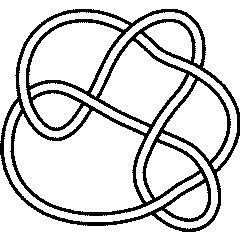}}{7_6}
\quad 
\stackrel{\includegraphics[width=0.7in]{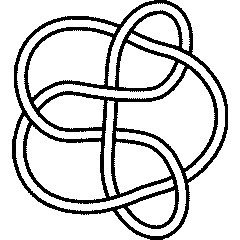}}{7_7}
 \]
Knots $7_2$ and $7_3$ have the next smallest gaps among the 7-crossing knots. 

For the 8-crossing knots, both $8_1$ and $8_2$ have small spectral gaps and large numbers of consecutive crossings, however the smallest gap that is encountered with 8-crossings comes from the knot $8_{20}$.  This is an example of a 3-strand pretzel knot $8_{20} = P(3,-3,2)$ meaning that it is a combination of twists of length 3, inverse twists of length 3, and a twist of length 2.    
\begin{align*}
& 
\stackrel{\includegraphics[width=0.7in]{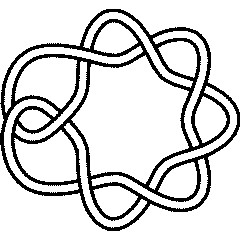}}{8_1}
\quad 
\stackrel{\includegraphics[width=0.7in]{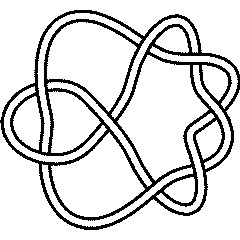}}{8_2}
\quad 
\stackrel{\includegraphics[width=0.7in]{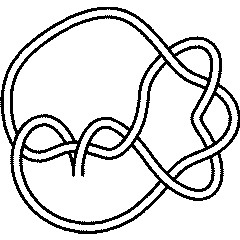}}{8_3}
\quad 
\stackrel{\includegraphics[width=0.7in]{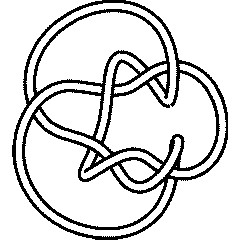}}{8_4}
\quad 
\stackrel{\includegraphics[width=0.7in]{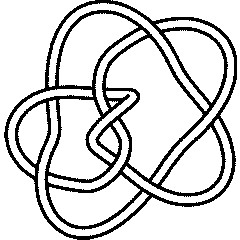}}{8_5}
\quad 
\stackrel{\includegraphics[width=0.7in]{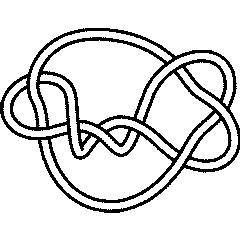}}{8_6}
\quad 
\stackrel{\includegraphics[width=0.7in]{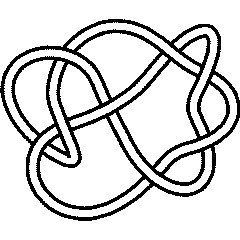}}{8_7}
\\
&
\stackrel{\includegraphics[width=0.7in]{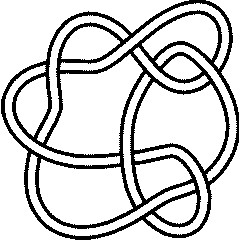}}{8_{8}}
\quad 
\stackrel{\includegraphics[width=0.7in]{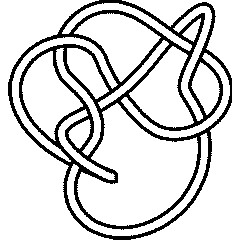}}{8_{9}}
\quad 
\stackrel{\includegraphics[width=0.7in]{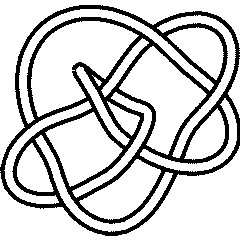}}{8_{10}}
\quad 
\stackrel{\includegraphics[width=0.7in]{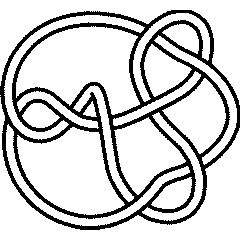}}{8_{11}}
\quad 
\stackrel{\includegraphics[width=0.7in]{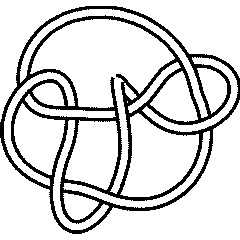}}{8_{12}}
\quad 
\stackrel{\includegraphics[width=0.7in]{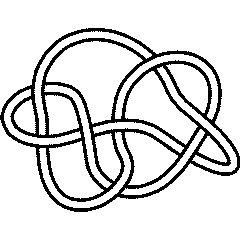}}{8_{13}}
\quad 
\stackrel{\includegraphics[width=0.7in]{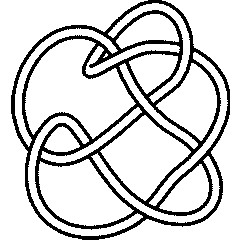}}{8_{14}}
\\ 
&
\stackrel{\includegraphics[width=0.7in]{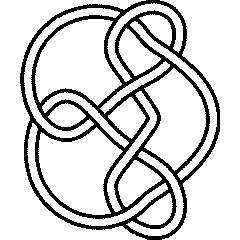}}{8_{15}}
\quad 
\stackrel{\includegraphics[width=0.7in]{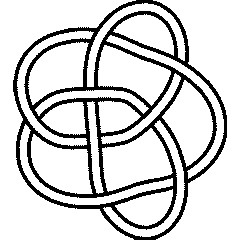}}{8_{16}}
\quad 
\stackrel{\includegraphics[width=0.7in]{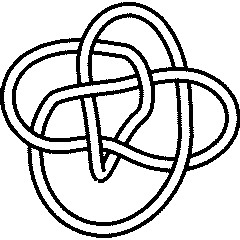}}{8_{17}}
\quad 
\stackrel{\includegraphics[width=0.7in]{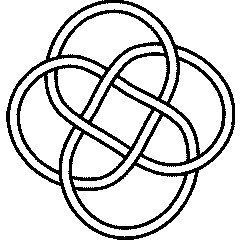}}{8_{18}}
\quad 
\stackrel{\includegraphics[width=0.7in]{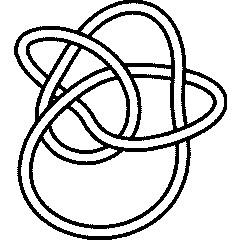}}{8_{19}}
\quad 
\stackrel{\includegraphics[width=0.7in]{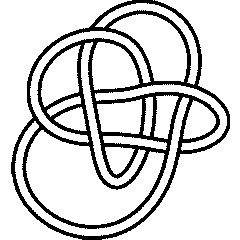}}{8_{20}}
\quad 
\stackrel{\includegraphics[width=0.7in]{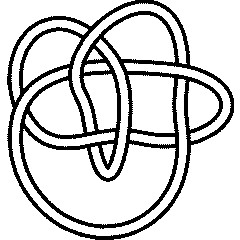}}{8_{21}}
\end{align*}

As expected, the knot $9_1$ with 9 consecutive twists has a much lower minimal spectral gap than the average of the 9-crossing knots, though not as small as the gap that occurs for $9_{46}=P(3,3,-3)$ that is also a 3-stranded pretzel knot.  This suggests that long sequences of consecutive twists are a factor, but not the only mechanism producing the smallest observed spectral gaps. The knots $9_{42}$ and $9_{43}$ both have smaller spectral gaps than $9_{1}$, though neither are pretzel knots.  
\[
\stackrel{\includegraphics[width=0.7in]{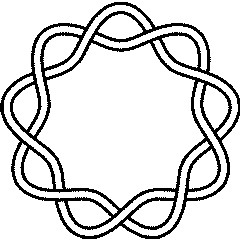}}{9_1}
\qquad 
\stackrel{\includegraphics[width=0.7in]{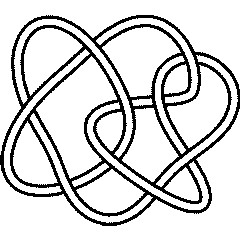}}{9_{42}}
\qquad 
\stackrel{\includegraphics[width=0.7in]{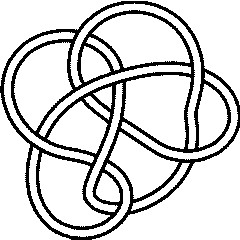}}{9_{43}}
\qquad 
\stackrel{\includegraphics[width=0.7in]{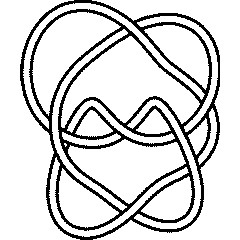}}{9_{46}}
\]

These observations from knots of low crossing number illustrate a correlation between the smallest spectral gaps and consecutive twists of the same two strands.  This is very promising as the methods described in Section~\ref{subsec:increase} provide a simple technique for greatly simplifying the computations of Khovanov homology associated with consecutive twists.  This suggests that incorporating these techniques into quantum algorithms for computing Khovanov homology is likely to produce speedups.  

For 10 crossing knots, the smallest spectral gaps encountered are clustered in the knot range $10_j$ for $128 \leq j \leq 160$ out of the 165 knots with 10 crossings.  The history of knot tabulation is a rich subject~\cite{MR1646740}.  Aside from listing the knots with a given crossing number in order of alternating to non-alternating, we do not know the rationale behind the ordering of the tabulation of knots other than aesthetics.  Nevertheless, this cluster of 10 crossing knots with the smallest spectral gap are all non-alternating knots with $10_{153}$ having the smallest spectral gap.

\section{A study of spectral gaps in homological degree zero} \label{sec:graph}

There are many well-known methods for bounding the spectral gap for combinatorial Laplacians of graphs, while general methods are harder to come by.  The Laplacians encountered in Khovanov homology are, in general, very different than graph Laplacians.  However, we show here that the Laplacians in homological degree 0 do have a graph theoretic interpretation, though not as graph Laplacian.   

Given a graph $G=(V,E)$, possibly with loops and weighted edges, we can consider the \emph{signless Laplace matrix} $Q = D+A$ where $D$ is the diagonal matrix of degrees of the vertices in the graph, and $A$ is the adjacency matrix~\cite{HAEMERS2004199}.  In contrast, the more studied graph Laplacian is $L=D-A$.  .  
The signless Laplacian matrix $Q$ of a graph is closely connected to the bipartiteness of the graph $G$.  In particular, $Q$ is always positive semi-definite, with the smallest eigenvalue equal to 0 if and only if $G$ is bipartite.  For $G$ non-bipartite, the matrix $Q$ will be positive definite.  In \cite{Desai}, the authors study the separation of the smallest eigenvalue from 0 using a certain measure on the non-bipartiteness of $G$ over subsets of vertices.  

In this section, we associate a weighted graph $G=(V,E)$ to the 0th chain group and its differentials in the Khovanov complex $\llangle K\rrangle$ in each $q$-degree.  We show that $\Delta_{(0,q)}(K)$ can be regarded as the matrix $Q=D+A$  of this graph.  The graph $G$ almost always contains loops and weighted edges. 
First, we recall the results from \cite{Desai} in the context of weighted graphs explained at the end of that article.    This interpretation gives rise to upper and lower bounds for the spectral gap of $\Delta_{(0,q)}(K)$.

\subsection{Bounds on minimal eigenvalues of weighted graphs}
A  \emph{weighted graph} $G = (V,E)$ with $V=\{1,2, \dots, n\}$ and $E=\{b_1,\dots, b_m\}$ is a graph with vertex set $V$, edge set $E$, and for each edge $(i,j) \in E$, a positive weight $w_{ij} >0$.      We will sometimes regard an edge with a positive integer weight $w$ as representing $w$ different edges of weight 1.  Our weighted graphs can have loops at vertices.

The \emph{adjacency matrix} $A$ of a weighted graph has entries $w_{ij}$, where $w_{ij}$ is the weight of the edge $(i,j)$, so that $A$ is a symmetric matrix.   The \emph{degree} of a vertex $i$ is defined as 
\[
d(i) = \sum_{(i,j) \in E} w_{ij}
\]
and the \emph{degree matrix} $D$ is the diagonal matrix of degrees.

For any subset $S \subset V$ of vertices of $G$, we can form the subgraph $G_S$ induced by $S$.  Denote by $e_{\rm min}(S)$ the minimum sum of edge weightings over edges that need to be removed from $G_S$ to make it bipartite. This sum includes weights for loops at vertices in $S$.   Denote by $cut(S)$ the set of edges with one edge boundary in $S$ and one in $V-S$.  If an edge carries a weight $w_{ij}$, then we count that edge with multiplicity $w_{ij}$.   
Define
\begin{equation} \label{eq:defPsi}
    \Psi := \min_{S \subset V}\Psi_S, \qquad \Psi_S:= \frac{e_{\rm min}(S)+|cut(S)|}{|S|} .
\end{equation}
Observe that for a bipartite graph $G$, taking $S=V$ we would have $e_{\rm min}(S)=|cut(S)|=0$. We would also get zero if the graph $G$ has a bipartite component.    

The following two theorems are stated in \cite{Desai} and proven in the context of non-weighted graphs. 

\begin{theorem}[Proposition 2.1 \cite{Desai}]
    The matrix $Q=D+A$ associated with a weighted graph $G$ is positive semidefinite.  It is singular if and only if $\Psi=0$ if and only if the graph contains a bipartite component. 
\end{theorem}

\begin{theorem}[Theorem 3.1 \& 3.2 \cite{Desai}] \label{thm:Desai}
    Let $G$ be a weighted graph and $\Psi$ defined in \cref{eq:defPsi}, then the smallest eigenvalue $\lambda_{\rm min}(Q)$ of $Q=D+A$ satisfies the inequality
\[
  \frac{\Psi^2}{4d^{\ast}}  \leq \lambda_{\rm min}(Q) \leq 4 \Psi
 \]
 where $d^{\ast}$ is the largest degree of a vertex in the graph $G$.
\end{theorem}

\subsection{Graphs from Khovanov Laplacians in homological degree zero}

\subsubsection{The structure of homological degree zero Khovanov homology} \label{subsubsec:homzero-structure}
Recall from Section~\ref{subsec:Kh-boundary-intro} that the Khovanov chain complex $\llangle K\rrangle$ associates the vector space  $V = {\rm span} ( \1,X )$ to a single circle with $\deg(\1)=1$ and $\deg(X)=-1$. 
Let $\ket{r} = \ket{0\dots 0}$ denote the all 0-resolution of $K$, and let $\ell = \ell(r)$ denote the number of circles in $\ket{r}$.

Recall from Section~\ref{subsec:Kh-encoding} our convention for indexing the copies of $V$ arising in a resolution of a knot $K$. If we ignore the over and under-crossing information of a knot and regard each crossing as a four-valent vertex, each knot diagram contains $2N$ edges between the $N$ vertices.  We label each edge by choosing a starting point, then each segment is labeled in order as we traverse the knot getting back to the original starting point.   When we resolve all of the crossings, each circle in the resulting resolutions will contain two or more labeled edges.  We label that copy of $V$ by the minimal edge it contains. 
  We write $V^{\otimes \ell} = V_{i_1} \otimes V_{i_2} \otimes \dots \otimes V_{i_\ell}$, where $1\leq i_1<i_2<\dots i_{\ell} \leq 2N$ are the smallest label appearing in the resolution.

\begin{example} \label{example:trefoil-label}
Consider the knot $K=6_3$ with a labeling $\und{1}, \dots, \und{6}$ of its crossings and a labeling from $1$ to $12$ of its edges.  Below we illustrate the process of lablling all edge components and how this labelleing gives rise to an indexing of the tensor powers appearing in the all 0-resolution of this knot.
\[
\xy
(0,0)*+{\includegraphics[width=1.8in]{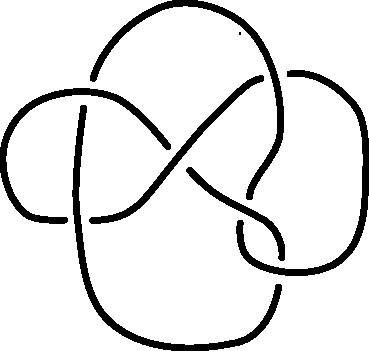}}:
(0,0)*+{
    \begin{tikzpicture}[scale=1]
    %
        \node[black] at (1.05, 2.9) {$\scriptstyle 1$};
        \node[black] at (.1, 1.16) {$\scriptstyle 2$};
        \node[black] at (-1.9, 2) {$\scriptstyle 3$};
        \node[black] at (.1, 2.9) {$\scriptstyle 4$};
        \node[black] at (.8, 1.35) {$\scriptstyle 5$};
         \node[black] at (2, 1.02) {$\scriptstyle 6$};
        \node[black] at (.63, -0.7) {$\scriptstyle 7$};
        \node[black] at (-.9, 2) {$\scriptstyle 8$};
        \node[black] at (.63, 4.2) {$\scriptstyle 9$};
        \node[black] at (2.1, 2.2) {$\scriptstyle 10$};
         \node[black] at (1.1, .8) {$\scriptstyle 11$};
        \node[black] at (3.1, 2) {$\scriptstyle 12$};
      %
        \node[red] at (.61, 2.45) {\underline{1}};
        \node[red] at (-.7, 3.2) {\underline{2}};
        \node[red] at (2.1, 3.3) {\underline{3}};
        \node[red] at (1.3, 1.85) {\underline{4}};
        \node[red] at (2.15, .15) {\underline{5}}; 
        \node[red] at (-.9, 0.7) {\underline{6}};
    \end{tikzpicture}
    };
\endxy  
\qquad 
\quad 
    \xy
    (0,0)*+{\includegraphics[width=1.8in]{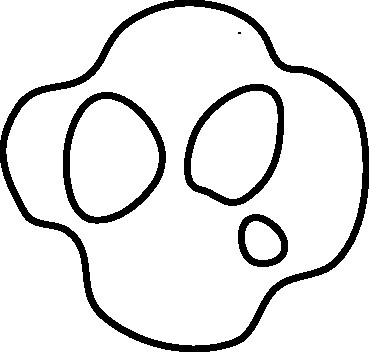}}:
    (0,0)*+{
        \begin{tikzpicture}[scale=1]
        %
            \node[black] at (1.05, 2.9) {$\scriptstyle 1$};
            \node[black] at (.1, 1.16) {$\scriptstyle 2$};
            \node[black] at (-1.9, 2) {$\scriptstyle 3$};
            \node[black] at (.1, 2.9) {$\scriptstyle 4$};
            \node[black] at (.8, 1.35) {$\scriptstyle 5$};
             \node[black] at (2, 1.02) {$\scriptstyle 6$};
            \node[black] at (.63, -0.6) {$\scriptstyle 7$};
            \node[black] at (-.9, 2) {$\scriptstyle 8$};
            \node[black] at (.63, 4.1) {$\scriptstyle 9$};
            \node[black] at (2.1, 2.2) {$\scriptstyle 10$};
             \node[black] at (1.1, .8) {$\scriptstyle 11$};
            \node[black] at (3.1, 2) {$\scriptstyle 12$};
        \end{tikzpicture}
        };
    \endxy  
\]
On the right, each loop is labeled by multiple edges from the original knot.  The tensor factors carry the minimal label appearing on each loop. In this case, $V_1 \otimes V_2 \otimes V_3 \otimes V_6$.
\end{example}

The graded vector space $V^{\otimes \ell}$ is a direct sum of its homogeneous components $V^{\otimes \ell}_q$ consisting of the span of the homogeneous basis given by $v=v_{i_1} \otimes \dots v_{i_\ell}$ with $v_{i} \in \{ \1, X\}$,  $\deg(v) = q=k-n$ (up to an overall shift), where $k$ is the number of $v_i=\1$ and $n$ is is the number of $v_i=X$.

If $K$ is an $N$ crossing knot, the differential $d_0$ is a direct sum of $N$ maps $d_{rr'}$ mapping $\ket{r}$ to resolutions $\ket{r'}$ obtained by replacing a 0 by a single $1$ in one of the $N$ slots of $\ket{r}$, see \Cref{subsec:Kh-boundary-intro}.   
Each of the maps $d_{rr'}$ is obtained by merging and splitting circles, corresponding to the linear maps called multiplication or comultiplication.  These maps act on two, respectively one, of the tensor factors in $V^{\otimes \ell}$ with $m_{ij} \maps V^{\otimes \ell} \to V^{\otimes (\ell -1)}$ and $\delta_i \maps V^{\otimes \ell} \to V^{\otimes (\ell+1)}$.    These maps act as the identity on all tensor factors except for
\begin{alignat*}{3}
 m_{ij} \maps V_i\otimes V_j&\to V_{\min(i,j)}   \qquad \qquad    && \delta_{i,j} \maps V_{\min(i,j)} \to V_i \otimes V_j\\
 \ket{\1\1} &\mapsto \ket{\1} \qquad \qquad    && \quad\; \ket{\1} \mapsto \ket{\1X} + \ket{X\1} \\
 \ket{\1X} &\mapsto \ket{X} \qquad \qquad    && \quad\; \ket{X} \mapsto \ket{XX}\\
\ket{X\1}  &\mapsto \ket{X} \\
\ket{XX}  &\mapsto 0 
\end{alignat*}
where $i$ and $j$ indices of two tensor factors in $V^{\otimes \ell}$ and $\min(i,j)$ is a label of a single tensor factor in $V^{\otimes \ell}$.  In $\delta_{i,j}$, the newly created circle is labeled by $\max(i,j)$. 

The Laplacian $\Delta_{0}(K)\maps  V^{\otimes \ell} \to V^{\otimes \ell}$ in homological degree zero is then a direct sum of $N$ maps of the form 
 \begin{alignat}{3} \label{eq:Delta0-bar}
 \mathfrak{m}_{i,j:= }m^{\dagger}_{ij} m_{ij} \maps V_i \otimes V_j&\to V_i \otimes V_j   \qquad \qquad    
        && \mathfrak{d}_{i} := \delta^{\dagger}_{is} \delta_{is} \maps V_i \to V_i  \nonumber \\
 \ket{\1\1} &\mapsto \ket{\1\1} \qquad \qquad    
    && \quad\; \ket{\1} \mapsto 2\ket{\1 }   \nonumber\\
 \ket{\1X} &\mapsto \ket{\1X} + \ket{X\1} \qquad \qquad    && \quad\; \ket{X} \mapsto \ket{X}\\
\ket{X\1}  &\mapsto \ket{\1X} + \ket{X\1}   \nonumber\\
\ket{XX}  &\mapsto 0  \nonumber
\end{alignat}
 that act by the identity on the non-indexed tensor factors.  In particular,
\begin{equation} \label{eq:0Lsplit}
  \Delta_0(K) = \sum_{a=1}^{p} \mathfrak{m}_{i_a, j_a} + \sum_{b=1}^t \mathfrak{d}_{i_b } 
\end{equation}
for some   $1 \leq i_a < j_a \leq 2N$, $1 \leq i_b\leq 2N$, where  $N=p+t$, and $p$ is the number of multiplications in $\Delta_0$ and $t$ is the number of comultiplications in $\Delta_0$.

\begin{example}\label{example:trefoil-zero}
For the knot $6_3$, the all-zero resolution is a tensor product of $\ell=4$ copies of $V$ carrying labels $V_1 \otimes V_2 \otimes V_3 \otimes V_6$.  The Laplacian $\Delta_0(K)$ takes the form
\[ 
\Delta_0(6_3) = \mathfrak{m}_{12}+2\mathfrak{m}_{23} + \mathfrak{m}_{13} + \mathfrak{m}_{16} + \mathfrak{m}_{36}
\]
Figure~\ref{fig:6-3-Hom=-zero} illustrates each term of the differential and the corresponding resolutions. 
\end{example}

\begin{figure}
    \centering
$
\xy
(-30,0)*+{
    \xy
    (0,0)*+{\includegraphics[width=1.35in]{figures/6_3-000000.jpg}}:
    (0,0)*+{
        \begin{tikzpicture}[scale=.8]
        %
            \node[black] at (1.05, 2.9) {$\scriptstyle 1$};
            \node[black] at (.1, 1.16) {$\scriptstyle 2$};
            \node[black] at (-1.9, 2) {$\scriptstyle 3$};
            \node[black] at (.1, 2.9) {$\scriptstyle 4$};
            \node[black] at (.8, 1.35) {$\scriptstyle 5$};
             \node[black] at (2, 1.02) {$\scriptstyle 6$};
            \node[black] at (.63, -0.6) {$\scriptstyle 7$};
            \node[black] at (-.9, 2) {$\scriptstyle 8$};
            \node[black] at (.63, 4.1) {$\scriptstyle 9$};
            \node[black] at (2.1, 2.2) {$\scriptstyle 10$};
             \node[black] at (1.1, .8) {$\scriptstyle 11$};
            \node[black] at (3.1, 2) {$\scriptstyle 12$};
        \end{tikzpicture}
        };
    \endxy  
    }="0";
%
(18,65)*{
    \xy
    (0,0)*+{\includegraphics[width=1in]{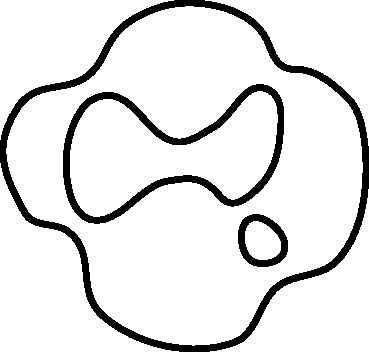}}:
    (0,0)*+{
    \begin{tikzpicture}[scale=.6]
    %
        \node[black] at (1.05, 2.9) {$\scriptscriptstyle 1$};
        \node[black] at (.1, 1.16) {$\scriptscriptstyle 2$};
        \node[black] at (-1.9, 2) {$\scriptscriptstyle 3$};
        \node[black] at (.1, 2.9) {$\scriptscriptstyle 4$};
        \node[black] at (.8, 1.35) {$\scriptscriptstyle 5$};
         \node[black] at (2, 1.02) {$\scriptscriptstyle 6$};
        \node[black] at (.63, -0.5) {$\scriptscriptstyle 7$};
        \node[black] at (-.9, 2) {$\scriptscriptstyle 8$};
        \node[black] at (.63, 4.0) {$\scriptscriptstyle 9$};
        \node[black] at (2.1, 2.2) {$\scriptscriptstyle 10$};
         \node[black] at (1.1, .8) {$\scriptscriptstyle 11$};
        \node[black] at (3.1, 2) {$\scriptscriptstyle 12$};
    \end{tikzpicture}
         };
    \endxy  
    }="t1";
%
(32,40)*+{
    \xy
    (0,0)*+{\includegraphics[width=1in]{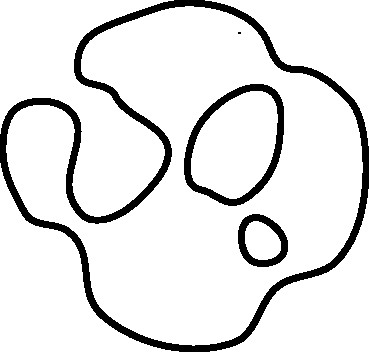}}:
    (0,0)*+{
    \begin{tikzpicture}[scale=.6]
    %
        \node[black] at (1.05, 2.9) {$\scriptscriptstyle 1$};
        \node[black] at (.1, 1.16) {$\scriptscriptstyle 2$};
        \node[black] at (-1.9, 2) {$\scriptscriptstyle 3$};
        \node[black] at (.1, 2.9) {$\scriptscriptstyle 4$};
        \node[black] at (.8, 1.35) {$\scriptscriptstyle 5$};
         \node[black] at (2, 1.02) {$\scriptscriptstyle 6$};
        \node[black] at (.63, -0.5) {$\scriptscriptstyle 7$};
        \node[black] at (-.9, 2) {$\scriptscriptstyle 8$};
        \node[black] at (.63, 4.0) {$\scriptscriptstyle 9$};
        \node[black] at (2.1, 2.2) {$\scriptscriptstyle 10$};
         \node[black] at (1.1, .8) {$\scriptscriptstyle 11$};
        \node[black] at (3.1, 2) {$\scriptscriptstyle 12$};
    \end{tikzpicture}
         };
    \endxy  
    }="t2";
%
(45,15)*+{
    \xy
    (0,0)*+{\includegraphics[width=1in]{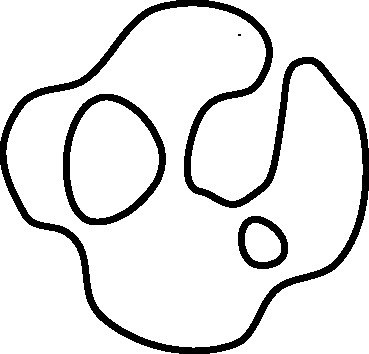}}:
    (0,0)*+{
    \begin{tikzpicture}[scale=.6]
    %
        \node[black] at (1.05, 2.9) {$\scriptscriptstyle 1$};
        \node[black] at (.1, 1.16) {$\scriptscriptstyle 2$};
        \node[black] at (-1.9, 2) {$\scriptscriptstyle 3$};
        \node[black] at (.1, 2.9) {$\scriptscriptstyle 4$};
        \node[black] at (.8, 1.35) {$\scriptscriptstyle 5$};
         \node[black] at (2, 1.02) {$\scriptscriptstyle 6$};
        \node[black] at (.63, -0.5) {$\scriptscriptstyle 7$};
        \node[black] at (-.9, 2) {$\scriptscriptstyle 8$};
        \node[black] at (.63, 4.0) {$\scriptscriptstyle 9$};
        \node[black] at (2.1, 2.2) {$\scriptscriptstyle 10$};
         \node[black] at (1.1, .8) {$\scriptscriptstyle 11$};
        \node[black] at (3.1, 2) {$\scriptscriptstyle 12$};
    \end{tikzpicture}
         };
    \endxy  
    }="t3";
%
(45,-15)*+{
    \xy
    (0,0)*+{\includegraphics[width=1in]{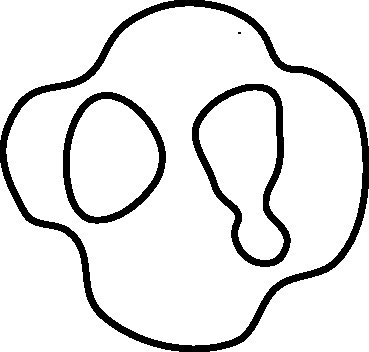}}:
    (0,0)*+{
    \begin{tikzpicture}[scale=.6]
    %
        \node[black] at (1.05, 2.9) {$\scriptscriptstyle 1$};
        \node[black] at (.1, 1.16) {$\scriptscriptstyle 2$};
        \node[black] at (-1.9, 2) {$\scriptscriptstyle 3$};
        \node[black] at (.1, 2.9) {$\scriptscriptstyle 4$};
        \node[black] at (.8, 1.35) {$\scriptscriptstyle 5$};
         \node[black] at (2, 1.02) {$\scriptscriptstyle 6$};
        \node[black] at (.63, -0.5) {$\scriptscriptstyle 7$};
        \node[black] at (-.9, 2) {$\scriptscriptstyle 8$};
        \node[black] at (.63, 4.0) {$\scriptscriptstyle 9$};
        \node[black] at (2.1, 2.2) {$\scriptscriptstyle 10$};
         \node[black] at (1.1, .8) {$\scriptscriptstyle 11$};
        \node[black] at (3.1, 2) {$\scriptscriptstyle 12$};
    \end{tikzpicture}
         };
    \endxy  
    }="t4";
%
(32,-40)*+{
    \xy
    (0,0)*+{\includegraphics[width=1in]{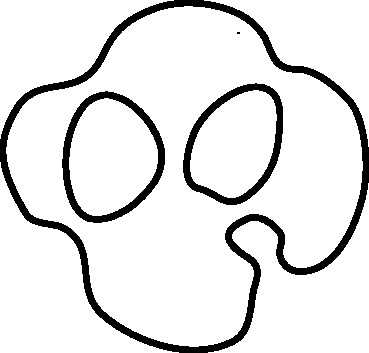}}:
    (0,0)*+{
    \begin{tikzpicture}[scale=.6]
    %
        \node[black] at (1.05, 2.9) {$\scriptscriptstyle 1$};
        \node[black] at (.1, 1.16) {$\scriptscriptstyle 2$};
        \node[black] at (-1.9, 2) {$\scriptscriptstyle 3$};
        \node[black] at (.1, 2.9) {$\scriptscriptstyle 4$};
        \node[black] at (.8, 1.35) {$\scriptscriptstyle 5$};
         \node[black] at (2, 1.02) {$\scriptscriptstyle 6$};
        \node[black] at (.63, -0.5) {$\scriptscriptstyle 7$};
        \node[black] at (-.9, 2) {$\scriptscriptstyle 8$};
        \node[black] at (.63, 4.0) {$\scriptscriptstyle 9$};
        \node[black] at (2.1, 2.2) {$\scriptscriptstyle 10$};
         \node[black] at (1.1, .8) {$\scriptscriptstyle 11$};
        \node[black] at (3.1, 2) {$\scriptscriptstyle 12$};
    \end{tikzpicture}
         };
    \endxy  
    }="t5";
%
(18,-65)*+{
    \xy
    (0,0)*+{\includegraphics[width=1in]{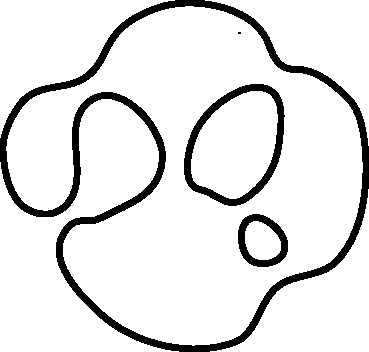}}:
    (0,0)*+{
    \begin{tikzpicture}[scale=.6]
    %
        \node[black] at (1.05, 2.9) {$\scriptscriptstyle 1$};
        \node[black] at (.1, 1.16) {$\scriptscriptstyle 2$};
        \node[black] at (-1.9, 2) {$\scriptscriptstyle 3$};
        \node[black] at (.1, 2.9) {$\scriptscriptstyle 4$};
        \node[black] at (.8, 1.35) {$\scriptscriptstyle 5$};
         \node[black] at (2, 1.02) {$\scriptscriptstyle 6$};
        \node[black] at (.63, -0.5) {$\scriptscriptstyle 7$};
        \node[black] at (-.9, 2) {$\scriptscriptstyle 8$};
        \node[black] at (.63, 4.0) {$\scriptscriptstyle 9$};
        \node[black] at (2.1, 2.2) {$\scriptscriptstyle 10$};
         \node[black] at (1.1, .8) {$\scriptscriptstyle 11$};
        \node[black] at (3.1, 2) {$\scriptscriptstyle 12$};
    \end{tikzpicture}
         };
    \endxy  
    }="t6";
{\ar^{d_{\ast00000}=m_{12}} "0"+(10,20);"t1"}; 
{\ar^{d_{0\ast0000}=m_{23}} "0";"t2"}; 
{\ar^{d_{00\ast000}=m_{13}} "0";"t3"}; 
{\ar^{d_{000\ast00}=m_{16}} "0";"t4"}; 
{\ar^{d_{0000\ast0}=m_{36}} "0";"t5"}; 
{\ar_{d_{00000\ast}=m_{23}} "0";"t6"}; 
(-48,15)*{\scriptstyle V_1\otimes V_2 \otimes V_3 \otimes V_6};
(34,75)*{\scriptstyle V_1 \otimes V_3 \otimes V_6};
(48,50)*{\scriptstyle V_1 \otimes V_2 \otimes V_6};
(62,25)*{\scriptstyle V_1 \otimes V_2 \otimes V_6};
(62,-05)*{\scriptstyle V_1 \otimes V_2 \otimes V_3};
(50,-33)*{\scriptstyle V_1 \otimes V_2 \otimes V_3};
(35,-57)*{\scriptstyle V_1 \otimes V_2 \otimes V_6};
\endxy 
$
    \caption{An illustration of the chain spaces of the knot $6_3$ in homological degree zero and one.  For this knot, all components of the differential $d\maps C^0 \to C^1$ come from merge or multiplication maps of the various loops.   }
    \label{fig:6-3-Hom=-zero}
\end{figure}

The Laplacian $\Delta_0$ in homological degree zero splits as a direct sum of operators $\Delta_{(0,q)}$, with fixed $q$ degree,  
\[
\Delta_0 =  \bigoplus_{k=0}^{\ell} \Delta_{(0,q)}=\bigoplus_{k=0}^{\ell} \Delta_{(0,2k-\ell)}
\]
where $q=k-n$, $\ell = n+k$.  

\subsubsection{Defining 0-homology knot graphs}

Given a knot $K$ with $N$ crossings and $\ell = \ell(r)$ loops in its all zero resolution  $\ket{r} = \ket{0\dots 0}$, we define a weighted graph $G_{n,k}(K)$ for $n,k\geq 0$ and $\ell = n +k$.   The graph $G_{n,k}(K)$ is defined using the decomposition of the Laplacian in homological degree zero from \cref{eq:0Lsplit}.

\begin{itemize}
    \item The set of vertices $V$ will be the set of sequences $v = v_{i_1} v_{i_2} \dots v_{i_\ell}$ of length $\ell = n +k$ with $n$ terms equal to $X$ and $k$ terms equal to $\1$.  
    The indices of $v$ are chosen to match the circle labels of the all 0-resolution of $K$ as explained in Section~\ref{subsubsec:homzero-structure}.  
    
    \item For each $\mathfrak{m}_{ij}$ appearing in \cref{eq:0Lsplit},  
        \begin{itemize}
            \item we add a loop from $v$ to itself with weight $w_{ii}=\frac{1}{2}$ if $v_i=v_j=\1$,
            \item we add an edge from $v=v_1 \dots v_i \dots v_j \dots v_\ell$ to $v'=v_1 \dots v_j \dots v_i \dots v_\ell$, where the $i$th and $j$th terms have been swapped, if $v_iv_j \in \{\1X, X\1\}$. 
        \end{itemize} 
    \item For each $\mathfrak{d}_{i}$ appearing in \cref{eq:0Lsplit}, 
        \begin{itemize}
            \item we add a loop of weight $\frac{1}{2}$ to vertex $v$ if $v_i=X$,
            \item add a loops of weight $+1$ to vertex $v$ if $v_i=\1$. 
        \end{itemize}
\end{itemize}

Note that because $Q=D+A$ and the adjacency matrix $A$ and $D$ both have the edge weights associated with loops, the matrix $Q$ will always be integer valued with no fractional values. 

\begin{example} \label{example:6-3-graph}
Let $K=6_3$ with crossings and edges labeled as in Example~\ref{example:trefoil-label}.  As shown in Example~\ref{example:trefoil-zero}, the all-zero resolution is a tensor product of $\ell=4$ copies of $V$ carrying labels $V_1 \otimes V_2 \otimes V_3 \otimes V_6$.  The Laplacian $\Delta_0(6_3)$ takes the form
\[ 
\Delta_0(6_3) = \mathfrak{m}_{12}+2\mathfrak{m}_{23} + \mathfrak{m}_{13} + \mathfrak{m}_{16} + \mathfrak{m}_{36}.
\]

Consider $n=k=2$ corresponding to $q$-degree zero.  One can compute the Laplacian in this bidegree takes the form 
\[
\Delta_{(0,0)}(6_3) =  
\begin{pmatrix}
5 & 2 & 1 & 0 & 1 & 0 \\
2 & 6 & 1 & 1 & 0 & 1 \\
1 & 1 & 5 & 0 & 1 & 0 \\
0 & 1 & 0 & 4 & 1 & 1 \\
1 & 0 & 1 & 1 & 5 & 2 \\
0 & 1 & 0 & 1 & 2 & 5 \\
\end{pmatrix}
\]
where we have chosen the ordered basis $\{\1\1XX, \1X\1X, X\1\1X, \1XX\1, X\1X\1, XX\1\1\}$ to match conventions in Bar-Natan's code for computing Khovanov homology.  
The graph $G_{2,2}(6_3)$ then takes the form
\[
\xy
(0,0)*+{\includegraphics[width=5in]{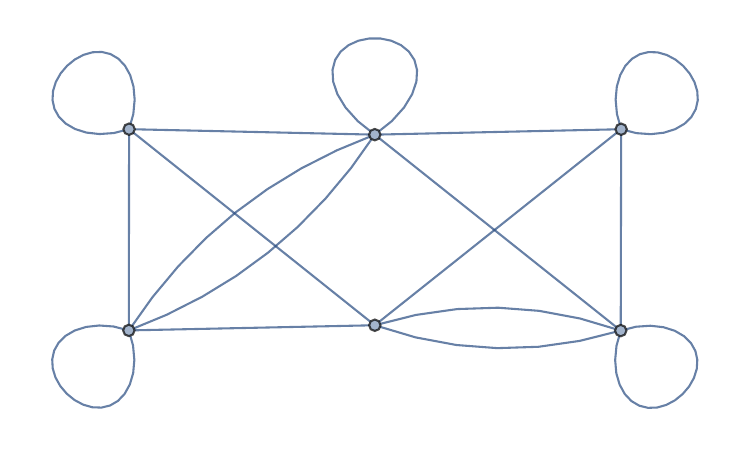}};
(-35,-20)*+{\scriptstyle \1\1XX};
(0,-20)*+{\scriptstyle X\1X\1};
(47,-13)*+{\scriptstyle XX\1\1};
(-35,20)*+{\scriptstyle X\1\1 X};
(10,20)*+{\scriptstyle  \1X\1X};
(47,13)*+{\scriptstyle \1XX\1};
(-55,-29)*+{\scriptstyle \frac{1}{2}};
(56,-25)*+{\scriptstyle \frac{1}{2}};
(-57,24)*+{\scriptstyle  1};
(-9,25)*+{\scriptstyle   \frac{1}{2}};
(40,25)*+{\scriptstyle  \frac{1}{2}};
(-40,-31)*+{\scriptstyle \textcolor[rgb]{0.00,0.07,1.00}{m_{12}}};
(-20,-20)*+{\scriptstyle \textcolor[rgb]{0.00,0.07,1.00}{m_{16}}};
(-6,-7)*+{\scriptstyle \textcolor[rgb]{0.00,0.07,1.00}{m_{36}}};
(-45,-2)*+{\scriptstyle \textcolor[rgb]{0.00,0.07,1.00}{m_{13}}};
(-30,-8)*+{\scriptstyle \textcolor[rgb]{0.00,0.07,1.00}{2m_{23}}};
(20,-17)*+{\scriptstyle \textcolor[rgb]{0.00,0.07,1.00}{2m_{23}}};
(38,-29)*+{\scriptstyle\textcolor[rgb]{0.00,0.07,1.00}{m_{36}}};
(-38,29)*+{\scriptstyle \textcolor[rgb]{0.00,0.07,1.00}{2m_{23}}};
(12,25)*+{\scriptstyle \textcolor[rgb]{0.00,0.07,1.00}{m_{13}}};
(58,17)*+{\scriptstyle \textcolor[rgb]{0.00,0.07,1.00}{m_{16}}};
(-20,20)*+{\scriptstyle \textcolor[rgb]{0.00,0.07,1.00}{m_{12}}};
(45,-2)*+{\scriptstyle \textcolor[rgb]{0.00,0.07,1.00}{m_{13}}};
(35,-7)*+{\scriptstyle \textcolor[rgb]{0.00,0.07,1.00}{m_{16}}};
(6,-7)*+{\scriptstyle \textcolor[rgb]{0.00,0.07,1.00}{m_{12}}};
(25,20)*+{\scriptstyle \textcolor[rgb]{0.00,0.07,1.00}{m_{36}}};
\endxy
\]
where we have labelled each edge by the $m_{ij}$ that produced it.  
One can check that the adjacency and degree matrix take the form
\[
D =  
\begin{pmatrix}
9/2 & 0 & 0 & 0 & 0 & 0 \\
0 & 11/2 & 0 & 0 & 0 & 0\\
0 & 0 & 4 & 0 & 0 & 0 \\
0 & 0 & 0 & 7/2 & 0 & 0 \\
0 & 0 & 0 & 0 & 5 & 0 \\
0 & 0 & 0 & 0 & 0 & 9/2 \\
\end{pmatrix}
, \qquad 
A =  
\begin{pmatrix}
1/2 & 2 & 1 & 0 & 1 & 0 \\
2 & 1/2 & 1 & 1 & 0 & 1 \\
1 & 1 & 1 & 0 & 1 & 0 \\
0 & 1 & 0 &1/2 & 1 & 1 \\
1 & 0 & 1 & 1 & 0 & 2 \\
0 & 1 & 0 & 1 & 2 & 1/2 \\
\end{pmatrix}
\]
so that $\Delta_{(0,0)}(6_3) = Q=D+A$. 
\end{example}

\begin{theorem}
  Let $Q=Q(G_{n,k}(K))$ be the signless Laplace matrix associated to the graph $G_{n,k}(K)$, then $Q=\Delta_{0,k-n}(K)$.    
\end{theorem}

\begin{proof}
The Laplacian $\Delta_{0,k-n}(K)$ is a sum of the contributions of each of the $N$ terms from \cref{eq:0Lsplit}.  Decompose the Laplacian $\Delta_0(K)$ into a sum of $N=p+t$ maps 
\begin{equation}  \label{eq:xxx}
  \Delta_0(K) = \sum_{a=1}^{p} \mathfrak{m}_{i_a, j_a} + \sum_{b=1}^t \mathfrak{d}_{i_b } 
\end{equation}
as in  \cref{eq:0Lsplit}.   By definition,  $G_{n,k}(K)$ has 
an edge from $v=v_1 \dots v_i \dots v_j \dots v_\ell$ to $v'=v_1 \dots v_j \dots v_i \dots v_\ell$ for each $\mathfrak{m}_{ij}$ in \cref{eq:xxx} with $v_iv_j \in \{\1X, X\1\}$.  Hence, the entry of the adjacency matrix $A$ takes the form 
\[
w_{vv'} =  \sum_{a=1}^{p}
        \left(\delta_{v_{i_a},\1} \delta_{v_{j_a},X} \delta_{v'_{i_a},X} \delta_{v'_{j_a},\1} \right)  
        +
        \left(\delta_{v_{i_a},X} \delta_{v_{j_a},\1} \delta_{v'_{i_a},\1} \delta_{v'_{j_a},X} \right)  
\]
for $v \neq v'$, 
and for $v'=v$ we have a weighted loop contributions from $\mathfrak{m}_{ij}$ when $v_iv_j=\1\1$ and always for $\mathfrak{d}_i$ depending on the entry for $v_i$, 
\[
w_{vv} = \sum_{a=1}^{p}\frac{1}{2}\delta_{v_{i_a},\1} \delta_{v_{j_a},\1}   + \sum_{b=1}^t (\delta_{v_{i_b},\1} +\frac{1}{2}\delta_{v_{i_b},X}  ) .
\]
Hence, $Q=A+D$ for $G_{n,k}(K)$ has entries 
\begin{align}
Q_{vv'} & =   w_{vv'} + \delta_{vv'}\sum_{v''} w_{vv''} = 
\delta_{vv'} \left( 2w_{vv} + \sum_{v''\neq v} w_{vv''} \right) + (1-\delta_{vv})w_{vv'}  
\end{align}
so that  
\begin{align*}
    Q_{vv} &= \sum_{a=1}^{p} \delta_{v_{i_a},\1} \delta_{v_{j_a},\1}   + \sum_{b=1}^t (2\delta_{v_{i_b},\1} 
     + \delta_{v_{i_b},X}  ) 
     \\  &\quad+ \sum_{v''\neq v} \sum_{a=1}^{p}
        \left(\delta_{v_{i_a},\1} \delta_{v_{j_a},X} \delta_{v''_{i_a},X} \delta_{v''_{j_a},\1} \right)  
        +
        \left(\delta_{v_{i_a},X} \delta_{v_{j_a},\1} \delta_{v''_{i_a},\1} \delta_{v''_{j_a},X} \right) 
        \\
     &=\sum_{a=1}^{p} \left( \delta_{v_{i_a},\1} \delta_{v_{j_a},\1}   
      + \sum_{v''\neq v}  
        \left(\delta_{v_{i_a},\1} \delta_{v_{j_a},X} \delta_{v''_{i_a},X} \delta_{v''_{j_a},\1} \right)  
        +
        \left(\delta_{v_{i_a},X} \delta_{v_{j_a},\1} \delta_{v''_{i_a},\1} \delta_{v''_{j_a},X} \right)
     \right) \\
     &\quad 
     + \sum_{b=1}^t (2\delta_{v_{i_b},\1} + \delta_{v_{i_b},X}  )
\end{align*}
Examining the terms in the summation over $p$,  each $\mathfrak{m}_{ij}$ in \cref{eq:xxx} contributes one to the diagonal of $Q$ if $v_iv_j=\1\1$, and one to the diagonal for each $v''$ with $v_iv_j=\1X$ and $v''_iv_j''=X\1$ or  $v_iv_j=X\1$ and $v''_iv_j''=\1 X$.  The off-diagonal entries $Q_{vv'}=w_{vv'}$ for $v\neq v'$ similarly have a contribution of one if $v_iv_j=\1X$ and $v'_iv_j'=X\1$ or  $v_iv_j=X\1$ and $v'_iv_j'=\1 X$.   Each term in the summation over $t$ contributes 2 to the diagonal for each $\mathfrak{d}_{i}$ in \cref{eq:xxx} if $v_i=\1$ and 1 if $v_i=X$.  

Comparing the contributions for each of the terms in the $p$ and $t$ summation with the action of $\mathfrak{m}_{ij}$ and $\mathfrak{d}_i$ in \cref{eq:Delta0-bar} completes the proof.  

\end{proof}

\subsection{An upper bound for $\Psi$}

While directly computing $\Psi$ for a given graph $G_{n,k}(K)$ can be challenging in general, we can more easily compute the quantity $\Psi_V$ corresponding to $S=V$.  This is not the minimal value over all subsets in general, but it does provide an upper bound for $\Psi$, and hence, the spectral gap of $\Delta_{0,2k-n}(K)$ by Theorem~\ref{thm:Desai}. 

To compute the quantity $\Psi_V$, observe we have $|V| = \binom{n+k}{k}$ and $|cut(V)|=0$.  The quantity $e_{\min}(V)$ has contributions from the sum of the weights over all weighted loops in $G$, since these all must be removed to form a bipartite graph, and contributions from edges that must be removed to make $G$ bipartite.  Example~\ref{example:6-3-graph} shows that the graphs $G_{n,k}(K)$ are, in general, non-bipartite even after forming the graph $G'_{n,k}(K)$ obtained from $G_{n,k}(K)$ by forgetting loops at vertices.

\subsubsection{Counting loops}

The number of loops at a vertex $v=v_{i_1} \dots v_{i_{\ell}}$ of $G_{n,k}(K)$ is easy to compute from the decomposition \cref{eq:0Lsplit}.   As in   \cref{eq:0Lsplit}, let $p$ be the number of multiplications and let $t$ be the number of comultiplications in the decomposition. 
Each $\mathfrak{m}_{ij}$ gives a loop of weight $1/2$ if $v_i=v_j=\1$ and each $\mathfrak{d}_i$ gives a loop of weight 1 if $v_i=\1$ and weight $1/2$ if $v_i = X$.  

First, we count the weighted sum of loops in $G_{n,k}(K)$ corresponding to contributions from terms in the decomposition of the form $\mathfrak{m}_{ij}$.  
To count the total number of loops in $G_{n,k}(K)$ coming from a specific $\mathfrak{m}_{ij}$, we need to count the number of sequences $v=v_{i_1} \dots v_{i_{\ell}}$ of length $\ell = n+k$ with $k$ terms equal to $\1$ and $n$ terms equal to $X$, where $v_i=v_j=\1$. 
Since we consider all vertices $V$ of $G_{n,k}(K)$, this number is independent of the specific $i$ and $j$ chosen.   We select the $i$ and $j$ entries of our sequences of length $\ell$ to be $\1$, and we have a loop of weight $1/2$ for every choice of the remaining $\ell-2$ entries that contain $k-2$ entries equal to $\1$, i.e. $\binom{\ell-2}{k-2}$.  Hence, the sum of weights   of loops in the in $G_{n,k}(K)$ corresponding to contributions from multiplication terms is 
\[
 \frac{p}{2} \binom{\ell-2}{k-2}.
\]

Next, we count weighted sums of loops corresponding to the $t$ comultiplication terms in the decomposition of $\Delta_0(K)$.  To compute this, we compute how many of the $\binom{n+k}{k}$ sequences have $v_i=\1$ and how many have $v_i=X$.  Again, these numbers are independent of $i$.  There are $\binom{\ell-1}{k-1}$ such sequences with $v_i=\1$ and $\binom{\ell-1}{k}$ with $v_i=X$.  Hence, the weighted sum of loops from the $t$ comultiplications is 
\[
t \binom{\ell-1}{k-1} + \frac{t}{2} \binom{\ell-1}{k}.
\]

Hence, we have proven the following Lemma.
\begin{lemma} \label{lem:loop-count}
The sum of weights of loops in $G_{n,k}(K)$ is given by 
\[
 \frac{p}{2} \binom{\ell-2}{k-2} + t \binom{\ell-1}{k-1} + \frac{t}{2} \binom{\ell-1}{k}
\]
where $p$ is the number of multiplications and $t$ is the number of compultiplications in the decomposition of $\Delta_0(K)$ from \cref{eq:0Lsplit}.
\end{lemma}

\begin{example}
In Example~\ref{example:6-3-graph} of the graph $G_{2,2}(6_3)$ for the knot $6_3$ at $k=2$, we have $\ell=4$, $k=n=2$, $p=6$, and $t=0$.  The weighted sum of the loops appearing in $G_{2,2}(6_3)$ is 
\[
\frac{6}{2} \binom{2}{2} = 3.
\]
\end{example}

\subsubsection{Odd cycles in $G_{n,k}(K)$  } \label{subsec:odd-cycle}
The number of odd cycles in $G_{n,k}(K)$ can also be determined from the decomposition of $\Delta_0(K)$.  Let $G'_{n,k}(K)$ denote the graph $G_{n,k}(K)$ with all loops removed.  Since comultiplication only contributes loops, we need only consider the $p$ multiplications appearing in the decomposition of $\Delta_0(K)$
\[
 \sum_{a=1}^{p} \mathfrak{m}_{i_a, j_a}
\]
The cycle structure of $G'_{n,k}(K)$ is the same as the cycle structure in the set of unordered pairs $\left\{(i_a,j_a) \right\}_{a=1}^p$. 

Let $\Upsilon$ denote the set of all maximal length odd cycles $(i,j_1)(j_1,j_2), \dots, (j_{2k+1}, j)$ formed from $\left\{(i_a,j_a) \right\}_{a=1}^p$ with no repeated pairs.     In the computation of $e_{\min}(V)$, we must remove edges to make $G'_{n,k}(K)$ bipartite.  This graph will have two odd cycles for each element in $\Upsilon$ corresponding to a choice of $\1X$ or $X\1$.  A (possibly weighted) edge must be removed from each of these cycles to make $G'_{n,k}(K)$ bipartite.

Some odd cycles in $\Upsilon$ may intersect, and some edges may have weightings, so selecting a minimal size set of weighted edges to remove from $G'_{n,k}(K)$ to make it bipartite can be nontrivial.  Clearly, $e_{\min}(V)$ is less than or equal to the sum over $u \in \Upsilon$ of the minimal weighting of an edge appearing in $u$.  Specifically, for  $u= (i,j_1)(j_1,j_2), \dots, (j_{2k+1}, j) \in \Upsilon$, define  $\varrho(u)$ and $\Upsilon_K$ as
\[
\varrho(u)  = \min_{(i,j) \in u} w_{ij} , \qquad  \Upsilon_K:= \sum_{u \in \Upsilon} \varrho(u).
\]  

\begin{lemma} \label{lem:eV-bound}
For the graph $G_{n,k}(K)$, 
\begin{align} \label{eq:eminV}
   e_{\min}(V) &=  \frac{p}{2} \binom{\ell-2}{k-2} + 2t \binom{\ell-1}{k-1} + t \binom{\ell-1}{k} + e'_{\min}(V) \\ &\leq  \frac{p}{2} \binom{\ell-2}{k-2} + 2t \binom{\ell-1}{k-1} + t \binom{\ell-1}{k} +2 \Upsilon_K  \nonumber 
\end{align}
where $e'_{\min}(V)$ is the minimum sum of weighted edges that must be removed from $G'_{n,k}(K)$.
\end{lemma}

\begin{proof}
This follows from Lemma~\ref{lem:loop-count} and the discussion above. 
\end{proof}

\begin{proposition} \label{prop:upper}
The smallest eigenvalue $\lambda_{\min}$ of $\Delta_{0,2k-n}(K)$ satisfies 
\[
\lambda_{\min} \leq \frac{\frac{p}{2} \binom{\ell-2}{k-2} + 2t \binom{\ell-1}{k-1} + t \binom{\ell-1}{k} +2 \Upsilon_K }{\binom{n+k}{k}}
\]  
\end{proposition}

\begin{proof}
This is immediate from Theorem~\ref{thm:Desai} and Lemma~\ref{lem:eV-bound}.
\end{proof}

\begin{example}
In Example~\ref{example:6-3-graph}, the homological degree zero Laplacian of the knot $K=6_3$ decomposes as 
\[ 
\Delta_0(6_3) = \mathfrak{m}_{12}+2\mathfrak{m}_{23} + \mathfrak{m}_{13} + \mathfrak{m}_{16} + \mathfrak{m}_{36}.
\]
and 
\[
\Upsilon = \left\{  
(1,2)(2,3)(1,3), \; (1,3) (3,6) (1,6)
\right\}
\]
The edge $(2,3)$ has weight 2, and (1,3) appears in both sequences.  Removing both of the non-loop edges labeled (1,3) from $G'_{2,2}(6_3)$ produces a bipartite graph so that $e_{\min}(V) = 2+ \#\text{loops}$.  However, $\Upsilon_K=2$ so the upper bound from Proposition~\ref{prop:upper} ignores that (1,3) appears in both cycles and implies $e_{\min}(V)  \leq 4 +   \#\text{loops}$ and $\lambda_{\min} \leq \frac{3+2*2}{\binom{4}{2}} = 7/6 \sim 1.1666$.  The actual value is $\lambda_{\min} = 0.78916$.  Computing $\Psi_V= e_{\min}(V)/|V| = 5/6 \sim 0.8333$ gives a better upper bound. 
\end{example}

\begin{remark}
In this section, we have focused on constructing graphs for Khovanov homology at homological degree zero.  A similar technique can be used in the top homological degree using the maps
 \begin{alignat*}{3}
  \delta \delta^{\dagger} \maps V \otimes V&\to V \otimes V   \qquad \qquad    
        &&  mm^{\dagger} \maps V \to V  \\
 \ket{\1\1} &\mapsto 0 \qquad \qquad    
    && \quad\; \ket{\1} \mapsto \ket{\1 }   \\
 \ket{\1X} &\mapsto \ket{\1X} + \ket{X\1} \qquad \qquad    && \quad\; \ket{X} \mapsto 2\ket{X}\\
\ket{X\1}  &\mapsto \ket{\1X} + \ket{X\1}  \\
\ket{XX}  &\mapsto \ket{XX} \\
\end{alignat*}
This essentially just swaps the role of $\1$ and $X$ in our analysis. 
\end{remark}

\subsection{Graphs for twisted unknots}
Here we examine the spectral gaps of the twisted unknots $TU_N$ from Section~\ref{subsec:twisted}. 
\[
TU_N \;\; := \;\; 
\hackcenter{   \begin{tikzpicture}[scale=.6]
      \draw [ very thick]    (1,1) .. controls +(0,.35) and +(0,-.35) .. (2,2);
  \draw [ very thick]    (1,2) .. controls +(0,.35) and +(0,-.35) .. (2,3);
   \draw [ very thick]    (1,4) .. controls +(0,.35) and +(0,-.35) .. (2,5);
  \path [fill=white] (1.35,1) rectangle (1.65,5);
  \draw [ very thick]    (2,1) .. controls +(0,.35) and +(0,-.35) .. (1,2);
  \draw [ very thick]    (2,2) .. controls +(0,.35) and +(0,-.35) .. (1,3);
    \draw [ very thick]    (2,4) .. controls +(0,.35) and +(0,-.35) .. (1,5);
     \draw [ very thick]    (2,5) .. controls +(0,.5) and +(0,.5) .. (1,5);
     \draw [ very thick]    (2,1) .. controls +(0,-.5) and +(0,-.5) .. (1,1);
      \node at (1.5,3.75) {$\vdots$};
    \end{tikzpicture}  } 
    \qquad \qquad 
    \sigma^N \;\; := \;\; 
\hackcenter{   \begin{tikzpicture}[scale=.6]
      \draw [ very thick]    (1,1) .. controls +(0,.35) and +(0,-.35) .. (2,2);
  \draw [ very thick]    (1,2) .. controls +(0,.35) and +(0,-.35) .. (2,3);
   \draw [ very thick]    (1,4) .. controls +(0,.35) and +(0,-.35) .. (2,5);
  \path [fill=white] (1.35,1) rectangle (1.65,5);
  \draw [ very thick]    (2,1) .. controls +(0,.35) and +(0,-.35) .. (1,2);
  \draw [ very thick]    (2,2) .. controls +(0,.35) and +(0,-.35) .. (1,3);
    \draw [ very thick]    (2,4) .. controls +(0,.35) and +(0,-.35) .. (1,5); 
      \node at (1.5,3.75) {$\vdots$};
    \end{tikzpicture}  }
\]
Twisted unknots are much easier to analyze than general knots because the decomposition of $\Delta_0(TU_N)$ from \cref{eq:0Lsplit} takes an especially nice form.  The all 0-resolution of $TU_N$ has $N+1$ circles and the Laplacian decomposes as
\[
\Delta_{0}(TU_N) = \sum_{i=1}^N \mathfrak{m}_{i,i+1}.    
\]

\begin{proposition} \label{prop-TU-bipartite}
Removing all weighted loops from $G_{n,k}(TU_N)$ to form $G'_{n,k}(TU_N)$ gives a bipartite connected graph. 
\end{proposition}

\begin{proof}
Following the analysis in Section~\ref{subsec:odd-cycle}, it is clear that for twisted unknots, $\Upsilon = \emptyset$ since there are no cycles in the set $\{(i,i+1) \}_{i=1}^N$.  Another way to see this is to write any vertex $v$ in $G_{n,k}(TU_N)$ as $w(\1 \dots \1 X \dots X)$ where $w$ is the permutation mapping $\1 \dots \1 X \dots X$ (with $k$ entries $\1$ and $n$ entries X) to $v$ without swapping the set of $\1$s or thee set of $X$'s\footnote{Mathematicians would call these coset representatives of $\mathfrak{S}_k \times \mathfrak{S}_n \subset \mathfrak{S}_{N+1}$.}.  Let $\ell(w)$ denote the length of any reduced expression of $w$ as a product of elementary transpositions.  

To see that $G'_{n,k}(TU_N)$ bipartite, observe that edges  connect vertices $w(\1 \dots \1 X \dots X)$ to vertices $w'(\1 \dots \1 X \dots X)$ only if $\ell(w)=\ell(w')\pm 1$. Moreover, the graph structure is determined by the Bruhat order of the symmetric group.  Specifically, for $w',w$ two coset representatives of $\mathfrak{S}_n \times \mathfrak{S}_k \in \mathfrak{S}_{N+1}$, there will be an edge from the vertex labeled by $w$ to the vertex labeled $w'$ if and only if $w'=s_iw$ is a reduced expression for $w'$, where $s_i$ is a simple transposition.   This also shows that the resulting graph will be connected. 
\end{proof}

\begin{remark}
The proof of Proposition~\ref{prop-TU-bipartite} shows that the graphs $G_{n,k}(TU_N)$  can be interpreted as Schreier coset graphs $Sch(\mathfrak{S}_{N+1},\mathfrak{S}_k \times \mathfrak{S}_n, X)$ for the symmetric group $\mathfrak{S}_{N+1}$ with subgroup $\mathfrak{S}_k \times \mathfrak{S}_n$ and generating set $S=\{s_1, \dots , s_N\}$ the set of simple transpositions.  When $X$ generates $G/H$ the graph $Sch(\mathfrak{S}_{N+1},\mathfrak{S}_k \times \mathfrak{S}_n, X)$ is connected~\cite[Section 6.3]{MR2882891}.  
\end{remark}

Proposition~\ref{prop-TU-bipartite} simplifies the computation of $\Psi$ for twisted unknots as we only need to count weights for loops at vertices in subsets $S\subset V$. 
Given a subset $S\subset V$, let $loop(S)$ denote the weighted sum of loops at the vertices in $S$.  Then $e_{\min}(S) = loop(S)$ and $\Psi_S=(loop(S) + cut(S))/|S|$. 
  Lemma~\ref{lem:loop-count} implies that for $G_{n,k}(TU_N)$ we have 
 \begin{equation} \label{eq:PsiV}
    \Psi_V =  \frac{N}{2}\binom{N-1}{k-2} / \binom{N+1}{k} = \frac{k(k-1)}{2(N+1)}
 \end{equation}

The following is immediate from Theorem~\ref{thm:Desai}. 

\begin{corollary} \label{cor:gapbound}
Let $n+k=N+1$.  The minimal eigenvalue $\lambda_{\min}$ of $\Delta_{(0,2k-n)}(TU_N)$  satisfies
\[
\lambda_{\min} \leq \frac{2k(k-1)}{(N+1)}.
\]
In particular, when $k=0$, and $k=1$, the kernel of the Laplacian is nontrivial.
\end{corollary}

\begin{remark}
Corollary~\ref{cor:gapbound} implies that the bound from \ref{prop:upper} cannot be used to find an example of a knot with an exponentially small spectral gap using $TU_N$ in homological degree zero. 
\end{remark}

Computing $\Psi$ explicitly, and hence a lower bound, even for $TU_N$, can be challenging.  The subset $S$ giving the minimal value of $\Psi_S$ is, in general, not the entire set of vertices $V$. 
 Given subsets $S\subset S' \subset V$, we write $cut(S,S')$ for $cut(S)$ thought of as a subset of the restricted graph $G_{S'}$.  Clearly, 
 \begin{equation}  \label{eq:loopcut-ind}
loop(S') = loop(S)+loop(S'-S), \qquad cut(S') = cut(S) - cut(S,S') + cut(S'-S,V-S),
\end{equation}
where we write $S'-S$ for set subtraction.

\begin{theorem} \label{thm:subsets}
For $K=TU_N$ and $S\subset S' \subset V$, we have $\Psi_{S} < \Psi_{S'}$ if and only if $t > |S'-S|\Psi_{S'}$ where $t=loop(S'/S) - cut(S,S') + cut(S'-S, V-S)$.
\end{theorem}

\begin{proof}
By Proposition~\ref{prop-TU-bipartite} we have
\begin{align}
    \Psi_{S} &= \frac{e_{\min}(S) + cut(S)}{|S|} = \frac{loop(S) + cut(S)}{|S|} \\
    &= \frac{loop(S')-loop(S'-S) + cut(S') + cut(S,S') - cut(S'-S,V-S)}{|S'|-|S'-S|} \nonumber \\
     &= \frac{loop(S')+ cut(S') -t}{|S'|-|S'-S|} \nonumber
\end{align}
so that $\Psi_{S'} < \Psi_S$ if and only if 
\[
\frac{loop(S)+ cut(S) -t}{|S|+|S'-S|} < \frac{loop(S)+ cut(S) }{|S|}
\]
and the result follows. 
\end{proof}

We now focus on $G_{n,2}(TU_N)$.  For $k=2$, as in Section~\ref{subsec:twisted}, it is convenient to indicate the location of the two $\1$'s by a 2-tuple $\{ a_1,a_2\}$ with $1 \leq a_1 < a_2 \leq N+1$.
Any such $v=\{a_1,a_2\}$ is connected to between 4 and 1 edges $\{a_1\pm 1, a_2\pm 1 \}$ in $G_{n,2}(TU_N)$ depending on if $a_1\neq 1$, $a_2\neq N+1$ or $a_2=a_1+1$.   Recall our convention that $\{a,b\}=0$ if $a=b$, $a<1$ or $b>N+1$. 

Observe that a vertex $v$ of $G_{n,2}(TU_N)$ has a loop if and only if the $k=2$ $\1$'s appearing in $v$ are adjacent, so that $v=\{a,a+1\}$ for $1\leq a \leq N$.
Vertices connected to a vertex  $v=\{ a, a+1\}$ with a loop will never have a loop since $\{a-1,a+1\}$ and $\{a, a+2\}$ never have adjacent $\1$'s. 

\begin{lemma} \label{lem:loopLcut}
For any proper subset $T \subset V$ in $G_{n,2}(TU_N)$ we have $loop(T) < cut(T)$.
\end{lemma}

\begin{proof}
Observe that if $v=\{a,a+1\}$ is a vertex with a weight $1/2$ loop it will have at least one edge $(v,v')$ in $G_{n,2}(TU_N)$. The only way this vertex can contribute positively to the quantity $loop(T)-cut(T)$ is if all of the edges $(v,v')$ from $v$ connect to vertices $v'$ that are also in $T$.  Loop edges with $k=2$ always connect to non-loop edges $\{a-1,a+1\}$ or $\{ a,a+2\}$.  Non-loop vertices contribute nothing to $loop(T)$ but always have at least two incident edges.  Hence, the vertex $v$ contributes positively only if every vertex in $V$ is in $T$, contradicting that $T$ was proper. 
\end{proof}

\begin{corollary} \label{cor:bound}
For $k=0,1,2$, $\Psi = \Psi_V$ in the graph $G_{n,k}(TU_N)$, so that
\[
 \frac{\Psi_V^2}{4 d^{\ast}} =  \frac{1}{16}\left(\frac{2k(k-1)}{(N+1)}\right)^2 \leq \lambda_{\rm \min} \leq  \frac{2k(k-1)}{(N+1)}=4\Psi_V.
\]  
where $d^{\ast}$ is the maximal degree of a vertex in $G_{n,k}(TU_N)$.   
\end{corollary}

\begin{proof}
Theorem~\ref{thm:subsets} can be helpful in computing the subset $S$ minimizing $\Psi_S$.  Applying Theorem~\ref{thm:subsets} when $S'=V$ and using \cref{eq:PsiV} implies that $\Psi_S < \Psi_V$ if and only if
\begin{equation} \label{eq:t}
    t = loop(V-S)-cut(S) = loop(V-S)-cut(V-S)  > |V-S| \frac{k(k-1)}{2(N+1)} 
\end{equation}
This implies that if one is trying to find a subset $S$ with a smaller value of $\Psi_S$ than $\Psi_V$, one must select a subset of the vertices of $V$ with the property that the loop count in the vertices we remove is larger (by $|V-S| \frac{k(k-1)}{2(N+1)}$) than the connections from $S$ to the removed vertices in the graph.  For $k=0$ or $k=1$ this is impossible. For $k=2$, taking $T=V/S$, it follows that $t\leq 0$, so that $t$ never satisfies the inequality \cref{eq:t}. 
\end{proof}

\begin{example}[$N=3$ Twisted unnknot]
Let $n=k=2$ with $n+k=N+1=4$.  Then the Laplacian has the form 
\[
\Delta_{0,0}(TU_3) =   \mathfrak{m}_{12} + \mathfrak{m}_{23}+\mathfrak{m}_{34} = 
\left(
\begin{array}{cccccc}
 2 & 1 & 0 & 0 & 0 & 0 \\
 1 & 3 & 1 & 1 & 0 & 0 \\
 0 & 1 & 3 & 0 & 1 & 0 \\
 0 & 1 & 0 & 2 & 1 & 0 \\
 0 & 0 & 1 & 1 & 3 & 1 \\
 0 & 0 & 0 & 0 & 1 & 2 \\
\end{array}
\right)
\]
\[G_{2,2}(TU_3) = 
\xy
(0,0)*+{\includegraphics[width=3.5in]{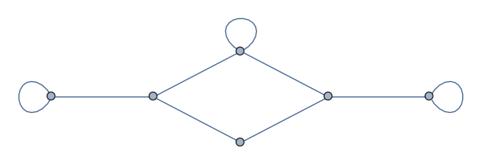} };
(-35,-8)*+{\scriptstyle \1\1XX};
(-18,-8)*+{\scriptstyle  \1X\1X};
(18,-8)*+{\scriptstyle X\1X\1};
(35,-8)*+{\scriptstyle XX\1\1};
(-8,6)*+{\scriptstyle X\1\1 X};
(0,-15)*+{\scriptstyle \1XX\1};
(-37,2)*+{\scriptstyle \frac{1}{2}};
(38,2)*+{\scriptstyle \frac{1}{2}}; 
(6,10)*+{\scriptstyle   \frac{1}{2}};
(-45,-5)*+{\scriptstyle \textcolor[rgb]{0.00,0.07,1.00}{m_{12}}};
(-25,-1)*+{\scriptstyle \textcolor[rgb]{0.00,0.07,1.00}{m_{23}}};
(26,-1)*+{\scriptstyle \textcolor[rgb]{0.00,0.07,1.00}{m_{23}}};
(-7,-11)*+{\scriptstyle \textcolor[rgb]{0.00,0.07,1.00}{m_{34}}};
(-11,2)*+{\scriptstyle \textcolor[rgb]{0.00,0.07,1.00}{m_{12}}};
(9,-11)*+{\scriptstyle \textcolor[rgb]{0.00,0.07,1.00}{m_{12}}};
(10,2)*+{\scriptstyle \textcolor[rgb]{0.00,0.07,1.00}{m_{34}}};
(45,-5)*+{\scriptstyle \textcolor[rgb]{0.00,0.07,1.00}{m_{34}}};
(0,12)*+{\scriptstyle \textcolor[rgb]{0.00,0.07,1.00}{m_{23}}};
\endxy
\]
the minimal eigenvalue is $\lambda_{\min}(Q)=\frac{1}{2} \left(3-\sqrt{5}\right)=0.381966$.  In this case, the minimum for $\Psi$ occurs for $S=V$ where $e_{\rm min}(S)=loop(S)=3/2$, $cut(S)=0$, and $|S|=6$, so that $\Psi=1/4$. Then Theorem~\ref{thm:Desai} implies that 
\[
    \frac{\Psi^2}{4d^{\ast}} =   \frac{(1/4)^2}{12} = \frac{1}{192} =0.005208\leq 0.381966= \lambda_{\min}(Q) \leq 4\Psi = 1. 
\]
\end{example}

\subsection{Graphs for $(2,n)$-torus knots}
Our analysis of twisted unknots $TU_N$ can be extended to $(2,N)$ torus knots.  These are knots that can be obtained as a different closure of the $N$-fold twist 
\[
T_{(2,N)}\;\; := \;\; 
\hackcenter{   \begin{tikzpicture}[scale=.6]
      \draw [ very thick]    (1,1) .. controls +(0,.35) and +(0,-.35) .. (2,2);
  \draw [ very thick]    (1,2) .. controls +(0,.35) and +(0,-.35) .. (2,3);
   \draw [ very thick]    (1,4) .. controls +(0,.35) and +(0,-.35) .. (2,5);
  \path [fill=white] (1.35,1) rectangle (1.65,5);
  \draw [ very thick]    (2,1) .. controls +(0,.35) and +(0,-.35) .. (1,2);
  \draw [ very thick]    (2,2) .. controls +(0,.35) and +(0,-.35) .. (1,3);
    \draw [ very thick]    (2,4) .. controls +(0,.35) and +(0,-.35) .. (1,5);
     \draw [ very thick]    (2,5) .. controls +(0,.5) and +(0,.5) .. (3,5) to (3,1);
     \draw [ very thick]    (2,1) .. controls +(0,-.5) and +(0,-.5) .. (3,1);
      \draw [ very thick]    (1,5) .. controls +(0,1.1) and +(0,1.1) .. (4,5) to (4,1);
     \draw [ very thick]    (1,1) .. controls +(0,-1.1) and +(0,-1.1) .. (4,1);
      \node at (1.5,3.75) {$\vdots$};
    \end{tikzpicture}  } 
    \qquad \qquad 
    \sigma^N \;\; := \;\; 
\hackcenter{   \begin{tikzpicture}[scale=.6]
      \draw [ very thick]    (1,1) .. controls +(0,.35) and +(0,-.35) .. (2,2);
  \draw [ very thick]    (1,2) .. controls +(0,.35) and +(0,-.35) .. (2,3);
   \draw [ very thick]    (1,4) .. controls +(0,.35) and +(0,-.35) .. (2,5);
  \path [fill=white] (1.35,1) rectangle (1.65,5);
  \draw [ very thick]    (2,1) .. controls +(0,.35) and +(0,-.35) .. (1,2);
  \draw [ very thick]    (2,2) .. controls +(0,.35) and +(0,-.35) .. (1,3);
    \draw [ very thick]    (2,4) .. controls +(0,.35) and +(0,-.35) .. (1,5); 
      \node at (1.5,3.75) {$\vdots$};
    \end{tikzpicture}  }
\]
The closure $T_{(2,N)}$ will form a knot when $N$ is odd and a two-strand link when $N$ is even.  For simplicity in our analysis, we consider the knot case when $N$ is odd.  

For $N$ odd, the all zero resolution of $T_{(2,N)}$ will have $N$ circles and $\Delta_0$ has the decomposition
\begin{equation}
\Delta_0(T_{(2,N)})  = \sum_{i=1}^{N-1} \mathfrak{m}_{i i+1} + \mathfrak{m}_{N1}
\end{equation}
In this case, it is clear that $\Upsilon$ contains a single odd cycle $(1,2)(2,3) \dots (N-1,N)(1,N)$, so that two edges must be removed from $G_{n,k}(T_{(2,N)})$ to make it bipartite.    

Lemma~\ref{lem:loop-count} implies that the sum of the weighted loops in $G_{n,k}(T_{(2,N)})$ are given by $\frac{N}{2}\binom{N-2}{k-2}$. Taking $S=V$, we have $e_{\min}(V)=2 +loop(V)$, since we must remove one edge from $G'_{n,k}(T_{(2,N)})$ to make it bipartite.  Thus,
\begin{equation}
    \Psi_V = \frac{2 + loop(V)}{|V|} = \frac{2 + \frac{N}{2}\binom{N-2}{k-2} }{\binom{N}{k} }
\end{equation}
Theorem~\ref{thm:Desai} then implies that
\[
\lambda_{\min} \leq 4 \Psi_V. 
\]
In certain $q$-degrees one could use a variant of Theorem~\ref{thm:subsets} to establish a lower bound.  But again, it will be challenging to explicitly compute the minimal $\Psi_S$ in all $q$-degrees.

\section{Thermalization and uniform sampling} \label{sec:thermalization}
Unlike previous quantum homology algorithms, our approach to computing Khovanov homology outlined in \Cref{subsec:Kh-algorithm} relies on a pre-thermalization procedure that replaces uniform sampling of enhanced states with Gibbs sampling. Here, we explain the motivation and technical framework behind this procedure in more detail. 

\subsection{Limitations of uniform sampling}
In \Cref{sec:general_homology}, we summarized sufficient conditions for an efficient quantum algorithm that computes the Betti numbers $\beta_{ij}$ of Khovanov homology. A standard approach -- employed in almost all prior works on quantum homology algorithms \cite{lloyd2016quantum,quantumAdvantage,persistent,GoogleBettiBerry,AmazonBettiMcArdle,cade2021complexity} -- is to uniformly sample from the elements of $C_{ij}$, that is, efficiently prepare the state \begin{equation}
   \rho_{ij} =  \sum_{\ket{e}\in C_{ij}} \frac{1}{\dim(C_{ij})} \ket{e}\bra{e}.
    \label{eq:sampling1}
\end{equation} Here, $\{\ket{e}\}$ is some orthogonal basis of $C_{ij}$. The Betti numbers are then estimated by applying quantum phase estimation followed by estimating the relative frequency of zero-eigenvalues. This uniform sampling procedure is typically the bottleneck of quantum homology algorithms. Indeed, in the case of clique-homology, this sampling step (preparing $\rho_{ij}$) is already $\sharpP$-hard \cite{schmidhuber2023complexity, valiant1979}. In the context of Khovanov homology, \cref{eq:sampling1} corresponds to uniformly sampling enhanced states. While no analogue complexity-theoretic hardness results for uniformly sampling enhanced states are known to us, it is not clear how to achieve it efficiently.

A second, and perhaps more important, limitation arises from our empirical analysis of the structure of Khovanov homology. 
As we have seen in our numerical exploration of Khovanov homology for knots with up to 11 crossings (cf. \Cref{sec:numerics_gap}), the Betti numbers $\beta_{ij}$ of Khovanov homology can be exponentially small compared to the dimension of the corresponding chain space $C_{ij}$. This implies that uniform sampling of enhanced states is not a viable subroutine for accurately estimating Betti numbers: Even with efficient uniform sampling, the standard quantum homology algorithm takes an exponential time to succeed, as projecting onto the kernel of the Hodge Laplacian requires an exponential number of steps.  
\subsection{Estimating Betti numbers via Gibbs sampling}
To address both these challenges, we introduce the concept of pre-thermalization, where uniform sampling is replaced by sampling from a low-temperature Gibbs state. For a system described by Hamiltonian $H$, the Gibbs state at temperature $T$ is given by the mixed state
\begin{equation}
    \rho_{\beta_T} = \frac{e^{-H/kT}}{Z(T)}
    \label{eq:gibbs_state}
\end{equation}
where $Z(T) = \text{Tr}(e^{-H/kT})$ is the partition function. This Gibbs state has exponentially enhanced overlap with the ground state sector of the physical system whose Hamiltonian corresponds to the Hodge Laplacian. 

Assuming we have successfully prepared a low-temperature Gibbs state for $H = \Delta$, we perform quantum phase estimation and measure whether we found a zero eigenvalue or not. This implements a projection onto the Kernel of $\Delta$. As long as the overlap of the result of the approximate thermalization has inverse-polynomial overlap with the kernel of $\Delta$, which will be approximately the case if the temperature of the thermal state is $O({\rm gap}(\Delta))$\footnote{See also discussion in \Cref{sec:numerics_gap}.0.}, then this projection succeeds with inverse polynomial probability. Repeating this procedure multiple times yields multiple copies of the state $I_0/\beta$, the fully mixed state on the kernel of $\Delta$. Here, the normalization $\beta$ is the kernel dimension (= the Betti number).    

Given multiple copies of $I_0/\beta$, we perform a {\em SWAP} test on pairs of copies of $I_0/\beta$. The {\em SWAP} test succeeds with probability 
\begin{equation}
    p_{\text{succ}} = \frac{1}{2} + \frac{1}{2 \beta}. 
\end{equation}
As long as $\beta$ is polynomially large, then the time it takes to estimate $\beta$ to the desired accuracy grows only polynomially with $m$. As mentioned above, in the original quantum homology algorithm the kernel’s low dimension was an obstruction to estimating the Betti numbers. Now it is an asset: the lower the dimension, the fewer times the SWAP test has to be performed to give an accurate estimate of the Betti numbers.

Using the SWAP test to estimate topological invariants was proposed in reference \cite{Aharonov-Jones-Landau} in the context of approximating the Jones polynomial, and in \cite{scali2024topology} in the context of approximating Betti numbers. However, the authors of \cite{scali2024topology} do not use quantum phase estimation and thus require cooling down the Hodge Laplacian to very small (asymptotically zero) temperature. Our method uses the less stringent thermalization requirement that $kT$ is comparable to the Laplacian gap. 
\subsection{Thermalization of the Hodge Laplacian}
The main challenge of our pre-thermalization approach is, of course, to efficiently prepare (an approximation of) the Gibbs state \cref{eq:gibbs_state}. Recent work on quantum Gibbs sampling shows that such thermalization can often be accomplished efficiently \cite{PoulinWocjan09, TemmeEtAl11, BrandaoKastoryano19, BrandaoThermal21, SommaThermal22, GilyenThermal23a, GilyenThermal23, GilyenThermal24, LinLinThermal24, VerstraeteEtAl09}. The search for efficiently thermalizing Lindblad equations depends crucially on the structure of the Hamiltonian. While we do not expect to be able to construct low-temperature Gibbs states for general $\sharpP$ hard problems, we have not identified any hindrance to thermalizing the Hodge Laplacian for typical (average case) knots in Khovanov homology. It is a key open question to fully characterize
under what conditions efficient thermalization of the Hodge Laplacian is possible.

Some challenges in constructing appropriate Lindblad operators that drive the quantum system associated with Khovanov homology toward thermal equilibrium are as follows. A fundamental theoretical difficulty lies in developing operators that simultaneously preserve the detailed balance condition, thereby ensuring evolution toward the correct thermal state proportional to $e^{-\Delta/kT}$, while maintaining efficiency and complete positivity of the density matrix throughout the evolution. These operators must account for the system Hamiltonian's complete energy level structure and the appropriate balance between excitation and de-excitation processes. These challenges become particularly pronounced for the specific scenarios that appear in Khovanov homology. Systems with degenerate energy levels require special handling to ensure proper thermalization and strong system-bath couplings introduce additional complexity as they invalidate the secular approximation commonly used in weaker coupling regimes.

Note that quantum annealing can also produce approximate low-temperature Gibbs states, which represents another promising approach for pre-thermalization \cite{BravyiGosset16}.  As seen above, the initial state need not be a Gibbs state as long as it has a reasonable (i.e., only polynomially small) overlap with the ground state sector.   Indeed, the method proposed here is a direct extension of the original quantum algorithm for preparing ground states by starting from a state with polynomially small overlap with the ground state and then using quantum phase estimation to project onto the ground state \cite{abrams1999quantum, kitaev1995quantum}.

While pre-thermalization or quantum annealing are likely more efficient in all practical scenarios, for completeness we sketch one possible approach towards achieving uniform sampling. Recall that $\ell(r)$ is the number of loops appearing in the resolution $r$. A prerequisite for uniform sampling is to sample from \begin{equation}
    \rho = \sum_{r \in \mathbb{F}_2^m}\frac{2^{\ell(r)}}{Z}\ket{r}\bra{r},
    \label{sampling3}
\end{equation} where $Z= \sum_r 2^{\ell(r)}$ is a normalizing constant. The hardness of sampling bitstrings $x$ with probability $2^{f(x)}$ is well studied and depends entirely on the function $f(x)$. For example, if $f(x)$ is concave, efficient sampling is possible (log-concave sampling). Our $\ell(r)$ is not concave, but it enjoys a certain locality feature: If we flip one bit of the bitstring $\ket{r}$, then $\ell(r)$ changes by exactly $\pm1$. This is because the boundary operator in Khovanov homology either merges two loops into one or splits one loop into two. Thus $\ell$ is Lipshitz continuous on the hypercube with Lipshitz constant 1. This is promising because it restricts the generality of the sampling problem \cref{eq:sampling1}. This continuity property
might suggest efficient sampling could be possible through algorithms like sequential quantum Metropolis-Hastings, though we do not explore this question further.

\section{Conclusion}
This paper introduces the first quantum algorithm for Khovanov homology, a topological knot invariant that categorifies the Jones polynomial. Our quantum algorithm maps the problem of computing generators of Khovanov homology to the problem of finding ground states of a certain sparse Hamiltonian, which is a task that is particularly well-suited for quantum computers. We generalize the original quantum homology algorithm by incorporating a pre-thermalization procedure, which allows our algorithm to succeed even if the dimension (the Betti number) of the homology group is exponentially small compared to the dimension of the chain space. This pre-thermalization step is necessitated by the fact that the Betti numbers of Khovanov homology are indeed small, which we demonstrate through extensive numerical computations. 

We show that our quantum algorithm is efficient if the Hodge Laplacian of Khovanov homology has a sufficiently large spectral gap and efficiently thermalizes, but we do not prove that these two conditions are satisfied in general. For the spectral gap, we give numerical evidence that it is indeed large enough through exhaustive computation for all knots with up to 10 crossings and some knots with 11 crossings. These numerical results suggest that the gap decays only inverse-polynomially in the crossing number. We support these numerical results with matching analytic lower bounds in homological degree zero. Our analytic proofs are based on a new connection between Khovanov homology and graph theory that might be of independent interest.

We also leverage homological perturbation theory to characterize the behavior of the spectral gap (which is not topologically invariant) under changes in the knot diagram. Based on classical techniques for simplifying the computation of Khovanov homology, we discuss a strategy for increasing the spectral gap. 

To reason about the efficiency of thermalization, we give complexity-theoretic proofs that increasingly accurate approximations to Khovanov homology are $\DQC$-hard, $\BQP$-hard, and $\sharpP$-hard, respectively. Based on standard complexity-theoretic assumptions, we, therefore, do not expect that the thermalization procedure is efficient in general. It is a key open question to characterize under what conditions efficient thermalization of the Hodge Laplacian is possible.

\bigskip\noindent Our work identifies and motivates several other open questions for future research. From a complexity-theoretic perspective, one central question is whether 
there exists an approximation regime for which (the decision variant of) estimating Betti numbers of Khovanov homology is $\BQP$-complete.  Another question is whether our complexity-theoretic hardness results can be strengthened to establish $\QMA_1$-hardness in a stronger regime by developing new perturbative gadgets. 

Another important challenge is to give a complete characterization of the asymptotic scaling of the spectral gap of the Hodge Laplacian. While we provide an analytical characterization in extremal homological degrees, it is interesting to verify whether our approach generalizes to arbitrary degrees. A compelling alternative approach toward bounding the spectral gap is based on spectral sequences, such as in \cite{king2023promise}. Spectral sequences might prove helpful more broadly, for example, in the context of quantum algorithms for complexes admitting a natural filtration or for the homology of fiber bundles. One example in this context is Lee homology~\cite{Lee}, a filtered chain complex closely connected to Khovanov homology, though this situation is backward from the usual role spectral sequences play, as the spectral sequence gives increasingly accurate approximations of Lee homology starting from Khovanov homology.  
Nevertheless, developing quantum algorithms for computing the Lee spectral sequence may still be of interest as this spectral sequence gives rise to Rasmussen's $s$-invariant~\cite{Rasmussen},  which has proven to be a powerful tool for probing the smooth structure of 4-manifolds~\cite{MR2657647,MR4668522,Ren-WIllis}.  

From a physical standpoint, the original quantum algorithm for the Jones polynomial was inspired by its connection to topological quantum field theory via 3D Chern--Simons theory \cite{freedman2002simulation}. This connection extends to the categorified level: Khovanov homology is related to observables in 4D supersymmetric Yang--Mills theory. However, our quantum algorithm is not inspired by this connection but is rather based on a generalization of recent developments in the field of quantum Topological Data Analysis. It is an interesting direction for further research to see whether this physical connection might offer more direct insights into efficient quantum algorithms for Khovanov homology.  

An important note in this context is that the known quantum algorithms for the Jones polynomial give approximations to evaluations of the polynomial only when the quantum parameter $q$ is set equal to a root of unity, rather than at generic $q$.  In contrast, Khovanov homology is a categorification of the Jones polynomial at generic $q$.  Categorification at a root of unity has proven to be mathematically challenging, requiring the use of generalizations of chain complexes where some higher power of the differential is zero, rather than the square of the differential~\cite{khovanov2006hopfological}.  Thus far, no quantum algorithms have been developed for computing homologies of these generalized complexes and it would be interesting to see if the categorification of the Jones polynomial at a root of unity~\cite{Qi_2022} can be more efficiently computed with quantum algorithms.

Our work positions Khovanov homology as a compelling candidate for exponential quantum speedups and a rich subject for further study. The techniques developed in this paper extend beyond Khovanov homology to a broader class of categorical invariants and homological structures in low-dimensional topology. This research also motivates quantum algorithms for several other open problems, such as the unknotting problem, persistent Khovanov homology, and other homological invariants. As quantum computing technology advances, the algorithmic framework presented here may contribute to practical applications in knot theory and related areas of mathematics.

\appendix 

\printbibliography

\end{document}